\documentclass[11pt,reqno]{amsart}

\usepackage[utf8]{inputenc}
\usepackage{amsmath}
\usepackage{amsfonts}
\usepackage{amssymb}

\usepackage{amsmath, amsthm, amssymb}
\usepackage{mathrsfs}
\usepackage{array}
\usepackage{caption}
\evensidemargin \oddsidemargin
\def\bfb{{\mathbf b}}

\def\bfj{{\boldsymbol j}}
\def\bfi{{\boldsymbol i}}
\def\bf1{{\boldsymbol 1}}
\def\bfa{{\boldsymbol a}}
\def\bfb{{\boldsymbol b}}
\def\bfx{{\boldsymbol x}}

\newtheorem{thm}{Theorem}[section]
\newtheorem{cor}{Corollary}[section]

\newtheorem{lem}{Lemma}[section]

\newtheorem{prop}{Proposition}[section]

\numberwithin{equation}{section} \numberwithin{thm}{section}
\numberwithin{lem}{section} \numberwithin{problem}{section}
\numberwithin{cor}{section}

\def\grm{{\mathfrak m}}\def\grM{{\mathfrak M}}\def\grN{{\mathfrak N}}\def\grX{{\mathfrak X}}

\newcommand{\NN}{\mathbb N}
\newcommand{\Nk}{\mathbb{N}_{0}^{k}}

\begin{document}
\title{On Vu's theorem in Waring's problem for thinner sequences}
\author[Javier Pliego]{Javier Pliego}
\address{Universit\`a di Genova, Dipartimento di Matematica. Via Dodecaneso 35, 16146 Genova, Italy}
\email{Javier.Pliego.Garcia@edu.unige.it}

\begin{abstract} Let $k\geq14 $ and $s\geq k(\log k+3.20032)$. Let $\NN_{0}^{k}$ be the set of $k$-th powers of nonnegative integers. Assume that $\psi(x)=o(\log x)$ is an increasing function tending to infinity and satifying some regularity conditions. Then, there exists a subsequence $\mathfrak{X}_{k}=\mathfrak{X}_{k}(s)\subset \NN_{0}^{k}$ for which the number of representations $R_{s}(n;\mathfrak{X}_{k})$ of each $n\in\NN$ as
$$n=x_{1}^{k}+\ldots+x_{s}^{k}\ \ \ \ \ \ \ \  \ \ \ \ \ \ \ x_{i}^{k}\in\mathfrak{X}_{k}$$ satisfies for almost all natural numbers $n$ the asymptotic formula
$$ R_{s}(n;\mathfrak{X}_{k})\sim \mathfrak{S}(n)\psi(n),$$ with $\mathfrak{S}(n)$ being the singular series associated to Waring's problem. If moreover $s\geq k(\log k+4.20032)$ the above conclusion holds for almost all $n\in [X,X+\log X]$ as $X\to\infty$.

Let $T(k)$ be the least natural number for which it is known that all large integers are the sum of $T(k)$ $k$-th powers of natural numbers. We also show for $k\geq 14$ and every $s\geq T(k)$ the existence of a sequence $\mathfrak{X}_{k}'\subset \NN_{0}^{k}$ satisfying $$R_{s}(n;\mathfrak{X}_{k}')\asymp \log n$$ for every sufficiently large $n$. The latter conclusion sharpens a result of Wooley and addresses a question of Vu. 
\end{abstract}
\maketitle
\section{Introduction}
Problems about asymptotic basis of order $s$, which are subsequences $\mathfrak{B}\subset \NN$ for which every sufficiently large natural number is the sum of $s$ elements in $\mathfrak{B}$, constitute a central topic in additive number theory. In particular, finding the least $G(k)$ such that the set $\NN_{0}^{k}$ of $k$-th powers of nonnegative integers is an asymptotic basis of order $G(k)$ plays a prominent role, a recent article of Brudern-Wooley \cite{Bru-Woo} delivering in its simplest formulation the bound
\begin{equation}\label{G(k)}G(k)\leq \lceil k(\log k+4.20032)\rceil.\end{equation}

It seems natural to enquire whether there is a subbasis $\mathfrak{X}_{k}\subset \NN_{0}^{k}$ with the corresponding number of representations of each natural number being small. We then denote for given $\mathfrak{X}_{k}\subset \NN_{0}^{k}$ and $n\in\NN$ by $R_{s}(n;\mathfrak{X}_{k})$ to the number of solutions of $$n=x_{1}^{k}+\ldots+x_{s}^{k},\ \ \ \ \ \ \ \ \ \ \ \ x_{i}^{k}\in\mathfrak{X}_{k}.$$ Answering a query of Nathanson \cite{Nat} about the existence for some $s=s(k)$ of an asymptotic basis $\mathfrak{X}_{k}\subset \NN_{0}^{k}$ of order $s$ satisfying $\lvert \mathfrak{X}_{k}\cap [1,X]\rvert\ll X^{1/s+o(1)}$, Vu \cite{Vu} in fact showed that there are indeed sequences as above satisfying the stronger proviso \begin{equation}\label{RsR}R_{s}(n;\mathfrak{X}_{k})\asymp \log n.\end{equation}
The above result remedied the deficiency of literature when $k\geq 3$ (see \cite{Din,Nat}) and matched what was known for the analogous linear problem \cite{ErTe} with respect to the size of $R_{s}(n;\mathfrak{X}_{k})$. In contrast, the conclusions pertaining to the number of variables were far beyond the bounds for $G(k)$ available, it being implicit in Vu's work that $s\gg k^{4}8^{k}$. This matter was essentially resolved soon after by Wooley \cite{WooVu}, the arguments of that paper thus roughly speaking showing that whenever technology from the Hardy-Littlewood method permits one to derive a bound of the shape $G(k)\leq T(k)$ for some function $T(k)$ then for every \begin{equation}\label{Tks}s\geq T(k)+2\end{equation} there would exist $\mathfrak{X}_{k}=\mathfrak{X}_{k}(s)\subset \NN_{0}^{k}$ satisfying (\ref{RsR}).

Two central questions emerge if one were to go beyond the aforementioned work. Namely, whether the conclusion (\ref{RsR}) may be refined by either an asymptotic formula or an analogous formula with $\log n$ being replaced by some $\psi(n)=o(\log n)$, and whether the two extra variables underlying (\ref{Tks}) could be eliminated. In order to address the second one we define for each $s$ the parameter $\Delta_{2s}=\Delta_{2s}(k)$ to be the unique solution of the equation $\Delta_{2s}e^{\Delta_{2s}/k}=ke^{1-2s/k}.$ We further present for convenience 
\begin{equation}\label{eq67}\tau(k)=\max_{w\in\mathbb{N}}\frac{k-2\Delta_{2w}}{4w^{2}},\ \ \ \ \ 
\ \ \ \ \ G_{0}(k)=\min_{\substack{v\geq 1\\ v\in\NN}}\Big(2v+\frac{\Delta_{2v}}{\tau(k)}\Big).\end{equation}
\begin{thm}\label{thm1.1}
Let $k,s\in \NN$ such that $s\geq \max( \lfloor G_{0}(k)\rfloor +1,4k+1).$ There exists $\mathfrak{X}_{k}=\mathfrak{X}_{k}(s)\subset \Nk$ for which whenever $n$ is a sufficiently large integer in terms of $k$ and $s$ then \begin{equation}\label{XXX}\log n\ll R_{s}(n;\mathfrak{X}_{k})\ll \log n.\end{equation}
In particular, the cardinality of the truncated sequence satisfies \begin{equation}\label{cardi}\lvert \grX_{k}\cap [1,X]\rvert\asymp (X\log X)^{1/s}. \end{equation}
\end{thm}
We remark that despite having only considered even numbers $2v$ in the definition (\ref{eq67}), as opposed to \cite[(6.11)]{Bru-Woo}, the same quantitative conclusions as therein are deduced by following their ideas to provide an upper bound for $G_{0}(k)$ in the upcoming corollaries. The term $4k+1$ may be improved with more work, such a refinement having no impact herein.
\begin{cor}\label{cooor1}
Let $k\in\mathbb{N}$ and $s\geq  k(\log k+4.20032)$. Then there exists a subset $\mathfrak{X}_{k}=\mathfrak{X}_{k}(s)\subset \NN_{0}^{k}$ satisfying (\ref{XXX}) for every sufficiently large integer $n$ and (\ref{cardi}).
\end{cor}
We observe in view of (\ref{G(k)}) that the above result is best possible with respect to $s$. In order to present another ensuing consequence we denote by $\omega$ to the unique real solution with $\omega\geq 1$ of the trascentental equation $\omega-2-1/\omega=\log \omega.$ We then put \begin{equation*}C_{1}=2+\log(\omega^{2}-3-2/\omega),\ \ \ \ \ \ \ \ \ \ \ \ C_{2}=\frac{\omega^{2}+3\omega-2}{\omega^{2}-\omega-2},\end{equation*} and note that $C_{1}=4.200189....$ and $C_{2}=3.015478...$.

\begin{cor}\label{cooor2}
Let $k\in\NN$ and $s\geq  \lceil k(\log k+C_{1})+C_{2}\rceil -1$. Then there exists a subset $\mathfrak{X}_{k}=\mathfrak{X}_{k}(s)\subset \Nk$ satisfying (\ref{XXX}) for every sufficiently large integer $n$, and (\ref{cardi}). Moreover, the same conclusion holds when $14\leq k\leq 20$ and $s\geq H(k)$, where $H(k)$ is defined as
\begin{table}[h!]
  \begin{center}

    \begin{tabular}{l l l c c c r r } 
      
      \hline
      $k$ & $14$ & $15$ & $16$ & $17$ & $18$ & $19$ & $20$\\
      $H(k)$ & $89$  & $97$ & $105$ & $113$ & $121$ & $129$ & $137$ \\
      \hline
    \end{tabular}
  \end{center}
\end{table}
\end{cor}

In view of the preceding corollaries and the conclusions in \cite{Bru-Woo} we note that the restriction on $s$ in the above results matches that in the literature pertaining to Waring's problem whenever $k\geq 14$. We thereby replace (\ref{Tks}) by $s\geq T(k)$ for the smallest currently known $T(k)$ satisfying $G(k)\leq T(k)$ and thus eliminate the two extra variables required hirtherto. We also announce our intention to investigate on another occasion the instance $k\leq 13.$

Despite asymptotic formulae of $R_{s}(n;\mathfrak{X}_{1})$ for the linear case being present in earlier work, no analogous evaluations had previously been obtained for higher powers. We then say that a function $\psi(t)$ is of \emph{uniform growth} when $\psi(t)$ is a positive function of a positive variable $t$ increasing monotonically to infinity. We also say that it is of \emph{uniform growth with exponent $\varepsilon$} if moreover one has $\psi(t) = O(t^{\varepsilon})$ for every $\varepsilon>0$. We may further consider functions with the additional property that there exists another function $\varphi$ of uniform growth such that
\begin{equation}\label{varphi}\psi\big(n/\varphi(n)\big)\sim\psi(n),\end{equation} and write
\begin{equation}\label{xz}\xi_{x}=1-\frac{\psi(x/\varphi(x)\big)}{\psi(x)},\ \ \ \ \ \ \ \ \ \ \ \ \xi(x)=\max_{y\in [x,2x]}\xi_{y}.\end{equation}
We also introduce for $k\geq 2$ and $s\geq \max(5,k+2)$ and $n\in\NN$ the singular series
\begin{equation}\label{SSi}\frak{S}(n)=\sum_{q=1}^{\infty}\sum_{\substack{a=1\\ (a,q)=1}}^{q}\Big(q^{-1}\sum_{r=1}^{q}e(ar^{k}/q)\Big)^{s}e(-an/q).\end{equation} 
\begin{thm}\label{thm1.2}
Let $k,s$ be as in either Theorem \ref{thm1.1}, Corollary \ref{cooor1} or \ref{cooor2}. Let $\psi$ be a function of uniform growth with exponent $\varepsilon$ satisfying (\ref{varphi}) and $\lim_{n\to\infty}\psi(n)/\log n=\infty$. There is a constant $\upsilon>0$ and $\mathfrak{X}_{k}=\mathfrak{X}_{k}(s)\subset \Nk$ such that for sufficiently large $n$ then $$R_{s}(n;\mathfrak{X}_{k})= \frak{S}(n)\psi(n)+O\Big(\psi(n)\Big( \xi(n)+\varphi(n)^{-\upsilon}+(\log n)^{-\upsilon}+\Big(\frac{\log n}{\psi(n)}\Big)^{1/2}\Big)\Big).$$ 

\end{thm}
We recall that \cite[Chapter 4]{Vau} entails $\frak{S}(n)\asymp 1$ whenever $s\geq 4k+1.$ Other authors have considered different formulations of regularity conditions (see \cite{ErD,Taf}). In particular Erd\"os \cite{ErD} imposed $\psi'(t)$ to be continuous, no such strong assumption being required here. Investigating the sharpest possible conclusions though is not the purpose of this memoir.

We shift our attention to the discussion after (\ref{Tks}) and note that in the linear case $k=1$ it was conjectured by Erd\"os and Tur\'an \cite{Erd} that whenever $\mathfrak{X}_{1}$ is an asymptotic basis of order $2$ then $R_{s}(n;\mathfrak{X}_{1})$ cannot be bounded. It is commonly believed that the analogous conclusion for $s> 2$ should also hold, the first author suggesting that even \begin{equation}\label{errdi}\limsup_{n\to\infty}\frac{R_{s}(n;\mathfrak{X}_{1})}{\log n}>0\end{equation} might always occur (see \cite{ERD,Taf}). The preceding discussion thereby lends credibility to the belief that $R_{s}(n;\mathfrak{X}_{k})$ when $k\geq 2$ should satisfy the same properties.

In view of Theorem \ref{thm1.2} it also seems worth deliberating whether (\ref{RsR}) could be replaced by an asymptotic formula. We allude to \cite{ErD}, where it is conjectured that the statement 
$$\lim_{n\to\infty} \frac{R_{2}(n;\mathfrak{X}_{1})}{\log n}=c$$ for any constant $c\neq 0$ is false, it being reasonable to extend such a speculation to the instance $s>2$. In virtue of the formula stemming from the aforementioned theorem and the preceding remark one would predict for any $\psi(n)=O( \log n)$ of uniform growth that 
$$R_{s}(n;\mathfrak{X}_{k})\sim \frak{S}(n)\psi(n)$$ cannot hold. We note that as far as the author is concerned, no previous work in the literature hirtherto had given account of results concerning the above regime.
\begin{thm}\label{thm1.3}
Let $k\geq 14$ and $s\geq  k(\log k+3.20032)$. Let $\psi(n)=O( \log n)$ be of uniform growth satisfying (\ref{varphi}). Let $\upsilon>0$ be a sufficiently small constant and $\delta:\mathbb{R}\rightarrow (0,1)$ with \begin{equation}\label{delll}\delta(x)\geq C_{0}\big(\xi(x)+\varphi(x)^{-\upsilon}+(\log x)^{-\upsilon}+\psi(x)^{-1/2}(\log \psi(x))^{1/2}\big)\end{equation} for some large enough $C_{0}=C_{0}(k,s)>0$. Then there exists $\mathfrak{X}_{k}=\mathfrak{X}_{k}(s)\subset \Nk$ for which 
\begin{equation}\label{RsRs}R_{s}(n;\mathfrak{X}_{k})= \frak{S}(n)\psi(n)+O(\psi(n)\delta(n))\end{equation} holds for all but $O(Ne^{-\delta(N)^{2}\psi(N)})$ integers $n\in [1,N],$ and \begin{equation}\label{cardi1}\lvert \grX_{k}\cap [1,X]\rvert\asymp (X\psi( X))^{1/s}.\end{equation} Moreover, if $\psi(n)=o(\log n)$ one has for some constant $\zeta_{k}>0$ and all but $O(N^{1-\zeta_{k}})$ integers $n\in [1,N]$ the bound \begin{equation}\label{sss1}R_{s}(n;\mathfrak{X}_{k})\ll \frac{\log n}{\log\big(\frac{\log n}{\psi(n)}\big)}.\end{equation}
\end{thm}
We illustrate the discussion by putting $\psi(n)=\sqrt{\log n}$ and applying the above theorem.
\begin{cor}
Let $k,s$ be as in Theorem \ref{thm1.3}. There exists $\mathfrak{X}_{k}=\mathfrak{X}_{k}(s)\subset \Nk$ for which 
\begin{equation*}R_{s}(n;\mathfrak{X}_{k})= \frak{S}(n)\sqrt{\log n}+O\big((\log n)^{1/2-\upsilon}\big)\end{equation*} holds for all but $O(Ne^{-(\log N)^{1/2-\upsilon}})$ integers $n\in [1,N]$ and some fixed $\upsilon>0$, and $$\lvert \grX_{k}\cap [1,X]\rvert\asymp X^{1/s}(\log X)^{1/2s}.$$ Moreover,  one has for all but $O(N^{1-\zeta_{k}})$ integers $n\in [1,N]$ the bound  \begin{equation*}R_{s}(n;\mathfrak{X}_{k})\ll \frac{\log n}{\log\log n}.\end{equation*} 
\end{cor}

We postpone the discussion concerning the conclusions depending on probabilistic arguments but note that the presence of the exceptional set pertaining to (\ref{sss1}) is only due to the limitations of the circle method. We derive though sharper results when $s$ lies on the range in Theorem \ref{thm1.1} and display the strongest conclusions available concerning the validity of the anticipated asymptotic formula when no number theoretic obstructions occur.
\begin{thm}\label{thm1.9}
Let $s$ and $k$ be as in Corollaries \ref{cooor1} or \ref{cooor2}. Let $\psi(n)=O(\log n)$ be a function of uniform growth satisfying (\ref{varphi}), and let $\delta:\mathbb{R}\rightarrow (0,1)$ with (\ref{delll}).  Let $\kappa\geq 1$ be a constant, and $\omega$ be a function of uniform growth such that $\omega(N)=O( e^{\delta(N)^{2}\psi(N)})$. Then, for any collection of sets $(\mathcal{M}_{j}(N))_{j=1}^{N^{\kappa}}$ with $\mathcal{M}_{j}(N)\subset [N,2N]$ for each $N\in\NN$ and $$\frac{(\log N)\omega(N)}{\delta(N)^{2}\psi(N)}\ll\lvert \mathcal{M}_{j}(N)\rvert\ll (\log N)e^{\delta(N)^{2}\psi(N)}$$ there exists $\mathfrak{X}_{k}=\mathfrak{X}_{k}(s)\subset \Nk$ for which for every sufficiently large $N\in \NN$ one has
\begin{equation}\label{RsRs2}R_{s}(n;\mathfrak{X}_{k})= \frak{S}(n)\psi(n)+O(\psi(n)\delta(n))\end{equation} for all but $O(\lvert \mathcal{M}_{j}(N)\rvert \omega(N)^{-1})$ integers $n\in \mathcal{M}_{j}(N)$. 
\end{thm}
We may then derive the following conclusion in the classical setting of short intervals.
\begin{cor}\label{cor92}
Let $s,k,\delta,\psi$ be as in Theorem \ref{thm1.9}. Then there exists $\mathfrak{X}_{k}=\mathfrak{X}_{k}(s)\subset \Nk$ for which the asymptotic formula (\ref{RsRs2}) holds for all but $O(\log X)$ natural numbers $n\in [X,X+(\log X)e^{\delta(X)^{2}\psi(X)}]$ as $X\to\infty$. Moreover, for every function $\omega(x)$ of uniform growth satisfying $\frac{\log n}{\psi(n)}=o(\omega(n))$ there exists $\mathfrak{X}_{k}=\mathfrak{X}_{k}(s)\subset \Nk$ such that \begin{equation}\label{halamad}R_{s}(n;\mathfrak{X}_{k})\sim\frak{S}(n)\psi(n)\end{equation} for all but $o(\omega(X))$ integers $n\in [X,X+\omega(X)]$ as $X\to\infty$.
\end{cor}
It thereby stems that whenever $\psi(n)=O(\log n)$ is of uniform growth, satisfies (\ref{varphi}), and $k,s$ are as in Corollary \ref{cor92} then (\ref{halamad}) holds for almost all $n\in [X,X+\log X]$ as $X\to\infty$. If moreover $\psi(n)\asymp \log n$ then (\ref{halamad}) holds for any $\omega(x)$ of uniform growth and almost all $n\in [X,X+\omega(X)]$ as $X\to\infty$. We note though that an astute modification of the arguments in \cite{Gog,Lan} would have delivered a similar conclusion concerning the bound $O(Ne^{-\delta(N)^{2}\psi(N)})$ in Theorem \ref{thm1.3} but no further sharpenings as the ones in Theorem \ref{thm1.9} and Corollary \ref{cor92}. Moreover (\ref{sss1}), which in particular entails $R_{s}(n;\mathfrak{X}_{k})=o(\log n)$, holds for sufficiently large $n$, the property failing in view of (\ref{errdi}) being that of constituting an asymptotic basis. Note that almost all results in problems involving the circle method typically show evidence for the veracity of the statements, in contrast to what occurs herein. 

We remark that no estimates concerning (\ref{sss1}) had been obtained hirtherto for $s\geq 3$ (see \cite{Erdos} for a related result when $s=2$). Previous approaches to obtain upper bounds for similar random variables would typically have had their genesis on the Sunflower lemma \cite{ErTe} or on concentration inequalities of Vu \cite{Vu2} or Kim and Vu \cite{Kim}. We employ instead that of Janson and Rucinski \cite{Jan2}, the others delivering conclusions not sufficient for our purposes. 

The starting point of our proof of the above theorems is inspired by that in Vu \cite{Vu} with respect to the probabilistic ideas and in \cite{WooVu} with respect to the circle method input. We depart from the latter in the choice of smooth numbers latent in the analysis, the random sequences considered therein comprising $k$-th powers of integers lying in a subset of the smooth numbers presented via a suitable partition. In contrast, the smoothness condition herein permits one with the aid of new technology \cite{Bru-Woo} to eventually deduce an asymptotic evaluation and enables one to derive pointwise estimates over minor arcs of the same strength than those available in the literature for conventional smooth Weyl sums.

A handful of additional difficulties arise in the context of Theorem \ref{thm1.3} as soon as the number of variables $s$ distants from the thresholds presented in both Theorems \ref{thm1.1} and \ref{thm1.2}. Rather than merely showing that almost all natural numbers may be written as a sum of $s$ positive $k$-th powers, one should moreover prove for almost every $n$ that for every $1\leq d\leq s-1$ and fixed $(y_{1},\ldots,y_{d})\in [1,n^{1/k}]^{d}$, the number of solutions of \begin{equation}\label{eve}n-y_{1}^{k}-\ldots-y_{d}^{k}=x_{d+1}^{k}+\ldots+x_{s}^{k},\ \ \ \ x_{i}\in\NN\end{equation} counted with weights $(x_{d+1}\cdots x_{s})^{-1+k/s}$ with the variables satisfying some smoothing condition is $O(n^{-\tau})$ for some $\tau>0$. The preceding proviso is required for the application of probabilistic concentration inequalities. Such a big collection of additional counting problems drastically impairs the ensuing conclusions with respect to the range of $s$, robust estimates for exceptional sets of natural numbers not represented as sums of positive $k$-th powers being particularly useful in order to bound the above quantities. 

Several complications are encountered in the preceding endeavour. First, as in \cite{WooVu}, one is forced to make a dissection in order to consider sums running over smooth numbers of similar size. When expressing (\ref{eve}) via orthogonality as an integral of smooth Weyl sums and performing the above decompositions, one is left to analyse integrals of the shape
\begin{equation}\label{evel}\int_{0}^{1}\prod_{l=1}^{s-d}g_{s}(\alpha, P_{l},R)e(-\alpha m)d\alpha.\end{equation} Here the parameters $P_{l}$ satisfy $1\leq P_{l}\leq P$ with $P=m^{1/k}$ and
$$g_{s}(\alpha,P,R)=\sum_{x\in \mathcal{A}(P,R)\setminus \mathcal{A}(P/2,R)}x^{-1+k/s}e(\alpha x^{k}),$$ the set of smooth numbers $\mathcal{A}(P,R)$ being defined in (\ref{mucha}). However, in order to apply Bessel's inequality effectively to the end of deriving the strongest bounds possible, it transpires that for each collection $(P_{l})_{l\leq s-d}$ the choice of major and minor arcs should be uniform. This creates recalcitrant situations whenever the sizes of $P_{l_{1}}$ and $P_{l_{2}}$ for some $l_{1}\neq l_{2}$ are significantly different since $g_{s}(\alpha, P_{l_{2}},R)$ over the major arcs cognate to $P_{l_{1}}$ may no longer exhibit suitable major arc behaviour and viceversa.

Moreover, when $s$ is as in Theorems \ref{thm1.3} it stems from \cite[Theorem 5.2]{Bru-Woo} that, upon defining
\begin{equation}\label{prato}f(\alpha,P,R)=\sum_{x\in \mathcal{A}(P,R)}e(\alpha x^{k})\end{equation} then for small fixed $c>0$ and $k$ large one has on the set of extreme minor arcs $\grN(cP^{k/2},P)$ defined in (\ref{Tru}) the estimate
\begin{equation*}\int_{\grN(cP^{k/2},P)}\lvert f(\alpha,P,R)\rvert^{2s}d\alpha\ll P^{2s-k-\log k/10}.\end{equation*}
However, if one were to save a factor of $P^{\log k/10}$ over the trivial bound for the corresponding exceptional set, estimates for the $s$-th moment over truncated minor arcs of the shape $P^{s-k-\delta}$ for some $\delta>0$ should be obtained. This though may no longer be possible with the current knowledge when the height $Q$ associated to such minor arcs is of intermediate size. It is also worth noting in view of (\ref{evel}) that $P_{l}$ may be considerably smaller than $P$, it thereby no longer being possible saving such a factor by means of the above procedure. 

The outlined difficulties are partially surmounted with the aid of the new technology introduced in \cite{Bru-Woo}, which permits one to enlarge the range of $Q$ for which mean values restricted to major arcs of such heights may be appropiately estimated. We then make a division according to the sizes $P_{l}$ and reduce the problem to that of estimating integrals (\ref{evel}) with $(P_{l})_{l\leq s-d}$ being suitably close in size so that major arcs corresponding to some $P_{l_{0}}$ are contained in a moderately enlarged set of major arcs cognate to $P_{l}$ for each $l$.

The exposition is structured as follows. We start in Sections \ref{sec2} and \ref{sec3} by adapting the new machinery in \cite{Bru-Woo} to the setting of weighted smooth Weyl sums over dyadic intervals. Sections \ref{sec4} and \ref{sec5} are devoted to get major arc estimates by combining Abel summation with classical bounds. We perform a prunning process in Section \ref{sec6} which culminates in Section \ref{sec7} with the obtention of an asymptotic formula, an analogous one for almost all numbers being derived in Section \ref{sec8}. Section \ref{sec9} concludes the circle method part with an intrincate analysis to derive estimates for exceptional sets cognate to (\ref{eve}). We prepare the ground for the application of the probabilistic method in Section \ref{sec11}, the proofs of Theorems \ref{thm1.1} and \ref{thm1.2} being completed in Section \ref{sec12}. Sections \ref{sec13} and \ref{sec14} contain the bulk of the probabilistic part of Theorems \ref{thm1.3} and \ref{thm1.9}, the latter being devoted to provide upper bounds for the representation function and the proofs of the theorems being delivered in Section \ref{sec15}. 

\emph{Acknowledgements:}
The author is partially funded by the Curiosity Driven grant ``Value distribution of quantum modular forms'' of the Universita degli Studi di Genova.
\section{Mean value estimates for weighted Weyl sums}\label{sec2}
We shall devote the present section to prepare the ground by adapting the machinery developed in \cite{Bru-Woo} to the context of weighted smooth Weyl sums relevant to our current needs. To such an end we fix $k\geq 2$ and introduce for $R,P\geq 1$ the set of smooth numbers
\begin{equation}\label{mucha}\mathcal{A}(P,R)=\Big\{n\in [1,P]\cap \mathbb{Z}:\ \ p|n\Longrightarrow p\leq R\Big\}\end{equation} and $\tilde{\mathcal{A}}(P,R)=\mathcal{A}(P,R)\setminus \mathcal{A}(P/2,R)$. Recall (\ref{prato}) and define for each $t>0$ the mean value
\begin{equation}\label{eq4}U_{t}(P,R)=\int_{0}^{1}\lvert f(\alpha,P,R)\rvert^{t}d\alpha.\end{equation} We say that $\Delta_{t}\geq 0$ is an admissible exponent if for $\varepsilon>0$ and sufficiently small $\eta>0$ in terms of $k,s,\varepsilon$ then whenever $1\leq R\leq P^{\eta}$ and $P$ is sufficiently large one has $U_{t}(P,R)\ll P^{t-k+\Delta_{t}+\varepsilon},$ the implicit constant depending on $k,s,\varepsilon,\eta$. When $s\geq 2$ we also consider
\begin{equation}\label{fuste}g_{s}(\alpha,P,R)=\sum_{x\in \tilde{\mathcal{A}}(P,R)}x^{-1+k/s}e(\alpha x^{k}),\ \ \ \ \ \ \ \ f_{s}(\alpha,P,R)=\sum_{x\in \mathcal{A}(P,R)}x^{-1+k/s}e(\alpha x^{k}).\end{equation}
For future purposes it may be convenient defining for $q\in\mathbb{N}$ the set
$$\mathscr{C}_{q}(P,R)=\Big\{n\in\mathcal{A}(P,R):\ \ p|n\Longrightarrow p| q\Big\}$$ and denoting $\widetilde{\mathscr{C}}_{q}(P,R)=\mathscr{C}_{q}(P,R)\setminus \mathscr{C}_{q}(P/2,R)$. We also introduce for $M\geq 1$ the sums
$$g^{*}_{s,q}(\alpha,P,M,R)=\sum_{\substack{v\in\mathcal{A}(P,R)\\ v>M, \ \ (v,q)=1}}v^{-1+k/s}\sum_{u\in\widetilde{\mathscr{C}}_{q}(P/v,R)}u^{-1+k/s}e(\alpha(uv)^{k})$$ and 
$$g_{s,q}^{\dag}(\alpha,P,M,R)=\sum_{\substack{v\in\mathcal{A}(M,R)\\ (v,q)=1}}v^{-1+k/s}\sum_{u\in\widetilde{\mathscr{C}}_{q}(P/v,R)}u^{-1+k/s}e(\alpha(uv)^{k}).$$
In view of the preceding definitions, it transpires that for every $q\in \mathbb{N}$ then 
\begin{equation}\label{cafe}g_{s}(\alpha,P,R)=g^{*}_{s,q}(\alpha,P,M,R)+g_{s,q}^{\dag}(\alpha,P,M,R).\end{equation}

Equipped with the above formula we shall compute next mean values over a suitable set of major arcs, it being desirable introducing beforehand for any prime number $\pi$ the sum
$$g_{s,q,\pi}^{*}(\alpha,P,m,R)=\sum_{\substack{w\in\mathcal{A}(P/m,\pi)\\ (w,q)=1}}w^{-1+k/s}\sum_{u\in \widetilde{\mathscr{C}}_{q}(P/mw,R)}u^{-1+k/s}e(\alpha(uw)^{k}).$$ Moreover, we also take for $M\geq 1$ the set
$$\mathscr{B}(M,\pi,R)=\Big\{v\in \mathscr{A}(M\pi,R):\ v>M,\ \pi|v \ \text{and $\pi'|v\Longrightarrow \pi'\geq \pi$}\Big\},$$ and make a Hardy-Littlewood dissection of the unit interval as follows. When $1\leq Q\leq P^{k/2}$ and $q\in\NN$ satisfies $1\leq q\leq Q$ we define $\grM_{q}(Q,P)$ to be the union of the sets
$$\grM_{a,q}(Q,P)=\Big\{\alpha\in [0,1):\ \ \ \lvert q\alpha-a\rvert\leq QP^{-k}\Big\}$$ for $a\in\mathbb{Z}$ with $0\leq a\leq q$ and $(a,q)=1$ and write
$$\grM(Q,P)=\bigcup_{q=1}^{Q}\grM_{q}(Q,P).$$ It may be pertinent to consider for future use the dyadically truncated major arcs
\begin{equation}\label{Tru}\grN(Q,P)=\grM(Q,P)\setminus \grM(Q/2,P)\end{equation} and for $q\leq Q$ the associated collection $\grN_{q}(Q,P)=\grM_{q}(Q,P)\setminus \grM_{q}(Q/2,P).$ We shall often write $\grM_{a,q}(Q)=\grM_{a,q}(Q,P)$. We also consider for $\mathfrak{B}$ being $\mathfrak{M}$ or $\mathfrak{N}$ and $t> 1$ the sum
\begin{equation}\label{eq1}I_{q,t}(M,\mathfrak{B})=\sum_{\pi\leq R}\sum_{\substack{m\in\mathscr{B}(M,\pi,R)\\ (m,q)=1}}m^{-1+k/s}\int_{\mathfrak{B}_{q}(Q,P)}\lvert g_{s,q,\pi}^{*}(\alpha m^{k}, P,m,R)\rvert^{t}d\alpha.\end{equation}
\begin{lem}\label{lem2}
Let $s,t\geq 2$ and $1\leq Q\leq P^{k/2}$. Let $q\in\mathbb{N}$ with $1\leq q\leq Q$ and $M\geq R$. Then
$$\int_{\frak{B}_{q}(Q,P)}\lvert g_{s}(\alpha,P,R)\rvert^{t}d\alpha \ll (R^{1+k/s}M^{k/s})^{t-1}I_{q,t}(M,\frak{B})+QM^{t}P^{\varepsilon-k-t(1-k/s)}.$$
\end{lem}
\begin{proof}
We draw the reader's attention to (\ref{cafe}) and start by noting that whenever $q\leq Q$ then \cite[Lemma 2.1]{Woo} yields $\lvert \widetilde{\mathscr{C}}_{q}(P,R)\rvert\ll P^{\varepsilon},$ and hence one trivially has
$$g_{s,q}^{\dag}(\alpha,P,M,R)\ll P^{-1+k/s}\sum_{\substack{v\in\mathcal{A}(M,R)\\ (v,q)=1}}\lvert \widetilde{\mathscr{C}}_{q}(P/v,R)\rvert\ll P^{-1+k/s+\varepsilon}M. $$
It also seems worth observing that the argument of \cite[Lemma 3.3]{Bru-Woo} entails
$$g_{s,q}^{*}(\alpha,P,M,R)=\sum_{\pi\leq R}\sum_{\substack{m\in\mathscr{B}(M,\pi,R)\\ (m,q)=1}}m^{-1+k/s}g_{s,q,\pi}^{*}(\alpha m^{k},P,m,R).$$
Consequently, an application of Holder's inequality would deliver the estimate
$$\lvert g_{s,q}^{*}(\alpha,P,M,R)\rvert^{t}\ll (R^{1+k/s}M^{k/s})^{t-1}\Big(\sum_{\pi\leq R}\sum_{\substack{m\in \mathscr{B}(M,\pi,R)\\ (m,q)=1}}m^{-1+k/s}\lvert g_{s,q,\pi}^{*}(\alpha m^{k},P,m,R)\rvert^{t}\Big).$$ Then by observing that $\lvert \frak{B}_{q}(Q,P)\rvert\ll QP^{-k}$ the lemma follows upon recalling (\ref{eq1}) by combining the preceding equations with (\ref{cafe}).  
\end{proof}
In what follows we prepare the ground to obtain a mean value estimate of the shape of \cite[Theorem 4.2]{Bru-Woo}. We thus fix $Q$ satisfying $1\leq Q\leq \frac{1}{2}P^{k/2}R^{-k}$ and $m\in \mathscr{B}(M,\pi,R)$, and
\begin{equation}\label{M}M=P(2Q)^{-2/k}R^{-1}.\end{equation}
The preceding assumptions assure that $R\leq M$ and that whenever $\pi\leq R$ and $m\in \mathscr{B}(M,\pi,R)$ then $m\leq M\pi\leq P(2Q)^{-2/k}$, and thus $Q\leq \frac{1}{2}(P/m)^{k/2},$ which entails that the arcs $\mathfrak{M}_{a,q}(Q,P/m)$ are disjoint. We recall (\ref{eq1}) and apply \cite[Lemma 2.3]{Bru-Woo} to obtain
\begin{equation}\label{eq3}I_{q,t}(M,\mathfrak{B})=\sum_{\pi\leq R}\sum_{\substack{m\in\mathscr{B}(M,\pi,R)\\ (m,q)=1}}m^{-k-1+k/s}\int_{\mathfrak{B}_{q}(Q,P/m)}\lvert g_{s,q,\pi}^{*}(\alpha , P,m,R)\rvert^{t}d\alpha.\end{equation}
In order to make further progress we introduce for $q\in\mathbb{N}$ and a prime $\pi\leq R$ the subset 
$$\mathscr{C}_{q,\pi}(P,R)=\Big\{n\in\mathscr{C}_{q}(P,R):\ \ p|n\Longrightarrow p>\pi\Big\}$$ and $\widetilde{\mathscr{C}}_{q,\pi}(P,R)=\mathscr{C}_{q,\pi}(P,R)\setminus \mathscr{C}_{q,\pi}(P/2,R)$. We then observe that the same argument employed in \cite[Lemma 4.1]{Bru-Woo} permits one to deduce 
\begin{align*}g_{s,q,\pi}^{*}(\alpha,P,m,R)&=\sum_{z\in \mathscr{C}_{q,\pi}(P/m,R)}z^{-1+k/s}\sum_{x\in \tilde{\mathcal{A}}(P/mz,\pi)}x^{-1+k/s}e(\alpha(xz)^{k})
\\
&=\sum_{0\leq j\leq \frac{\log (P/m)}{\log 2}}\ \ \sum_{z\in \widetilde{\mathscr{C}}_{q,\pi}(2^{-j}P/m,R)}z^{-1+k/s}g_{s}(\alpha z^{k},P/mz,\pi),
\end{align*}
where we made a dyadic dissection. Having been furnished with the previous identity we note when $t\geq 2$ that the bound $\lvert \widetilde{\mathscr{C}}_{q}(P,R)\rvert\ll P^{\varepsilon}$ combined with Holder's inequality delivers
\begin{align*}\lvert g_{s,q,\pi}^{*}(\alpha,P,m,R)\rvert^{t}&\ll P^{\varepsilon}\sum_{0\leq j\leq \frac{\log (P/m)}{\log 2}}(P2^{-j}/m)^{(-1+\frac{k}{s})t}\sum_{z\in \widetilde{\mathscr{C}}_{q,\pi}(2^{-j}P/m,R)}\big\lvert g_{s}(\alpha z^{k},P/mz,\pi)\big\rvert^{t}.
\end{align*}
Then, upon denoting
\begin{equation}\label{eq7}V_{t}(\pi,m,z,\mathfrak{B})=\int_{\frak{B}(Q,P/m)}\lvert g_{s}(\alpha z^{k},P/mz,\pi)\rvert^{t}d\alpha\end{equation} we deduce via the previous equation in conjunction with (\ref{eq3}) that 
\begin{align}\label{eq6}\sum_{1\leq q\leq Q}&I_{q,t}(M,\frak{B})\leq \sum_{\pi\leq R}\sum_{\substack{m\in\mathscr{B}(M,\pi,R)}}m^{-k-1+\frac{k}{s}}\sum_{1\leq q\leq Q}\int_{\frak{B}_{q}(Q,\frac{P}{m})}\lvert g_{s,q,\pi}^{*}(\alpha,P,m,R)\rvert^{t}d\alpha
\\ 
&\ll P^{\varepsilon}\max_{P_{0}\leq P} \Big(P_{0}^{(-1+k/s)t}\sum_{\pi\leq R}\sum_{\substack{m\in\mathscr{B}(M,\pi,R)}}m^{-k-1+k/s}\sum_{z\in \widetilde{\mathcal{A}}(P_{0},R)}V_{t}(\pi,m,z,\mathfrak{B})\Big),\nonumber
\end{align}
where we used the fact that the arcs $\frak{M}_{q}(Q,P/m)$ are disjoint as observed right before (\ref{eq3}). In order to bound the above inner sum it is desirable to consider for $Y\leq P$ the mean value
$$V_{t,s}(Y,R)=\int_{0}^{1}\lvert g_{s}(\alpha,Y,R)\rvert^{t}d\alpha,$$
and to furnish ourselves with the next lemma that will be employed throughout the paper. 
\begin{lem}\label{lem1}Let $t=2w$ with $w\in\mathbb{N}$. Let $\Delta_{t}$ be an admissible exponent.  There is some $\eta$ depending on $\varepsilon,k,s$ with the property that whenever $P$ is sufficiently large and $1\leq R\leq P^{\eta}$ then, uniformly in $Y\leq P$ one has  
$$V_{t,s}(Y,R)\ll P^{\varepsilon}Y^{tk/s-k+\Delta_{t}}.$$\end{lem}
\begin{proof}
We first note that by orthogonality, $V_{t,s}(Y,R)$ equals the number of solutions of $$x_{1}^{k}+\ldots+x_{w}^{k}=x_{w+1}^{k}+\ldots+x_{2w}^{k}$$ counted with weights $$(x_{1}\ldots x_{2w})^{-1+k/s}\asymp Y^{tk/s-t}.$$ It then transpires upon recalling (\ref{eq4}) that $V_{t,s}(Y,R)\ll Y^{tk/s-t}U_{t}(Y,R)$. The lemma follows by combining the preceding equation with \cite[Lemma 2.1]{Bru-Woo}.
\end{proof}
We are now prepared to present the following key proposition. We henceforth establish the convention that unless mentioned otherwise, whenever a statement involves the letter $R$, then for any $\varepsilon> 0$ there is $\eta>0$ such that the statement holds uniformly for $1\leq R\leq P^{\eta}$. 

\begin{prop}\label{prop2.1}
Let $s\geq 2$ and $t= 2w$ for some $w\in\mathbb{N}$ satisfying $\omega\geq 1$. Let $\Delta_{t}$ be an admissible exponent. Then whenever $1\leq Q\leq P^{k/2}$ one has the bound
$$\int_{\frak{M}(Q,P)}\lvert g_{s}(\alpha,P,R)\rvert^{t}d\alpha\ll P^{tk/s-k+\varepsilon}Q^{2\Delta_{t}/k}.$$
\end{prop}
\begin{proof}
When $\frac{1}{2}P^{k/2}R^{-k}\leq Q\leq P^{k/2}$, it transpires that the previous lemma yields
$$\int_{\frak{M}(Q,P)}\lvert g_{s}(\alpha,P,R)\rvert^{t}d\alpha\ll V_{t,s}(P,R)\ll P^{tk/s-k+\Delta_{t}+\varepsilon}\ll P^{tk/s-k+\varepsilon}Q^{2\Delta_{t}/k},$$as desired. If instead $Q<\frac{1}{2}P^{k/2}R^{-k}$ we observe by (\ref{eq6}) and a change of variables that
\begin{align*}\sum_{ q\leq Q}I_{q,t}(M,\frak{M})&\ll P^{\varepsilon}\max_{P_{0}\leq P}\Big(P_{0}^{(-1+k/s)t}\sum_{\pi\leq R}\sum_{\substack{m\in\mathscr{B}(M,\pi,R)}}m^{-k-1+k/s}\sum_{z\in \tilde{\mathcal{A}}(P_{0},R)}V_{t,s}(P/mz,R)\Big)
\\
&\ll P^{\varepsilon}\max_{P_{0}\leq P}\Big(P_{0}^{(-1+k/s)t+1}\sum_{\pi\leq R}\sum_{\substack{m\in\mathscr{B}(M,\pi,R)}}m^{-k-1+k/s}(P/mP_{0})^{tk/s-k+\Delta_{t}}\Big),
\end{align*}
where we used Lemma \ref{lem1}. We then note that $t\geq 2$ and $\Delta_{t}\geq k-t/2$, the latter being a consequence of the presence of diagonal solutions in the equation underlying (\ref{eq4}), whence
\begin{align*}\sum_{ q\leq Q}I_{q,t}(M,\frak{M})&\ll \max_{P_{0}\leq P}\big(P_{0}^{k-t+1-\Delta_{t}}P^{\frac{tk}{s}-k+\Delta_{t}+\varepsilon}M^{-\Delta_{t}-(t-1)\frac{k}{s}}\big)\ll P^{\frac{tk}{s}-k+\Delta_{t}+\varepsilon}M^{-\Delta_{t}-(t-1)\frac{k}{s}}.
\end{align*}
Therefore, Lemma \ref{lem2} in conjunction with (\ref{M}) and the preceding equation delivers
\begin{align*}\int_{\frak{M}(Q,P)}\lvert g_{s}(\alpha,P,R)\rvert^{t}d\alpha&\ll \sum_{1\leq q\leq Q}\int_{\frak{M}_{q}(Q,P)}\lvert g_{s}(\alpha,P,R)\rvert^{t}d\alpha\ll P^{\frac{tk}{s}-k+\varepsilon}Q^{\frac{2\Delta_{t}}{k}}+P^{\frac{tk}{s}-k+\varepsilon}Q^{2-\frac{2t}{k}}.
\end{align*}  The lemma then follows by inserting the aforementioned inequality pertaining $\Delta_{t}$. 
\end{proof}

\section{Mean values restricted to minor arcs}\label{sec3}
We shall next obtain a pointwise bound for the weighted smooth exponential sum at hand which shall be more effective on extreme sets of minor arcs.
\begin{lem}\label{lem3}
Let $l=2w$ with $w\in\mathbb{N}$, let $\Delta_{l}$ be an admissible exponent and $s\geq 2k$. Let $b\in\mathbb{Z},$ $r\in\mathbb{N}$ with $(b,r)=1$ and let $\Theta=r+P^{k}\lvert r\alpha-b\rvert$. Then, one has
$$g_{s}(\alpha,P,R)\ll \Theta^{\varepsilon}P^{k/s+\varepsilon}\big(P^{\Delta_{l}}(\Theta^{-1}+P^{-k/2}+\Theta P^{-k})\big)^{2/l^{2}}+(P\Theta )^{\varepsilon}.$$
\end{lem}
\begin{proof}
We take $a\in\mathbb{Z}$ and $q\in\mathbb{N}$ with $(a,q)=1$ and $\lvert \alpha-a/q\rvert\leq 1/q^{2}$, and employ partial summation combined with \cite[Lemma 3.1]{Woo2} in the same vein as in \cite[Lemma 5.1]{Bru-Woo} to derive
\begin{align*}g_{s}(\alpha,P,R)\ll &q^{\varepsilon}P^{k/s+\varepsilon}\big(P^{\Delta_{l}}(q^{-1}+P^{-k/2}+q P^{-k})\big)^{2/l^{2}}+P^{-1/2+k/s+\varepsilon}
\\
&+\int_{P/2}^{P}y^{-2+k/s}\Big\lvert\sum_{x\in\mathcal{A}(y,R)\setminus \mathcal{A}(P/2,R)}e(\alpha x^{k})\Big\rvert dy.
\end{align*}
In order to examine the preceding integral $I_{P}$, we apply the aforementioned lemma to obtain
\begin{align*}I_{P}\ll& q^{\varepsilon}\int_{P/2}^{P}y^{-1+k/s+\varepsilon}\big(y^{\Delta_{l}}(q^{-1}+y^{-k/2}+qy^{-k})\big)^{2/l^{2}}dy+\int_{P/2}^{P}y^{-3/2+k/s+\varepsilon}dy.
\end{align*}
Computing the above integrals and combining the previous equations one gets
\begin{equation*}g_{s}(\alpha,P,R)\ll q^{\varepsilon}P^{k/s+\varepsilon}\big(P^{\Delta_{l}}(q^{-1}+P^{-k/2}+q P^{-k})\big)^{2/l^{2}}+(qP)^{\varepsilon}.\end{equation*}
The statement of the lemma thereby follows replacing $q$ in the preceding equation by $\Theta=r+P^{k}\lvert r\alpha-b\rvert$ via the transference principle (see \cite[Lemma 14.1]{Woo33}).
\end{proof}

We shall explore the potential of the preceding analysis and suppose that $(\Delta_{2w})_{w\in\mathbb{N}}$ is a collection of admissible exponents. We then observe by the equation before (5.1) of \cite{Bru-Woo} that \begin{equation}\label{taucer}\tau(k)\leq 1/4k,\end{equation} where $\tau(k)$ was defined in (\ref{eq67}). We also introduce for any real number $t\geq 2$ the parameter
\begin{equation}\label{eq8}\Delta_{t}^{*}=\min_{\substack{1\leq v\leq t/2\\ v\in\NN}}\Big(\max\Big(\Delta_{2v}-(t-2v)\tau(k), \Delta_{2v}-(t-2v)k/s\Big)\Big)\end{equation} and say that $\Delta_{t}^{*}$ is an \emph{admissible exponent for minor arcs}. The reader may observe that the preceding definition differs midly from that in \cite[(5.3)]{Bru-Woo}, though in practice the results involving it shall ultimately deliver consequences of the same strength.
\begin{prop}\label{prop222} Whenever $t\geq 2$ with $s\geq 2k$ and $1\leq Q\leq \frac{1}{2}P^{k/2}R^{-k}$ one has
$$\int_{\grN(Q,P)}\lvert g_{s}(\alpha,P,R)\rvert^{t}d\alpha\ll P^{tk/s-k+\varepsilon}Q^{2\Delta_{t}^{*}/k}.$$
\end{prop}
\begin{proof}
We recall (\ref{M}) and observe by the discussion after equation (5.5) of \cite{Bru-Woo} that whenever $\alpha\in\grN(Q,P/m)$ for some $m\in \mathscr{B}(M,\pi,R)$, there exist $b\in\mathbb{Z}$ and $r\in\mathbb{N}$ with $(b,r)=1$ and
$$\frac{1}{2}Qz^{-k}<r+\Big(\frac{P}{mz}\Big)^{k}\lvert r(\alpha z^{k})-b\rvert\leq 2Q.$$
Equipped with this remark we deduce from Lemma \ref{lem3} that for $\pi\leq R$ prime then
$$g_{s}(\alpha z^{k},P/mz,\pi)\ll P^{\varepsilon}\Big(\frac{P}{mz}\Big)^{k/s}\Big(\Big(\frac{P}{mz}\Big)^{\Delta_{l}}\Big(z^{k}/Q+\Big(\frac{mz}{P}\Big)^{k/2}+\Big(\frac{mz}{P}\Big)^{k}Q\Big)\Big)^{2/l^{2}}+P^{\varepsilon}.$$
We next take $l=2w$ corresponding to the maximum in (\ref{eq67}) and note upon recalling (\ref{M}) that whenever $M\leq m\leq MR$ and $\alpha\in\grN(Q,P/m)$ it transpires that
\begin{align*}g_{s}(\alpha z^{k},P/mz,\pi)&\ll P^{\varepsilon}\Big(\frac{P}{mz}\Big)^{k/s+(\Delta_{l}-k/2)\frac{2}{l^{2}}}z^{\frac{k}{l^{2}}}+P^{\varepsilon}\ll P^{\varepsilon}\Big(\frac{P}{mz}\Big)^{k/s-\tau(k)}z^{\frac{k}{l^{2}}}+P^{\varepsilon}.
\end{align*}
We consider $v$ minimising (\ref{eq8}) and set $t=t_{0}+2v$. Recalling equation (\ref{eq7}) one has
\begin{align*}V_{t}(\pi, m,z,\frak{N})&\ll P^{\varepsilon}\Big(\Big(\frac{P}{mz}\Big)^{t_{0}(k/s-\tau(k))}z^{t_{0}k/l^{2}}+1\Big)\int_{0}^{1}\lvert g_{s}(\alpha z^{k}, P/mz, \pi)\rvert^{2v}d\alpha.
\\
&\ll  P^{\varepsilon}\Big(\frac{P}{mz}\Big)^{tk/s-k+\Delta_{t}^{*}}z^{t_{0}k/l^{2}}+P^{\varepsilon}(P/mz)^{2vk/s-k+ \Delta_{2v}},
\end{align*}
where in the last step we routinarily applied orthogonality in conjunction with Lemma \ref{lem1}. We draw the reader's attention back to (\ref{eq6}) and apply the preceding bound to get
\begin{align*}\sum_{1\leq q\leq Q}I_{q,t}(M,\frak{N})\ll& P^{tk/s-k+\Delta_{t}^{*}+\varepsilon}M^{-(t-1)k/s-\Delta_{t}^{*}}\max_{P_{0}\leq P} P_{0}^{-t+k+1-\Delta_{2v}+2t_{0}(k-\Delta_{l})/l^{2}}
\\
&+P^{2vk/s-k+\Delta_{2v}+\varepsilon}M^{(1-2v)k/s-\Delta_{2v}}\max_{P_{0}\leq P}P_{0}^{-t+k+1-\Delta_{2v}+t_{0}k/s}.
\end{align*}

We observe next that as a consequence of the condition $t\geq 2$ and the inequalities $\Delta_{l}\geq k-l/2$ and $\Delta_{2v}\geq k-v$, as was observed in Proposition \ref{prop2.1}, one gets
$$-t+k+1-\Delta_{2v}+\max(2t_{0}(k-\Delta_{l})/l^{2},t_{0}k/s)\leq -t+1+\frac{t_{0}}{2}+v\leq 1-t/2\leq 0,$$
where we employed the provisos $s\geq 2k$ and $l\geq 2$. Combining the above bounds entails
$$\sum_{1\leq q\leq Q}I_{q,t}(M,\frak{N})\ll P^{tk/s-k+\Delta_{t}^{*}+\varepsilon}M^{-(t-1)k/s-\Delta_{t}^{*}}.$$
We may conclude the proof by recalling (\ref{M}) and observing that Lemma \ref{lem2} yields
\begin{align*}\int_{\frak{N}(Q,P)}\lvert g_{s}(\alpha,P,R)\rvert^{t}d\alpha& \ll P^{\varepsilon}M^{(t-1)k/s}\sum_{1\leq q\leq Q}I_{q,t}(M,\frak{N})+Q^{2}M^{t}P^{\varepsilon-k-t(1-k/s)}
\\
&\ll P^{tk/s-k+\varepsilon}\Big(Q^{2\Delta_{t}^{*}/k}+Q^{2-2t/k}\Big)\ll P^{tk/s-k+\varepsilon}Q^{2\Delta_{t}^{*}/k},\end{align*} where in the last step we used (\ref{taucer}), the inequality $\Delta_{2v}\geq k-v$ and $s\geq 2k$ to deduce 
$$\Delta_{t}^{*}\geq \max\Big(k-v-(t-2v)k/s,k-v-\frac{(t-2v)}{4k}\Big)\geq k-t/2\geq k-t.$$
\end{proof}
It thereby transpires that we have prepared the ground to obtain mean value estimates over minor arcs, these being defined for $1\leq Q\leq P^{k/2}$ by means of $\grm(Q)=[0,1)\setminus \grM(Q,P).$
\begin{prop}\label{prop2222}
Let $t\geq 2$ and $s\geq 2k$. Let $\theta>0$ be a real number satisfying $\theta\leq k/2$. Then, when $P^{\theta}\leq Q\leq P^{k/2}$ and $\Delta_{t}^{*}<0$ one has the estimate
$$\int_{\grm(Q)}\lvert g_{s}(\alpha, P, R)\rvert^{t}d\alpha\ll_{\theta} P^{tk/s-k+\varepsilon}Q^{-2\lvert \Delta_{t}^{*}\rvert/k}. $$
\end{prop}
\begin{proof}
We begin by writing $J_{Q}=\big\lceil \frac{\log (P^{k/2}/Q)}{\log 2}\big\rceil$ and $J_{0}=\big\lceil \frac{\log (2 R^{k})}{\log 2}\big\rceil.$ It is a consequence of Dirichlet's approximation theorem (see the argument before (2.1) and (5.9) of \cite{Bru-Woo}) that
\begin{equation*}[0,1)=\bigcup_{0\leq j\leq J_{1}}\grN(2^{-j}P^{k/2},P),\ \ \ \ \ \ \ \ \ \grm(Q)\subset \bigcup_{0\leq j\leq J_{Q}}\grN(2^{-j}P^{k/2},P).\end{equation*}
We next observe that if $0\leq j\leq J_{0}$ then by the argument in \cite[Theorem 5.3]{Bru-Woo} it follows that whenever $\alpha\in \grN(2^{-j}P^{k/2},P)$ there exist $b\in\mathbb{Z}$ and $r\in\mathbb{N}$ with $(b,r)=1$ satisfying $P^{k/2}R^{-k}\ll r+P^{k}\lvert r\alpha-b\rvert\ll P^{k/2}.$ Therefore, Lemma \ref{lem3} combined with (\ref{eq67}) gives
\begin{equation*}g_{s}(\alpha,P,R)\ll P^{k/s-\tau(k)+\varepsilon}+P^{\varepsilon},\end{equation*} whence recalling (\ref{eq8}) and writing $t=t_{0}+2v$ and $P_{j}=2^{-j}P^{k/2}$, Lemma \ref{lem1} yields
\begin{align*}\label{chi}&\int_{\grN(P_{j},P)}\lvert g_{s}(\alpha, P, R)\rvert^{t}d\alpha\ll \Big(\sup_{\alpha\in \grN(P_{j},P)}\lvert  g_{s}(\alpha, P, R)\rvert\Big)^{t_{0}}\int_{0}^{1}\lvert g_{s}(\alpha, P, R)\rvert^{2v}d\alpha
\\
&\ll P^{tk/s-k+\Delta_{2v}-t_{0}\tau(k)+\varepsilon}+P^{2vk/s-k+\Delta_{2v}+\varepsilon}\ll P^{tk/s-k-\lvert \Delta_{t}^{*}\rvert+\varepsilon}\ll P^{tk/s-k+\varepsilon}Q^{-2\lvert \Delta_{t}^{*}\rvert/k}.\nonumber
\end{align*}
Consequently, the above bound and Proposition \ref{prop222} for the range $J_{0}<j\leq J_{Q}$ deliver
$$\int_{\grm(Q)}\lvert g_{s}(\alpha, P, R)\rvert^{t}d\alpha\ll \sum_{j=0}^{J_{Q}}\int_{\grN(2^{-j}P^{k/2},P)}\lvert g_{s}(\alpha,P,R)\rvert^{t}d\alpha \ll P^{tk/s-k+\varepsilon}Q^{-2\lvert \Delta_{t}^{*}\rvert/k}.$$ 
\end{proof}

\section{Preliminary major arc manoeuvres}\label{sec4}
We begin by stating some routinary estimates for weighted versions of auxiliary exponential sums and integrals over the major arcs, it being worth introducing beforehand \begin{equation}\label{i0}\tilde{i}_{s}=\Big\lceil\frac{ \log 2s}{k\log 2}\Big\rceil\ \ \ \ \ \ \ \ \ \ \ \ \ \ \ \  \ \ \ P_{-}=2^{-\tilde{i}_{s}-1}P,\end{equation} and for $\beta\in\mathbb{R}$ the exponential sums
\begin{equation}\label{po}w_{s}(\beta)=\frac{1}{k}\sum_{1\leq x\leq P^{k}}x^{-1+1/s}e(\beta x),\ \ \ \ \ \ \  \ \ \ \ \tilde{w}_{s}(\beta)=\frac{1}{k}\sum_{P_{-}^{k}< x\leq P^{k}}x^{-1+1/s}e(\beta x).\end{equation}

\begin{lem}\label{lem5.1}
Whenever $\lvert \beta\rvert\leq 1/2$ one has
$$w_{s}(\beta)\ll \frac{P^{k/s}}{(1+P^{k}\lvert \beta\rvert)^{1/s}}\ \ \ \ \ \ \text{and}\ \ \ \ \ \ \ \ \tilde{w}_{s}(\beta)\ll \frac{P^{k/s}}{1+P^{k}\lvert \beta \rvert}.$$
\end{lem}
\begin{proof}
The first part follows by \cite[Lemma 2.8]{Vau} and the second one by \cite[Lemma 6.2]{Vau}.
\end{proof}
The next lemma shall deliver similar bounds for integral analogues of the preceding sums.
\begin{lem}\label{lem5.22}
Let $\theta_{0},\theta_{1}> 0$ be real numbers satisfying $\theta_{0}<k\theta_{1}$ and let $\beta\in\mathbb{R}$ such that $\lvert \beta\rvert\leq 1/2.$ If $c>0$ is some fixed constant then 
$$\int_{cP}^{P}\frac{y^{\theta_{0}-1}}{(1+y^{k}\lvert \beta\rvert)^{\theta_{1}}}dy\ll \frac{P^{\theta_{0}}}{(1+P^{k}\lvert \beta\rvert)^{\theta_{1}}},\ \ \ \ \ \ \ \ \ \ \ \int_{1}^{P}\frac{y^{\theta_{0}-1}}{(1+y^{k}\lvert \beta\rvert)^{\theta_{1}}}dy\ll \frac{P^{\theta_{0}}}{(1+P^{k}\lvert \beta\rvert)^{\theta_{0}/k}}.$$ 
\end{lem}
\begin{proof}
In order to examine the second integral, we assume first $\lvert \beta\rvert^{-1}<P^{k}$ and divide the range of integration into $[1,\lvert \beta\rvert^{-1/k}]$ and $[\lvert \beta\rvert^{-1/k},P]$. We trivially obtain for the first one $$ \int_{1}^{\lvert \beta\rvert^{-1/k}}\frac{y^{\theta_{0}-1}}{(1+y^{k}\lvert \beta\rvert)^{\theta_{1}}}dy\ll  \int_{1}^{\lvert \beta\rvert^{-1/k}}y^{\theta_{0}-1}dy\ll \lvert \beta \rvert^{-\theta_{0}/k}\ll \frac{P^{\theta_{0}}}{(1+P^{k}\lvert \beta\rvert)^{\theta_{0}/k}}.$$ We then employ the restriction on the exponents and note that
$$ \int_{\lvert \beta\rvert^{-1/k}}^{P}\frac{y^{\theta_{0}-1}}{(1+y^{k}\lvert \beta\rvert)^{\theta_{1}}}dy\ll  \int_{\lvert \beta\rvert^{-1/k}}^{P}y^{-1-k\theta_{1}+\theta_{0}}\lvert \beta\rvert^{-\theta_{1}}dy\ll \lvert \beta \rvert^{-\theta_{0}/k}.$$ Combining the previous equations with a trivial estimate when $\lvert \beta \rvert^{-1}\geq P^{k}$ yields the second estimate in the statement of the lemma. For the first one we assume $\lvert \beta\rvert^{-1}<P^{k}$ and obtain
$$\int_{cP}^{P}\frac{y^{\theta_{0}-1}}{(1+y^{k}\lvert \beta\rvert)^{\theta_{1}}}dy\ll \int_{cP}^{P}y^{-1-k\theta_{1}+\theta_{0}}\lvert \beta\rvert^{-\theta_{1}}dy\ll P^{-k\theta_{1}+\theta_{0}}\lvert \beta\rvert^{-\theta_{1}},$$ as required. If instead $\lvert \beta \rvert^{-1}\geq P^{k}$ then the desired bound follows trivially.
\end{proof}
We shall present one last lemma concerning the customary approximation of the weighted smooth exponential sum on the major arcs by employing classical work. For such purposes, we recall (\ref{fuste}), (\ref{i0}) and introduce for $q\in\mathbb{N}$ and $a\in\mathbb{Z}$ with $(a,q)=1$ the sums
\begin{equation}\label{Sqa}S(q,a)=\sum_{r=1}^{q}e(ar^{k}/q),\ \ \ \ \ \ \ \tilde{f}_{s}(\alpha,P,R)=\sum_{\substack{x\in\mathcal{A}(P,R)\\ x>P_{-}}}x^{-1+k/s}e(\alpha x^{k}).\end{equation}
\begin{lem}\label{lem5.3}
Let $q\leq (\log P)^{1/8}$ and $a\in\mathbb{Z}$ such that $(a,q)=1$ and $\alpha\in \grM_{a,q}((\log P)^{1/8})$. Then, whenever $P^{\eta}\exp(-\eta(k^{-1}\log P)^{1/2})\leq R\leq P^{\eta}$ with $0<\eta<1/2$ one has
\begin{equation}\label{fff55}f_{s}(\alpha,P,R)-\rho(1/\eta)q^{-1}S(q,a)w_{s}(\alpha-a/q)\ll P^{k/s}(\log P)^{-1/2},\end{equation} where $\rho$ is the Dickman's function described, for instance, in \cite[Section 5]{Vau2}. The same estimate holds for $\tilde{f}_{s}$ and $\tilde{w}_{s}$ replacing $f_{s}$ and $w_{s}$ respectively.
\end{lem}
\begin{proof}
We set $\beta=\alpha-a/q$ and define for convenience when $y\leq P$ the auxiliary sums
$$S_{y}=\sum_{\substack{x\in\mathcal{A}(y,R)\\ x>R}}\Big(e(\alpha x^{k})-q^{-1}S(q,a)e(\beta x^{k})\Big),\ \ \ \ \ \ \ \ \ \ B(y)=\sum_{\substack{ x\in\mathcal{A}(y,R)\\ x>R}}e(\beta x^{k}).$$ A routinary application of summation by parts combined with equation (5.18) in \cite{Vau2} yields
\begin{align}\label{ll}\sum_{\substack{x\in\mathcal{A}(P,R)\\ x>R}}\big(e(\alpha x^{k})-q^{-1}S(q,a)e(\beta x^{k})\big)&x^{-1+k/s}\ll P^{-1+k/s}\lvert S_{P}\rvert +\int_{R}^{P}x^{-2+k/s}\lvert S_{x}\rvert dx
\\
& \ll qP^{k/s}(1+P^{k}\lvert \beta \rvert)(\log P)^{-1}\ll  P^{k/s}(\log P)^{-1/2}, \nonumber\end{align}
where on the last step we used the restrictions on $\alpha,q$. We next shift the attention to $B(y)$ and note by the equation before (5.19) in \cite[Lemma 5.4]{Vau2} that for $y> P(\log P)^{-1}$ one has
\begin{align*}B(y)&=\frac{1}{k}\sum_{R^{k}\leq m\leq y^{k}}m^{1/k-1}\rho\Big(\frac{\log m}{k\log R}\Big)e(\beta m)+O\Big(\frac{P}{\log P}(1+P^{k}\lvert \beta\rvert)\Big)
\\
&=k^{-1}\rho(1/\eta)\sum_{1\leq m\leq y^{k}}m^{1/k-1}e(\beta m)+O\big(P(\log P)^{-1/2}\big).
\end{align*}
 The reader may note that in the above formula we employed the continuity of $\rho'(u)$ whenever $u>1$ (see \cite[Section 5]{Vau2}) in conjunction with the mean value theorem. Equipped with the above utensils and recalling (\ref{po}) we may now employ Abel's summation to derive
\begin{align*}
&\sum_{x\in \mathcal{A}(P,R)}x^{-1+k/s}e(\beta x^{k})=P^{-1+k/s}B(P)+O(R^{k/s})+(1-k/s)\int_{R}^{P}y^{-2+k/s}B(y)dy
\\
&=c(\eta)\Big(\frac{1}{k}-\frac{1}{s}\Big)\sum_{m\leq P^{k}}m^{\frac{1}{k}-1}e(\beta m)\int_{m^{\frac{1}{k}}}^{P}y^{-2+\frac{k}{s}}dy+c(\eta)P^{-1+\frac{k}{s}}w_{k}(\beta)+O\big(P^{k/s}(\log P)^{-1/2}\big),
\end{align*} 
where $c(\eta)=\rho(1/\eta)$, whence integrating and rearranging terms one gets
$$\sum_{x\in \mathcal{A}(P,R)}x^{-1+k/s}e(\beta x^{k})=c(\eta)w_{s}(\beta)+O\big(P^{k/s}(\log P)^{-1/2}\big).$$ Combining the above expression with (\ref{ll}) we deduce (\ref{fff55}). In order to obtain the estimate pertaining to $\tilde{f}_{s}$ we observe that
$\tilde{f}_{s}(\alpha,P,R)=f_{s}(\alpha,P,R)-f_{s}(\alpha,P_{-},R)$ and apply (\ref{fff55}) to both of the summands.
\end{proof}

\section{Major arc estimates}\label{sec5}
We next employ the results derived in \cite{Bru-Woo2} to obtain suitable bounds via partial summation.
\begin{lem}\label{lem6.1}
Let $s\geq k+1$ with $k\geq 2$ and $2\leq R\leq P^{\eta}$ for $0<\eta<1/8$. Let $a\in\mathbb{Z},q\in\mathbb{N}$ with $(a,q)=1$. Then, when $\alpha\in \grM_{a,q}(P^{\theta})$ for some small enough $\theta=\theta(k,s)>0$ one has
\begin{align}\label{reef}f_{s}(\alpha,P,R)\ll \frac{(\log P)^{3}q^{\varepsilon-1/2k}P^{k/s}}{(1+P^{k}\lvert \alpha-a/q\rvert)^{1/s}},\ \ \ \ \ \ \ \ g_{s}(\alpha,P,R)\ll \frac{(\log P)^{3}q^{\varepsilon-1/2k}P^{k/s}}{(1+P^{k}\lvert \alpha-a/q\rvert)^{1/2}}.\end{align} 
\end{lem}
\begin{proof}
We note upon recalling (\ref{prato}) that for every $y\leq P$ then \cite[Theorem 1.1]{Bru-Woo2} delivers
\begin{align*}f(\alpha,y,R)&\ll (\log P)^{3}q^{\varepsilon}\Big(\frac{q^{-1/2k}y}{(1+y^{k}\lvert \alpha-a/q\rvert)^{1/k}}+P^{3/4}R^{1/2}(q+P^{k}\lvert q\alpha-a\rvert)^{1/8}\Big)
\\
&\ll \frac{(\log P)^{3}q^{\varepsilon-1/2k}y}{(1+y^{k}\lvert \alpha-a/q\rvert)^{1/k}},\end{align*} where we used the bound $(q+P^{k}\lvert q\alpha-a\rvert)\leq P^{\theta}$ for small $\theta$. Abel's summation then yields
\begin{align}\label{prl}f_{s}(\alpha,P,R)&\ll P^{-1+k/s}\lvert f(\alpha,P,R)\rvert+\int_{1}^{P}y^{-2+k/s}\lvert f(\alpha,y,R)\rvert dy
\\
&\ll\frac{(\log P)^{3}q^{\varepsilon-1/2k}P^{k/s}}{(1+P^{k}\lvert \alpha-a/q\rvert)^{1/s}}+(\log P)^{3}q^{\varepsilon-1/2k}\int_{1}^{P}\frac{y^{-1+k/s}}{(1+y^{k}\lvert \beta\rvert)^{1/k}}dy.\nonumber
\end{align} 
The first estimate in (\ref{reef}) thereby follows upon recalling $s>k$ by applying Lemma \ref{lem5.22} to the above integral.  In order to show the second one we combine partial summation with the bound for $g(\alpha,P,R)$ embodied in \cite[Theorem 1.1]{Bru-Woo2} and follow an analogous analysis. 
\end{proof}
\begin{lem}\label{lem5.2}
Let $s\geq k+1$ and $R,P$ be real numbers with $2\leq R\leq P^{\eta}$ for some $0<\eta<1/2$. Let $a\in\mathbb{Z}$, $q\in\mathbb{N}$ with $(a,q)=1$ and $\alpha\in\grM_{a,q}((\log P)^{A})$ for fixed $A>0$. Then one has
$$f_{s}(\alpha,P,R)\ll \frac{q^{\varepsilon-1/k}P^{k/s}}{(1+P^{k}\lvert \alpha-a/q\rvert)^{1/s}},\ \ \ \ \ \ g_{s}(\alpha,P,R)\ll \frac{q^{\varepsilon-1/k}P^{k/s}}{(1+P^{k}\lvert \alpha-a/q\rvert)}.$$ 
\end{lem}
\begin{proof}
We write $\beta=\alpha-a/q,$ observe that $q(1+y^{k}\lvert \alpha-a/q\rvert)\leq (\log P)^{A}$ and note upon recalling (\ref{prato}) that for every $ P^{\eta'}\leq y\leq P$ with $2\eta<\eta'<1$ then \cite[Theorem 1.2]{Bru-Woo2} yields
\begin{equation}\label{kk7}f(\alpha,y,R)\ll  \frac{q^{\varepsilon-1/k}y}{(1+y^{k}\lvert \beta\rvert)}+y(1+y^{k}\lvert \beta\rvert)\exp\big(-c(\log P)^{1/2}\big)\ll \frac{q^{\varepsilon-1/k}y}{(1+y^{k}\lvert \beta\rvert)},\end{equation} where $c>0$ is fixed. We recall (\ref{prl}) and observe first that the integral therein satisfies
$$ \int_{1}^{P}y^{-2+k/s}\lvert f(\alpha,y,R)\rvert dy=\int_{P^{\eta'}}^{P}y^{-2+k/s}\lvert f(\alpha,y,R)\rvert dy+O(P^{k\eta'/s}).$$ Moreover, the application of (\ref{kk7}) in conjunction with Lemma \ref{lem5.22} yields 
\begin{align*}\int_{P^{\eta'}}^{P}y^{-2+k/s}\lvert f(\alpha,y,R)\rvert dy\ll q^{\varepsilon-1/k}\int_{P^{\eta'}}^{P}\frac{y^{-1+k/s}dy}{(1+y^{k}\lvert \beta\rvert)^{1/k}}\ll q^{\varepsilon-1/k}\frac{P^{k/s}}{(1+P^{k}\lvert \beta\rvert)^{1/s}}.
\end{align*}We deduce the first estimate by combining the preceding bounds and inserting them on the first line of (\ref{prl}). The second one follows in a analogous manner.
\end{proof}

\section{Further pruning}\label{sec6}
We shall combine the work of last sections to estimate the contribution of major arcs of intermediate height. To such an end  we begin by defining for every $\alpha\in\grM_{a,q}(\frac{1}{2}P^{k/2},P)$ with $a\in\mathbb{Z}$ and $q\in\mathbb{N}$ satisfying $(a,q)=1$ and $0\leq a\leq q\leq \frac{1}{2}P^{k/2}$ the function
\begin{equation*}\Upsilon(\alpha)=q^{-2}(1+P^{k}\lvert \alpha-a/q\rvert)^{-1}.\end{equation*} If on the contrary $\alpha\notin \grM_{a,q}(\frac{1}{2}P^{k/2},P)$ for all $a,q$ as above we set $\Upsilon(\alpha)=0$. Observe that in view of the fact that the preceding intervals are disjoint this defines a function in $[0,1)$. In what follows we write $\grM(Q)$ to denote $\grM(Q,P)$ for simplicity.
\begin{lem}\label{kk}
Let $r>1$. For real numbers $Q_{0},Q_{1}>0$ satisfying $1\leq Q_{0}<Q_{1}\leq P^{k/2}$ then 
$$\int_{\grM(Q_{1})\setminus \grM(Q_{0})} \Upsilon(\alpha)^{r}d\alpha\ll P^{-k}Q_{0}^{1-r}.$$
\end{lem}
\begin{proof}
It seems worth noting first that whenever $\alpha\in\grM(Q_{1})\setminus \grM(Q_{0})$, there are $a\in\mathbb{Z}$ and $q\in\mathbb{N}$ satisfying $(a,q)=1$ and $0\leq a\leq q\leq Q_{1}$ with $\lvert \alpha-a/q\rvert\leq Q_{1}/qP^{k}$ such that either 
$\lvert \alpha-a/q\rvert> \frac{Q_{0}}{qP^{k}}$ or $q>Q_{0}.$ Consequently, the integral at hand may be estimated by
\begin{align*}&\int_{\grM(Q_{1})\setminus \grM(Q_{0})}\Upsilon(\alpha)^{r}d\alpha\ll \sum_{q\leq Q_{1}}\sum_{\substack{a=1\\ (a,q)=1}}^{q}q^{-2r}\int_{\lvert \beta\rvert> Q_{0}/qP^{k}}\frac{d\beta}{(1+P^{k}\lvert \beta\rvert)^{r}}
\\
&+\sum_{q>Q_{0}}\sum_{\substack{a=1\\ (a,q)=1}}^{q}q^{-2r}\int_{\lvert \beta\rvert\leq  Q_{1}/qP^{k}}\frac{d\beta}{(1+P^{k}\lvert \beta\rvert)^{r}}\ll  P^{-k}Q_{0}^{1-r}\sum_{q\leq Q_{1}}q^{-r}+P^{-k}\sum_{q>Q_{0}}q^{1-2r},
\end{align*}
which concludes the proof in view of the assumption $r>1$.
\end{proof}

Equipped with the above lemma we shall present yet another major arc type estimate.
\begin{lem}\label{lem6.2}
Let $s\geq k+1$ and $t\geq 4k+1$. Let $0<\theta<\theta_{0}(k,s)$ for some small enough $\theta_{0}(k,s)>0$ and assume that $1\leq Q\leq P^{\theta}$. Then one has
\begin{equation}\label{eq210}\int_{\grM(P^{\theta})\setminus \grM(Q)}\lvert g_{s}(\alpha,P,R)\rvert^{t}d\alpha\ll P^{tk/s-k}Q^{-1/53k}.\end{equation} 
\end{lem} 
\begin{proof}
We prepare the ground by applying Lemma \ref{lem6.1} whenever $\alpha\in\grM(P^{\theta})$ to derive
\begin{align*}g_{s}(\alpha,P,R)&\ll (\log P)^{3}P^{k/s}\Upsilon(\alpha)^{1/4k-\varepsilon}.
\end{align*}
Combining the preceding estimate with the conclusion of Lemma \ref{kk} and the fact that $\text{meas}(\grM(P^{\theta}))\leq P^{2\theta-k}$ and writing $\grM_{\theta,Q}=\grM(P^{\theta})\setminus \grM(Q)$ permits one to deduce
\begin{align*}\label{rrrrr}&\int_{\grM_{\theta,Q}}\lvert g_{s}(\alpha,P,R)\rvert^{t}d\alpha\ll (\log P)^{3t}P^{tk/s}\int_{\grM_{\theta,Q}}\Upsilon(\alpha)^{t/4k-\varepsilon}d\alpha\ll (\log P)^{3t}P^{tk/s-k}Q^{\varepsilon-1/4k}.
\end{align*}
Consequently, equation (\ref{eq210}) would follow provided $Q\geq L_{P}$, where $L_{P}=(\log P)^{13tk}$. If instead $Q< L_{P}$ one may apply Lemma \ref{lem5.2} to get whenever $ \alpha\in\grM(L_{P})\setminus \grM(Q)$ that 
\begin{equation*}g_{s}(\alpha,P,R)\ll P^{k/s}\Upsilon(\alpha)^{-\varepsilon+1/2k}.\end{equation*} Set $\grM_{1}=\grM(P^{\theta})\setminus\grM(L_{P})$, $\grM_{2}=\grM(L_{P})\setminus \grM(Q)$ and combine the above estimates to get
\begin{align*}
& \int_{\grM_{1}}\lvert g_{s}(\alpha,P,R)\rvert^{t}d\alpha+\int_{\grM_{2}}\lvert g_{s}(\alpha,P,R)\rvert^{t}d\alpha\ll P^{\frac{tk}{s}-k}(\log P)^{\varepsilon-t/4}+P^{\frac{tk}{s}}\int_{\grM_{2}}\Upsilon(\alpha)^{t/2k-\varepsilon}d\alpha,
\end{align*}
whence another application of Lemma \ref{kk} delivers 
$$\int_{\grM(P^{\theta})\setminus \grM(Q)}\lvert g_{s}(\alpha,P,R)\rvert^{t}d\alpha\ll P^{tk/s-k}(\log P)^{\varepsilon-t/4}+P^{tk/s-k}Q^{\varepsilon-1-1/2k},$$ which yields the desired result.
\end{proof}
Having been furnished with the preceding bounds we derive the following.
\begin{prop}\label{prop4}
Let $s\geq 2k$ and $ t\geq 4k+1$. Let $1\leq Q\leq P^{k/2}$ and assume that $\Delta_{t}^{*}<0$ is an admissible exponent for minor arcs. Then for every $\tilde{\nu}<\min(2\lvert \Delta_{t}^{*}\rvert/k, 1/53k)$ one has
\begin{equation}\label{eq21}\int_{\grm(Q)}\lvert g_{s}(\alpha,P,R)\rvert^{t}d\alpha\ll P^{tk/s-k}Q^{-\tilde{\nu}},\ \ \ \ \ \ \ \ \ \ \ \ \ \ \int_{0}^{1}\lvert g_{s}(\alpha,P,R)\rvert^{t}d\alpha\ll P^{tk/s-k}.\end{equation} 
\end{prop} 
\begin{proof}
In view of Proposition \ref{prop2222}, it suffices to show the first estimate whenever $1\leq Q\leq P^{\theta}$ for sufficiently small $\theta$. Such a proposition combined with Lemma \ref{lem6.2} thereby delivers
\begin{align*}\int_{\grm(Q)}\lvert g_{s}(\alpha,P,R)\rvert^{t}d\alpha&\ll \int_{\grm(P^{\theta})}\lvert g_{s}(\alpha,P,R)\rvert^{t}d\alpha+\int_{\grM(P^{\theta})\setminus \grM(Q)}\lvert g_{s}(\alpha,P,R)\rvert^{t}d\alpha\nonumber
\\
&\ll P^{tk/s-k-\theta\tilde{\nu}}+ P^{tk/s-k}Q^{-1/53k},
\end{align*}
as desired. In order to obtain the second estimate in (\ref{eq21}) we set $Q=1$ in (\ref{eq21}) to get
\begin{align*}\int_{\grm(1)}\lvert g_{s}(\alpha,P,R)\rvert ^{t} d\alpha+\int_{\grM(1)}\lvert g_{s}(\alpha,P,R)\rvert ^{t} d\alpha\ll P^{tk/s-k}+\int_{\grM(1)}\lvert g_{s}(\alpha,P,R)\rvert^{t}  d\alpha.
\end{align*} Inserting the bounds $\text{meas}(\grM(1))\leq P^{-k}$ and $g_{s}(\alpha,P,R)\ll P^{k/s}$ concludes the proof.
\end{proof}
\section{An asymptotic evaluation}\label{sec7}
We shall complete the major arc analysis in the context of Theorems \ref{thm1.1} and \ref{thm1.2} by deriving the relevant asymptotic formula. We thus consider for $N\in\mathbb{N}$ the parameter \begin{equation}\label{paar}P=(2N)^{1/k}\end{equation} and define for $a\in\mathbb{Z}$ and $q\in\mathbb{N}$ with $(a,q)=1$ and $0\leq a\leq q\leq (\log P)^{1/8}$ the arcs
\begin{equation}\label{vv}\frak{K}(a,q)=\Big\{\alpha\in [0,1):\ \lvert \alpha-a/q\rvert\leq (\log P)^{1/8}P^{-k}\Big\},\end{equation} denote $\frak{K}$ to the union of such sets and $\frak{k}=[0,1)\setminus \frak{K}.$ We further introduce for $0\leq j\leq s$ and whenever $n\in\NN$ with $N\leq n\leq 2N$ the integral 
\begin{equation}\label{rsk}r_{j}(n,R)=\int_{0}^{1}\tilde{f}_{s}(\alpha,P,R)^{j}f_{s}(\alpha,P,R)^{s-j}e(-\alpha n)d\alpha,\end{equation} 
which by orthogonality and upon recalling (\ref{Sqa}) satisfies
\begin{equation}\label{fasi} r_{j}(n,R)=\sum_{\substack{(x_{1},\ldots,x_{s})\in \mathcal{C}(n,R)\\ x_{l}> P_{-}^{k},\ \ l\leq j }}(x_{1}\cdots x_{s})^{-1+k/s},\end{equation} where 
\begin{equation}\label{pij}\mathcal{C}(n,R)=\Big\{(x_{1},\ldots,x_{s})\in \mathbb{N}:\ \ x_{1}^{k}+\ldots+x_{s}^{k}=n,\ \ \ \ x_{i}\in\mathcal{A}(P,R)\Big\}.\end{equation}
We also write $r_{s,k}(n,R)=r_{0}(n,R)$. In view of the preceding definitions it transpires that
\begin{equation}\label{ooo}r_{j}(n,R)=\int_{\frak{K}}\tilde{f_{s}}(\alpha,P,R)^{j}f_{s}(\alpha,P,R)^{s-j}e(-\alpha n)d\alpha+O\Big(\int_{\frak{k}}\lvert \tilde{f_{s}}(\alpha,P,R)^{j} f_{s}(\alpha,P,R)^{s-j}\rvert d\alpha\Big).\end{equation} 

We shall bound the contribution of the minor arcs with the aid of a more general estimate that shall be employed on multiple contexts, it being pertinent presenting the constant \begin{equation}\label{nu0}\nu=\min(\lvert \Delta_{s}^{*}\rvert/2sk, 1/107sk).\end{equation} We also introduce for a measurable set $\frak{B}\subset [0,1)$, a fixed real number $\gamma\in\mathbb{R}\setminus\{0\}$, any parameter $1\leq P_{0}\leq P$ and integers $1\leq j\leq s$ and $0\leq l\leq s-j$ the integral 
\begin{equation}\label{cho}I_{\frak{B},\gamma}^{j,l}(P,P_{0})=\int_{\frak{B}}\lvert \tilde{f_{s}}(\alpha,P,R)\rvert^{j}  \lvert f_{s}(\gamma\alpha,P,R)\rvert ^{s-j-l} \lvert f_{s}(\alpha,P_{0},R)\rvert^{l} d\alpha,\end{equation} and, for each $P_{1},P_{2},P_{3}\leq P$, its counterpart
$$I_{\frak{B},\gamma}^{j,l}(P_{1},P_{2},P_{3})=\int_{\frak{B}}\lvert g_{s}(\alpha,P_{1},R)\rvert^{j}  \lvert g_{s}(\gamma\alpha,P_{2},R)\rvert ^{s-j-l} \lvert g_{s}(\alpha,P_{3},R)\rvert^{l} d\alpha.$$
\begin{lem}\label{lem8.10}
Let $1\leq j\leq s$ and $0\leq l\leq s-j$. Let $1\leq P_{0}\leq P$ and $1\leq Q\leq P^{k/2}$ and assume that $s\geq 4k+1$ and $\gamma\in\mathbb{R}\setminus\{0\}$. Then, whenever $\Delta_{s}^{*}<0$ and $l\leq 1$ one has
\begin{equation}\label{hh}I_{\grm(Q),\gamma}^{j,l}(P,P_{0})\ll Q^{-\nu}.\end{equation} If instead $2\leq l\leq s-j$ but $P_{0}=1$, the same estimate holds.
\end{lem}

\begin{proof}
We make for convenience and upon recalling (\ref{i0}) the dyadic dissections \begin{equation}\label{oo}\tilde{f_{s}}(\alpha,P,R)=\sum_{0\leq i\leq\tilde{i}_{s}}g_{s}(\alpha,2^{-i}P,R),\ \ \ \ \ \ \ \ \ \ f_{s}(\alpha,P,R)=\sum_{0\leq i\leq \frac{\log P}{\log 2}}g_{s}(\alpha,2^{-i}P,R).\end{equation}
We observe first that an application of Holder's inequality and a change of variable deliver
\begin{align*}I&_{\grm(Q),\gamma}^{j,l}(P,P_{0})\ll (\log P)^{s-1}\sum_{0\leq i\leq \frac{\log P}{\log 2}}\ \ \sum_{0\leq r\leq \frac{\log P}{\log 2}}\max_{\overline{P}\asymp P}I_{\grm(Q),\gamma}^{j,l}(\overline{P},2^{-i}P,2^{-r}P)
\\
&\ll (\log P)^{s}\sum_{0\leq i\leq \frac{\log P}{\log 2}}\ \ \max_{\overline{P}\asymp P}\Big(\int_{\grm (Q)}\lvert g_{s}(\alpha,\overline{P},R)\rvert^{s}d\alpha\Big)^{1/s}  \Big(\int_{0}^{1}\lvert g_{s}(\alpha,2^{-i}P,R)\rvert^{s}d\alpha\Big)^{1-1/s}.
\end{align*}
In order to prepare the ground for the application of Proposition \ref{prop4} we note that if $\overline{P}\asymp P$ then $\grM(\overline{c}Q,\overline{P})\subset \grM(Q,P)$ whenever $\overline{c}>0$ is a sufficiently small constant. Therefore, 
\begin{equation}\label{dd}I_{\grm(Q),\gamma}^{j,l}(P,P_{0})\ll  (\log P)^{s+1}\max_{\overline{P}\asymp P}\Big(\int_{\grm (Q)}\lvert g_{s}(\alpha,\overline{P},R)\rvert^{s}d\alpha\Big)^{1/s}\ll (\log P)^{s+1}Q^{-2\nu}.\end{equation} Upon denoting $L_{\nu}(P)=(\log P)^{(s+1)/\nu}$, the bound (\ref{hh}) would follow unless $Q\leq L_{\nu}(P)$. Otherwise we apply Lemmata \ref{lem5.2} and \ref{kk} and recall $s\geq 4k+1$ to deduce for $l\leq 1$ that
\begin{align*}I_{\grM(L_{\nu}(P))\setminus \grM(Q),\gamma}^{j,l}(P,P_{0})&\ll P^{k}\int_{\grM(L_{\nu}(P))\setminus \grM(Q)}\Upsilon(\alpha)^{3/2}d\alpha\ll Q^{-1/2}.
\end{align*}
If instead $l\geq 2$ and $P_{0}=1$ then using the fact that $\text{meas}(\grM(L_{\nu}(P)))\ll P^{\varepsilon-k}$ in conjunction with the trivial bounds for the corresponding exponential sums delivers
$$I_{\grM(L_{\nu}(P))\setminus \grM(Q),\gamma}^{j,l}(P,1)\ll P^{-lk/s+\varepsilon}.$$
Consequently, the above estimates in conjunction with (\ref{dd}) for $Q=L_{\nu}(P)$ deliver
\begin{align*}I_{\grm(Q),\gamma}^{j,l}(P,P_{0})=I_{\grm(L_{\nu}(P)),\gamma}^{j,l}(P,P_{0})+I_{\grM(L_{\nu}(P))\setminus \grM(Q),\gamma}^{j,l}(P,P_{0})&\ll (\log P)^{-(s+1)}+Q^{-\frac{1}{2}}\ll Q^{-\nu}.\end{align*}
\end{proof}
Equipped with the above lemma we shall now bound the minor arc contribution in (\ref{ooo}).
\begin{cor}\label{cor1}
Let $s\geq 4k+1$ satisfying $\Delta_{s}^{*}<0$. Then, whenever $1\leq j\leq s$ one has
$$\int_{\frak{k}}\lvert \tilde{f_{s}}(\alpha,P,R)\rvert^{j} \lvert f_{s}(\alpha,P,R)\rvert^{s-j}d\alpha\ll (\log P)^{-\nu/15}.$$
\end{cor}
\begin{proof}
By taking $Q=(\log P)^{1/15}$ we observe that $\grM(Q)\subset\frak{K}$, and hence $\frak{k}\subset \grm(Q)$. We thus apply Lemma \ref{lem8.10} for the choices $l=0$ and $\gamma=1$ and the previous remark to obtain
$$\int_{\frak{k}}\lvert \tilde{f_{s}}(\alpha,P,R)\rvert^{j} \lvert f_{s}(\alpha,P,R)\rvert^{s-j}d\alpha\ll I_{\grm(Q),1}^{j,0}(P,P)\ll (\log P)^{-\nu/15},$$ as desired.
\end{proof}
We shall next shift the reader's attention to (\ref{ooo}) for the purpose of computing the major arc contribution. We recall beforehand the singular series defined in (\ref{SSi}) and introduce
\begin{equation}\label{cks}c_{k,s}(\eta)=\frac{1}{k^{s}}\rho(1/\eta)^{s}\Gamma(1/s)^{s}.\end{equation}
\begin{prop}\label{prop7.1}
Let $s\geq 4k+1$. Then one has whenever $n\in\NN$ and $N\leq n\leq 2N$ that
\begin{equation*}\sum_{j=1}^{s}(-1)^{j+1}\binom{s}{j}\int_{\frak{K}}\tilde{f_{s}}(\alpha,P,R)^{j}f_{s}(\alpha,P,R)^{s-j}e(-\alpha n)d\alpha=c_{k,s}(\eta)\frak{S}(n)+O((\log n)^{-\frac{1}{16}}).\end{equation*}
Consequently, if $\Delta_{s}^{*}<0$ is an admissible exponent for minor arcs and the inequality $P^{\eta}\exp(-\eta(k^{-1}\log P)^{1/2})\leq R\leq P^{\eta}$ holds one obtains
\begin{equation}\label{parral}r_{s,k}(n,R)=c_{k,s}(\eta)\frak{S}(n)+O\big((\log n)^{-\nu/15 }\big).\end{equation}
\end{prop}
\begin{proof}
We fix $Q=(\log P)^{1/8}$, recall (\ref{po}), (\ref{Sqa}) and consider when $1\leq j\leq s$ for convenience 
$$J_{j}(n,Q)=\int_{-\frac{Q}{P^{k}}}^{\frac{Q}{P^{k}}}\tilde{w}_{s}(\beta)^{j}w_{s}(\beta)^{s-j}e(-\beta n)d\beta,\ \ \ \ \ \ \frak{S}(n,Q)=\sum_{q=1}^{Q}\sum_{\substack{a=1\\ (a,q)=1}}^{q}q^{-s}S(q,a)^{s}e(-an/q).$$
We recall (\ref{vv}), observe that $\text{meas}(\frak{K})\ll Q^{3}P^{-k}$ and employ Lemma \ref{lem5.3} to get
\begin{equation}\label{kapa1}\int_{\frak{K}}\tilde{f_{s}}(\alpha,P,R)^{j}f_{s}(\alpha,P,R)^{s-j}e(-\alpha n)d\alpha=\rho(1/\eta)^{s}\frak{S}(n,Q)J_{j}(n,Q)+O((\log n)^{-1/8}).\end{equation}
It further seems worth denoting 
$$J_{j}(n)=\int_{-1/2}^{1/2} \tilde{w}_{s}(\beta)^{j} w_{s}(\beta)^{s-j}e(-\beta n)d\beta,$$ and observing that the application of Lemma \ref{lem5.1} permits one to deduce that 
\begin{align*}
J_{j}(n)-J_{j}(n,Q)&\ll \int_{QP^{-k}}^{1}\lvert \tilde{w}_{s}(\beta)\rvert^{j}\lvert w_{s}(\beta)\rvert^{s-j}d\beta\ll P^{k}\int_{QP^{-k}}^{1}(1+P^{k}\lvert \beta\rvert)^{-2+\frac{1}{s}}d\beta\ll Q^{-1+\frac{1}{s}}.
\end{align*}

In order to compute $J_{j}(n)$ we merely utilise orthogonality to observe 
$$J_{j}(n)=k^{-s}\sum_{\substack{ m_{1}+\ldots+m_{s}=n\\ m_{1},\ldots,m_{j}>P_{-}^{k}}}(m_{1}\cdots m_{s})^{1/s-1}, $$ 
and note that an inclusion-exclusion argument combined with the previous equation yields
$$\sum_{j=1}^{s}(-1)^{j+1}\binom{s}{j}J_{j}(n)=k^{-s}\sum_{\substack{ m_{1}+\ldots+m_{s}=n\\ \max\limits_{i\leq s} (m_{i})> P_{-}^{k} }}(m_{1}\cdots m_{s})^{\frac{1}{s}-1}=k^{-s}\sum_{\substack{ m_{1}+\ldots+m_{s}=n }}(m_{1}\cdots m_{s})^{\frac{1}{s}-1},$$
where we used the fact that the condition on $\max (m_{i})$ is redundant. We shall denote by $\tilde{J}(n)$ to the right side of the equation. We draw the reader's attention to \cite[Theorem 2.3]{Vau} and observe that using the notation therein and setting $k$ to be $s$ on that context one has $(s/k)^{s}J(n)=\tilde{J}(n)$, whence the application of such a theorem would yield
\begin{equation}\label{J} \tilde{J}(n)=\Big(\frac{s}{k}\Big)^{s}\Gamma\Big(1+\frac{1}{s}\Big)^{s}+O(n^{-1/s})=k^{-s}\Gamma\Big(\frac{1}{s}\Big)^{s}+O(n^{-1/s}).\end{equation}

The singular series is handled routinarily by invoking the estimate $S(q,a)\ll q^{1-1/k+\varepsilon}$ in \cite[Theorem 4.2]{Vau} to get whenever $s\geq 4k+1$ the bound $\mathfrak{S}(n)\ll 1$ and the approximation
\begin{equation}\label{S} \frak{S}(n,Q)=\frak{S}(n)+O(Q^{-2-1/k}).\end{equation} Then, combining (\ref{kapa1}), (\ref{J}) and (\ref{S}) and the previous observations one derives the first statement of the proposition. In order to prove (\ref{parral}) we recall (\ref{fasi}) and observe that
\begin{equation}\label{rrrx}r_{s,k}(n,R)=\sum_{\substack{(x_{1},\ldots,x_{s})\in\mathcal{C}(n,R) }}(x_{1}\cdots x_{s})^{-1+k/s}=\sum_{\substack{(x_{1},\ldots,x_{s})\in\mathcal{C}(n,R) \\ \max\limits_{i\leq s} (x_{i})> P_{-} }}(x_{1}\cdots x_{s})^{-1+k/s},\end{equation} whence by an inclusion-exclusion argument it follows that
\begin{equation}\label{main}r_{s,k}(n,R)=\sum_{j=1}^{s}(-1)^{j+1}\binom{s}{j}r_{j}(n,R).\end{equation}
One then may deduce by (\ref{rsk}), (\ref{ooo}) and the preceding equation that
\begin{align*}r_{s,k}(n,R)=&\sum_{j=1}^{s}(-1)^{j+1}\binom{s}{j}\int_{\frak{K}}\tilde{f_{s}}(\alpha,P,R)^{j}f_{s}(\alpha,P,R)^{s-j}e(-\alpha n)d\alpha
\\
&+O\Big(\sum_{j=1}^{s}\int_{\frak{k}}\lvert \tilde{f_{s}}(\alpha,P,R)\rvert^{j} \lvert f_{s}(\alpha,P,R)\rvert^{s-j}d\alpha\Big),
\end{align*}
whence the first statement combined with Corollary \ref{cor1} delivers (\ref{parral}), as desired.
\end{proof}
We shall conclude our analysis in this section by showing that the analogous counting function with one of the underlying variables being significantly smaller is negligible. To such an end it seems convenient defining for any function $\varphi(x)$ of uniform growth 
\begin{equation}\label{phi}r_{s,k}^{\varphi}(n,R)=\sum_{\substack{(x_{1},\ldots,x_{s})\in\mathcal{C}(n,R)\\ x_{i}> (n\varphi(n)^{-1})^{1/k} }}(x_{1}\cdots x_{s})^{-1+k/s}.\end{equation}

\begin{prop}\label{prop7.2}
Let $s\geq 4k+1$ and let $\Delta_{s}^{*}<0$ be an admissible exponent for minor arcs. Suppose that $P^{\eta}\exp(-\eta(k^{-1}\log P)^{1/2})\leq R\leq P^{\eta}$. Then for each $N\leq n\leq 2N$ one has
$$r_{s,k}^{\varphi}(n,R)=c_{k,s}(\eta)\frak{S}(n)+O\big((\log n)^{-\nu/15}+\varphi(n)^{-\nu/4s}\big).$$
\end{prop}
\begin{proof}
We begin by setting $P_{0}=P\varphi(n)^{-1/k}$ and noting upon recalling (\ref{rrrx}) that
\begin{align}\label{kkk}r_{s,k}(n,R)-r_{s,k}^{\varphi}(n,R)&=\sum_{\substack{(x_{1},\ldots,x_{s})\in\mathcal{C}(n,R)\\  \min\limits_{i\leq s} (x_{i})\leq (n\varphi(n)^{-1})^{1/k}}}(x_{1}\cdots x_{s})^{-1+k/s}
\\
&\ll\Big\lvert\int_{0}^{1} \tilde{f_{s}}(\alpha,P,R) f_{s}(\alpha,P,R)^{s-2} f_{s}(\alpha,P_{0},R)e(-\alpha n) d\alpha\Big\rvert.\nonumber
\end{align} 
We recall (\ref{cho}) and observe that the integral in the above equation thereby equals
\begin{align*}\int_{\grM(\varphi(n)^{1/4s})} \tilde{f_{s}}(\alpha,P,R) f_{s}(\alpha,P,R)^{s-2} f_{s}(\alpha,P_{0},R)e(-\alpha n) d\alpha+O\big(I_{\grm(\varphi(n)^{1/4s}),1}^{1,1}(P,P_{0})\big).
\end{align*} 
We find it worth noting that $\text{meas}(\grM(\varphi(n)^{1/4s}))=O( \varphi(n)^{1/2s}P^{-k})$, whence by the trivial bound $f_{s}(\alpha,P_{0},R)\ll P^{k/s}\varphi(n)^{-1/s}$, it follows that
\begin{equation}\label{hhhh}\int_{\grM(\varphi(n)^{1/4s})} \lvert \tilde{f_{s}}(\alpha,P,R) f_{s}(\alpha,P,R)^{s-2} f_{s}(\alpha,P_{0},R)\rvert d\alpha\ll \varphi(n)^{\frac{1}{2s}}\varphi(n)^{-\frac{1}{s}}=\varphi(n)^{-\frac{1}{2s}}.\end{equation}
Combining the previous estimates with the application of Lemma \ref{lem8.10} delivers 
\begin{equation}\label{ki}\lvert r_{s,k}(n,R)-r_{s,k}^{\varphi}(n,R)\rvert\ll \varphi(n)^{-\nu/4s},\end{equation}
which in conjunction with Proposition \ref{prop7.1} yields the desired result.
\end{proof}
We are now prepared to evaluate an analogue of the above counting function, namely
\begin{equation}\label{xxx}r_{s,k,\eta}^{\varphi}(n)=\sum_{\substack{n=x_{1}^{k}+\ldots+x_{s}^{k}\\  x_{i}\in\mathcal{A}(x_{i},x_{i}^{\eta})\\  x_{i}> (n\varphi(n)^{-1})^{1/k} }}(x_{1}\cdots x_{s})^{-1+k/s}.\end{equation}

\begin{cor}\label{cor2}
Let $s\geq 4k+1$ and $\Delta_{s}^{*}<0$ be an admissible exponent for minor arcs. Let $\varphi$ be of uniform growth satisfying $2\varphi(x)\leq \exp((\log x)^{1/2})$. Then for $N\leq n\leq 2N$ one has
$$r_{s,k,\eta}^{\varphi}(n)=c_{k,s}(\eta)\frak{S}(n)+O\big((\log n)^{-\nu/15}+\varphi(n)^{-\nu/4s}\big).$$
\end{cor}
\begin{proof}
We first observe that one trivially has $r_{s,k,\eta}^{\varphi}(n)\leq r_{s,k}(n,P^{\eta}),$ where we employed the fact that $ x_{i}\leq n^{1/k}\leq P$ for each $i\leq s$, whence Proposition \ref{prop7.1} enables one to deduce
\begin{equation}\label{ppp}r_{s,k,\eta}^{\varphi}(n)-c_{k,s}(\eta)\frak{S}(n)\leq C_{1} (\log n)^{-\nu/15}\end{equation} for some constant $C_{1}>0$. On the other hand, upon recalling (\ref{phi}) it transpires that
\begin{align*}r_{s,k,\eta}^{\varphi}(n)\geq\sum_{\substack{(x_{1},\ldots,x_{s})\in\mathcal{C}(n,(n\varphi(n)^{-1})^{\eta/k})\\ x_{i}> (n\varphi(n)^{-1})^{1/k}  }}(x_{1}\cdots x_{s})^{-1+k/s}&\geq r_{s,k}^{\varphi}\big(n,P^{\eta}\exp\big(-\eta(k^{-1}\log P)^{1/2}\big)\big),
\end{align*}
where we employed the restriction on $\varphi$. Therefore, Proposition \ref{prop7.2} delivers
\begin{equation}\label{ppps}r_{s,k,\eta}^{\varphi}(n)-c_{k,s}(\eta)\frak{S}(n)\geq -C_{2}\big( (\log n)^{-\nu/15}+\varphi(n)^{-\nu/4s}\big)\end{equation} for some constant $C_{2}>0$. The corollary follows by combining both (\ref{ppp}) and (\ref{ppps}).
\end{proof}

\section{Unrepresentation evaluations}\label{sec8}
We shall explore in the present section the validity of the preceding asymptotic formulae if the condition $\Delta_{s}^{*}<0$ no longer holds. To such an end we recall (\ref{paar}), (\ref{nu0}), and introduce as is customary for a set $\frak{B}\subset [0,1)$, integers $1\leq j\leq s$ and $0\leq l\leq s-j$, some large $N>1$, a natural number $n\in[N,2N]$ and $1\leq P_{0}\leq P$ the auxiliary Fourier coefficient
\begin{equation*}I_{\frak{B},P_{0}}^{j,l}(n)=\int_{\frak{B}}\tilde{f_{s}}(\alpha,P,R)^{j} f_{s}(\alpha,P,R)^{s-j-l} f_{s}(\alpha,P_{0},R)^{l}e(-\alpha n) d\alpha.\end{equation*} 
\begin{lem}\label{lem8.1}
Let $s\geq 2k+1$ and $\Delta_{2s}^{*}<0$ be an admissible exponent for minor arcs. Let $0<\theta< k/2$ and $j,l$ as above. Then, for all but $O(N^{1-\nu\theta/2sk})$ integers $n\in [N,2N]$ one has
$$I_{\grm(P^{\theta}),P_{0}}^{j,l}(n)\ll N^{-\nu\theta/6sk}.$$ 
\end{lem}
\begin{proof}
A routinary application of Bessel's and Holder's inequality permits one to obtain
\begin{align*}&\sum_{n\in [N,2N]}\lvert I_{\grm(P^{\theta},P),P_{0}}^{j,l}(n)\rvert^{2}\ll \int_{\grm}\lvert\tilde{f_{s}}(\alpha,P,R)\rvert^{2j} \lvert f_{s}(\alpha,P,R)\rvert^{2(s-j-l)} \lvert f_{s}(\alpha,P_{0},R)\rvert^{2l}d\alpha\nonumber
\\
&\ll \Big(\int_{\grm}\lvert \tilde{f_{s}}(\alpha,P,R)\rvert^{2s}d\alpha\Big)^{j/s}\Big(\int_{\grm}\lvert f_{s}(\alpha,P,R)\rvert^{2s}d\alpha\Big)^{1-(j+l)/s}\Big( \int_{\grm}\lvert f_{s}(\alpha,P_{0},R)\rvert^{2s}\Big)^{l/s},\end{align*}
where $\grm=\grm(P^{\theta})$. We insert (\ref{oo}) in the previous estimate and apply Holder's inequality to deduce that the right side of the above equation is bounded above by a constant times
$$(\log P)^{2s-2}\max_{P_{1}\leq P}\Big(\int_{0}^{1}\lvert{g_{s}}(\alpha,P_{1},R)\rvert^{2s}d\alpha\Big)^{1-j/s}\max_{P_{2}\asymp P}\Big(\int_{\grm}\lvert{g_{s}}(\alpha,P_{2},R)\rvert^{2s}d\alpha\Big)^{j/s}.$$ Consequently, in virtue of the negativity of $\Delta_{2s}^{*}$ one may employ Proposition \ref{prop4} to bound both integrals in the preceding line and thus derive
\begin{align*}\sum_{n\in [N,2N]}\lvert I_{\grm(P^{\theta},P),P_{0}}^{j,l}(n)\rvert^{2}&\ll (\log P)^{2s-2}P^{k(s-j)/s}\max_{P_{2}\asymp P}\Big(\int_{\grm}\lvert{g_{s}}(\alpha,P_{2},R)\rvert^{2s}d\alpha\Big)^{\frac{j}{s}}\ll P^{k-\frac{\theta\nu}{s}+\varepsilon}.
\end{align*} The lemma then follows by a simple pidgeonhole argument in conjunction with (\ref{paar}).
\end{proof}
 
Having been furnished with the preceding estimate we employ the major arc analysis in Section \ref{sec7} to derive for almost all integers an asymptotic evaluation of $r_{s,k}(n,R)$.
\begin{prop}\label{prop8.2}
Let $s\geq 4k+2$ with $\Delta_{2s}^{*}<0$ being an admissible exponent for minor arcs. Suppose that $P^{\eta}\exp(-\eta(k^{-1}\log P)^{1/2})\leq R\leq P^{\eta}$. Then, there exists $\zeta_{k,s}>0$ for which for all but $O(N^{1-\zeta_{k,s}})$ integers $n\in [N,2N]$ one has
\begin{equation*}r_{s,k}(n,R)=c_{k,s}(\eta)\frak{S}(n)+O\big((\log n)^{-\nu }\big).\end{equation*} 

\end{prop}
\begin{proof}
Set $P_{0}=P$, recall (\ref{vv}), (\ref{rsk}) and note for each $1\leq j\leq s$ and $0<\theta<k/2$ that
\begin{equation}\label{ui}r_{j}(n,R)=I_{\grm(P^{\theta}),P}^{j,0}(n)+I_{\grM(P^{\theta})\setminus \frak{K},P}^{j,0}(n)+I_{\frak{K},P}^{j,0}(n).\end{equation} 
We use Lemma \ref{lem6.1} to note when $\alpha\in \grM(P^{\theta})$, $s\geq 4k+2$ and $0\leq l\leq \min(s-j,1)$ that 
$$\lvert\tilde{f}_{s}(\alpha,P,R)\rvert^{j} \lvert f_{s}(\alpha,P,R)\rvert^{s-l-j}\lvert f_{s}(\alpha,P_{0},R)\rvert^{l}\ll (\log P)^{3s-3l}P^{k}\Upsilon(\alpha)^{1+1/s}.$$
Combining the above estimate with the conclusion of Lemma \ref{kk} for $Q<P^{\theta}$ yields
\begin{align}\label{pal}I_{\grM(P^{\theta})\setminus \grM(Q),1}^{j,l}(P,P_{0})&\ll (\log P)^{3s-3l}P^{k}\int_{\grM(P^{\theta})\setminus \grM(Q)}\Upsilon(\alpha)^{1+\frac{1}{s}}d\alpha \ll (\log P)^{3s-3l}Q^{-\frac{1}{s}}.
\end{align}
If moreover $Q< (\log P)^{4s^{2}}$, we use Lemma \ref{lem5.2} to note when $\alpha\in\grM((\log P)^{4s^{2}})\setminus \grM(Q)$ that
\begin{equation*}\lvert \tilde{f}_{s}(\alpha,P,R)\rvert^{j}\lvert f_{s}(\alpha,P,R)\rvert^{s-l-j}\ll P^{k-kl/s}\Upsilon(\alpha)^{2-2/s-\varepsilon}.\end{equation*} We thus combine the preceding estimate and Lemma \ref{kk} to deduce for $Q< (\log P)^{4s^{2}}$ that
\begin{align}\label{appl}
I_{\grM(P^{\theta})\setminus \grM(Q),1}^{j,l}(P,P_{0})\ll (\log P)^{-s}+P^{k}\int_{\grM((\log P)^{4s^{2}})\setminus \grM(Q)}\Upsilon(\alpha)^{2-\frac{2}{s}-\varepsilon}d\alpha\ll Q^{-1/4s}.
\end{align}
By the above discussion and the fact that whenever $Q=(\log P)^{\frac{1}{8}} $ then $\grM(Q)\subset \frak{K}$ one has
\begin{equation*}I_{\grM(P^{\theta})\setminus \frak{K},P}^{j,0}(n)\ll I_{\grM(P^{\theta})\setminus \grM(Q),1}^{j,0}(P,P)\ll (\log n)^{-1/32s}.\end{equation*}

We next observe that the first equation in Proposition \ref{prop7.1} yields
\begin{equation*}\sum_{j=1}^{s}(-1)^{j+1}\binom{s}{j}I_{\frak{K},P}^{j,0}(n)=c_{k,s}(\eta)\frak{S}(n)+O((\log n)^{-1/16}).\end{equation*} We conclude by estimating $I_{\grm(P^{\theta}),P}^{j,0}(n)$ via Lemma \ref{lem8.1} and combine such a conclusion with the above lines and (\ref{ui}) to get for all but $O(N^{1-\zeta_{k,s}})$ integers $n\in [N,2N]$ the equation
$$\sum_{j=1}^{s}(-1)^{j+1}\binom{s}{j}r_{j}(n,R)=c_{k,s}(\eta)\frak{S}(n)+O((\log n)^{-\nu}).$$ The result follows by the preceding line and the computation leading to (\ref{main}).
\end{proof}
We shall next deduce upon recalling (\ref{phi}) the analogue of Proposition \ref{prop7.2}. 
\begin{prop}\label{prop8.3}
Let $s,\Delta_{2s}^{*}$ be as in Proposition \ref{prop8.2} and $\varphi(x)\leq\exp((\log x)^{1/2})/2$ be of uniform growth. Let $P^{\eta}\exp(-\eta(k^{-1}\log P)^{1/2})\leq R\leq P^{\eta}$. For all but $O(N^{1-\zeta_{k,s}})$ integers $n\in [N,2N]$ then
\begin{equation*}r_{s,k}^{\varphi}(n,R)=c_{k,s}(\eta)\frak{S}(n)+O\big((\log n)^{-\nu}+\varphi(n)^{-1/16s^{2}}\big).\end{equation*} 
\end{prop}
\begin{proof}
We recall (\ref{kkk}) and write as in the previous proposition
\begin{equation*}\lvert r_{s,k}(n,R)-r_{s,k}^{\varphi}(n,R)\rvert\ll \big\lvert I_{\grm(P^{\theta}),P_{0}}^{1,1}(n)+I_{\grM(P^{\theta})\setminus \grM(Q),P_{0}}^{1,1}(n)+I_{\grM(Q),P_{0}}^{1,1}(n)\big\rvert,\end{equation*} where $Q=\varphi(n)^{\frac{1}{4s}}$ and $P_{0}=P\varphi(n)^{-\frac{1}{k}}$. The triangle inequality combined with (\ref{pal}) when $Q\geq (\log P)^{4s^{2}}$ and (\ref{appl}) if $Q< (\log P)^{4s^{2}}$ then yields
$$I_{\grM(P^{\theta})\setminus \grM(Q),P_{0}}^{1,1}(n)\ll I_{\grM(P^{\theta})\setminus \grM(Q),1}^{1,1}(P,P_{0})\ll \varphi(n)^{-1/16s^{2}}.$$
Therefore, the above lines in conjunction with Lemma \ref{lem8.1} and (\ref{hhhh}) permits one to deduce 
$$\lvert r_{s,k}(n,R)-r_{s,k}^{\varphi}(n,R)\rvert\ll \varphi(n)^{-1/16s^{2}}+\varphi(n)^{-1/2s}+N^{-\zeta_{k,s}}$$ for all but $O(N^{1-\zeta_{k,s}})$ integers, which combined with Proposition \ref{prop8.2} completes the proof.
\end{proof}

\begin{cor}\label{cor4}
Let $k,s,\varphi$ be as in Proposition \ref{prop8.3}. Then, for all but $O(N^{1-\zeta_{k,s}})$ integers $n\in [N,2N]$ one has
\begin{equation*}r_{s,k,\eta}^{\varphi}(n)=c_{k,s}(\eta)\frak{S}(n)+O\big((\log n)^{-\nu/15}+\varphi(n)^{-\nu/4s}\big).\end{equation*} 
\end{cor}
\begin{proof}
The desired result follows by employing the argument of Corollary \ref{cor2} replacing the use of Propositions \ref{prop7.1} and \ref{prop7.2} therein by their counterparts Propositions \ref{prop8.2} and \ref{prop8.3}. 
\end{proof}

\section{Almost all estimates for families of weighted representation functions}\label{sec9}
As foreshadow in the introduction, it is the analysis which we shall perform in the upcoming section which ultimately impairs the restriction on the number of variables in Theorem \ref{thm1.3}. It is then opportune to introduce for fixed $d\in\mathbb{N}$, fixed vector $\bfa\in [1,s]^{s-d}$ and every natural number $m\in [N,2N]$ the Fourier coefficient
\begin{equation}\label{Fda}F_{d,\bfa}(m)=\int_{0}^{1} \tilde{f}_{s}(a_{1}\alpha,P,R)\prod_{j=2}^{s-d}f_{s}(a_{j}\alpha,P,R)e(-\alpha m)d\alpha,\end{equation}
wherein we take $P$ as in (\ref{paar}). We also recall (\ref{eq67}) and fix any positive real number $D=D(k)\geq 1$ which shall be determined later satisfying the bound
\begin{equation}\label{tau}\tau(k)\geq (Dk)^{-1}.\end{equation} We also introduce for $s\geq 2$ and any $d\in\mathbb{N}$ the parameter 
\begin{equation}\label{Td}T_{d}(k)=\frac{s^{3}d}{2c_{k}k^{3}},\ \ \ \ \ \ \ \ \ \ \ \  c_{k}=\frac{39d}{40}\Big(1-\frac{(d+1)s}{k^{2}}-\frac{d}{k}\Big)\frac{s}{\frac{2s}{D}+dk},\end{equation} and for every positive integer $T_{0}\leq T_{d}(k)$ the constant
\begin{equation}\label{sto}s_{T_{0}}=2\Big\lfloor \frac{s-d-T_{0}}{2}\Big\rfloor.\end{equation} 
\begin{prop}\label{prop911}
Let $k\geq 100$ and $\max(15/2,\sqrt{D}/2)k\leq s\leq Dk^{2}$. Let $d\in\mathbb{N}$ with $d\leq \frac{s}{Dk}+\frac{k}{4s}$ and $\bfa\in [1,s]^{s-d}$. Assume that there is some natural number $s_{0}>s$ with the property that there is an admissible exponent for minor arcs $\Delta_{s_{0}}^{*}<0$, that $T_{d}(k)\leq 3k/4$, that $2(s-d-k)>s_{0}$ with $s_{0}-1-k\leq s<s_{0}$ and that $D\leq 4c_{k}$. Suppose that for every $T_{0}\leq T_{d}(k)$ then $\Delta_{s_{T_{0}}}$ is an admissible exponent satisfying \begin{equation}\label{deelta}0\leq \Delta_{s_{T_{0}}}\leq \frac{1}{D}\Big(1+\frac{T_{0}+d+1}{k}+\Big(\frac{T_{0}+d+1}{k}\Big)^{2}\Big).\end{equation}
Then for all but $O(N^{1-d/k-1/240s})$ integers $m\in [N,2N]$ one has
\begin{equation}\label{primo}F_{d,\bfa}(m)\ll m^{-k/4s^{3}}.\end{equation} 
\end{prop}
\begin{proof}
In order to derive the best possible estimates for the exceptional set it is appropiate to consider the dyadic dissection underlying (\ref{oo}), making a distinction between the corresponding sizes of the underlying variables being a desideratum to the end of applying the pertinent approach accordingly. We thereby recall (\ref{paar}), denote for each $j\in\mathbb{N}$ the parameter $P_{j}=2^{-j}P$ and introduce for each integer $0\leq T_{0}\leq s-d$ the sets
$$\mathcal{J}_{T_{0}}(P)=\Big\{\bfj\in \Big[0,\Big\lfloor \frac{\log P}{\log 2}\Big\rfloor\Big]^{s-d}:\  P_{j_{l}}\geq P^{1-c(\frac{k}{s})^{2}},\   l> T_{0};\ \ \ P_{j_{l}}\leq P^{1-c(\frac{k}{s})^{2}}, \ 0\leq l\leq T_{0} \Big\}$$ with $c=c_{k}$ and $P_{j_{0}}=1$. We take for each $\bfj\in \mathcal{J}_{T_{0}}(P)$ and $\mathfrak{B}\subset [0,1)$ the Fourier coefficient 
$$I_{\bfj}(m,\mathfrak{B})=\int_{\frak{B}}\prod_{l=1}^{s-d}g_{s}(a_{l}\alpha, P_{j_{l}},R)e(-\alpha m)d\alpha,$$
wherein we dropped the dependence on $\bfa$ for brevity, employ the dissection (\ref{oo}) and write
\begin{equation}\label{Fdm}F_{d,\bfa}(m)\ll B(m)+C(m),\end{equation} where
\begin{equation*}B(m)=\sum_{T_{0}\leq T_{d}(k)}\sum_{\substack{\bfj\in\mathcal{J}_{T_{0}}(P)}}\big\lvert I_{\bfj}(m,[0,1))\big\rvert,\ \ \ \ \ C(m)=\sum_{T_{0}> T_{d}(k)}\sum_{\substack{\bfj\in\mathcal{J}_{T_{0}}(P)}}\big\lvert I_{\bfj}(m,[0,1))\big\rvert.\end{equation*}

We shall examine first $B(m)$. We make a distinction between major and minor arcs suitably chosen and write $B(m)=B_{1}(m)+B_{2}(m),$ where
$$B_{j}(m)=\sum_{\substack{T_{0}\leq T_{d}(k)\\ \bfj\in\mathcal{J}_{T_{0}}(P)}}\big\lvert  I_{\bfj}\big(m,\grN_{j}\big)\big\rvert,\  j=1,2,\ \ \ \ \ \  \ \ \ \grN_{2}= \grM\big(P^{\frac{k^{2}Dd}{2s}(1-V(T_{0})k/s)},P^{1-c(k/s)^{2}}) $$
and $\grN_{1}=[0,1)\setminus \grN_{2}$ for $V(T_{0})$, that will be shown to satisfy $V(T_{0})<s/k,$ being defined as
\begin{equation}\label{VV}V(T_{0})=\frac{(d+1+T_{0})s}{k^{2}}+c\Big(\frac{2}{Dd}+\frac{k}{s}\Big).\end{equation} 
Indeed, employing the constraint $T_{d}(k)\leq 3k/4$ and the bound $c_{k}< Dd/2$ one gets
\begin{equation}\label{frac}s^{3}d\leq \frac{3}{2}k^{4}c\leq \frac{3Ddk^{4}}{4},\end{equation} from where it follows since $s\geq 15k/2$ that $s^{2}/Dk\leq   k^{2}/10.$
We then utilise the aforementioned collection of ingredients and the bounds $k\geq 100$ and $s\geq 15k/2$ to derive 
\begin{align*}(d+1+T_{0})s+ck^{2}\Big(\frac{2}{dD}+\frac{k}{s}\Big)&\leq \frac{s^{2}}{Dk}+\frac{k}{4}+\frac{304sk}{400}+k^{2}\leq \frac{s^{2}}{Dk}+\frac{304sk}{400}+\frac{401k^{2}}{400}
\\
&\leq \frac{441k^{2}}{400}+\frac{304sk}{400}\leq \Big(\frac{147}{1000}+\frac{304}{400}\Big)sk<sk.
\end{align*}
Note that in the first step we used the bound $T_{0}\leq 3k/4$, that pertaining to $d$ in the statement and the estimates $c_{k}< (\frac{2}{Dd}+\frac{k}{s})^{-1}$ and $k\geq 100$, and on the third step we employed $s^{2}/Dk\leq k^{2}/10$. The condition $V(T_{0})<s/k$ then follows accordingly. In view of the bound $c_{k}< Dd/2$ and the restrictions on $d$ and $s$ we also observe that 
$$c_{k}< \frac{s}{2k}+\frac{Dk}{8s}\leq \frac{1}{2}\Big(\frac{s}{k}\Big)^{2}+\frac{D}{8}\leq \Big(\frac{s}{k}\Big)^{2}.$$

We shall next focus on the minor arc contribution and write $V=V(T_{0})$. A routinary application of Bessel's and Holder's inequality permits one to obtain
\begin{align}\label{B1}&\sum_{m\in[N,2N]}\lvert B_{1}(m) \rvert^{2}\ll (\log P)^{s-d}\sum_{\substack{T_{0}\leq T_{d}(k)\\ \bfj\in\mathcal{J}_{T_{0}}(P)}}\sum_{m\in[N,2N]}\big\lvert  I_{\bfj}\big(m,\grm\big(P^{\frac{k^{2}Dd}{2s}(1-Vk/s)},P^{1-c(k/s)^{2}})\big)\big\rvert^{2}\nonumber
\\
&\ll (\log P)^{s-d}\sum_{\substack{T_{0}\leq T_{d}(k)}}\sum_{\bfj\in\mathcal{J}_{T_{0}}(P)}\int_{\grm\big(P^{\frac{k^{2}Dd}{2s}(1-Vk/s)},P^{1-c(k/s)^{2}}\big)}\prod_{l=1}^{s-d}\lvert g_{s}(a_{l}\alpha, P_{j_{l}},R)\rvert^{2}d\alpha. 
\end{align}
We then insert the trivial bound \begin{equation}\label{triv}g_{s}(a_{l}\alpha, P_{j_{l}},R)\ll P_{j_{l}}^{k/s}\end{equation} on the factors with $l\leq T_{0}$ to the end of estimating the above integral by a constant times
\begin{equation}\label{pia}\sum_{T_{0}\leq T_{d}(k)}\sum_{\substack{j_{l}\leq \frac{\log P}{\log 2}\\ P_{j_{l}}\geq P_{k,s},\  l> T_{0}}}(P^{1-\frac{ck^{2}}{s^{2}}})^{\frac{2kT_{0}}{s}}\int_{\grm\big(P^{\frac{k^{2}Dd}{2s}(1-\frac{Vk}{s})},P^{1-\frac{ck^{2}}{s^{2}}}\big)}\prod_{l=T_{0}+1}^{s-d}\lvert g_{s}(a_{l}\alpha, P_{j_{l}},R)\rvert^{2}d\alpha
\end{equation}
with $P_{k,s}=P^{1-ck^{2}/s^{2}}$. We next observe in view of $P_{k,s}\leq P_{j_{l}}\leq P$ for each $l> T_{0}$ that \begin{equation}\label{grm}\grm\big(P^{\frac{k^{2}Dd}{2s}(1-k/s)},P^{1-c(k/s)^{2}}\big)\subset \grm\big(P^{\frac{k^{2}Dd}{2s}(1-k/s)},P_{j_{l}}\big).\end{equation} It may also be opportune to assume first for each $T_{0}$ that
$d<(s^{2}-ck^{2})/Dk(s-Vk)$ and emphasize that whenever $l\geq T_{0}$ then in light of the above restriction it transpires that
\begin{equation}\label{psps}P_{j_{l}}^{\frac{k^{2}Dd}{2s}(1-Vk/s)}\leq P^{\frac{k^{2}Dd}{2s}(1-Vk/s)}\leq P_{j_{l}}^{\frac{k^{2}Dd(s-Vk)}{2(s^{2}-ck^{2})}}\leq P_{j_{l}}^{k/2}.\end{equation}
We further note for prompt convenience that by hypothesis one has \begin{equation}\label{99}2(s-d-T_{0})\geq 2(s-d-T_{d}(k))\geq 2(s-d-k)>s_{0}>s.\end{equation}
We draw the reader's attention back to (\ref{eq8}) to note in view of the provisos $ s\leq Dk^{2}$ and (\ref{tau}) and the preceding conclusion that then
\begin{equation}\label{deltas}\Delta_{2(s-d-T_{0})}^{*}\leq  \Delta_{s_{0}}^{*}-\big(2(s-d-T_{0})-s_{0})\tau(k)<-\big(2(s-d-T_{0})-s_{0})\tau(k).\end{equation}
In view of equations (\ref{grm}), (\ref{psps}) and (\ref{deltas}) we have reached a position from which to apply Proposition \ref{prop2222} to the integral in (\ref{pia}). We thus denote $\grm_{0}=\grm\big(P^{\frac{k^{2}Dd}{2s}(1-\frac{Vk}{s})},P^{1-c(\frac{k}{s})^{2}}\big)$ for simplicity and get
\begin{align*}\int_{\grm_{0}}\prod_{l=T_{0}+1}^{s-d}\lvert g_{s}(a_{l}\alpha, P_{j_{l}},R)\rvert^{2}d\alpha&\ll \max_{l\geq T_{0}}\int_{\grm\big(P^{\frac{k^{2}Dd}{2s}(1-\frac{Vk}{s})},P_{j_{l}}\big)}\lvert g_{s}(a_{l}\alpha, P_{j_{l}},R)\rvert^{2(s-d-T_{0})}d\alpha
\\
&\ll\max_{l\geq T_{0}}P_{j_{l}}^{2(s-d-T_{0})k/s-k}P^{-\frac{dDk}{s}(1-\frac{Vk}{s})\tau(k)(2(s-d-T_{0})-s_{0})+\varepsilon}
\\
&\ll P^{2(s-d-T_{0})k/s-k-\frac{dDk}{s}(1-\frac{Vk}{s})\tau(k)(2(s-d-T_{0})-s_{0})+\varepsilon},
\end{align*}the last step being a consequence of (\ref{99}). We employ (\ref{tau}) and insert the preceding estimate in (\ref{B1}) and (\ref{pia}) to deduce that
\begin{equation}\label{B1B1}\sum_{m\in[N,2N]}\lvert B_{1}(m) \rvert^{2}\ll \sum_{T_{0}\leq T_{d}(k)}P^{k-\frac{2kd}{s}+\varepsilon-\frac{d}{s}(1-Vk/s)(2(s-d)-s_{0})+T_{0}\big(\frac{2d}{s}(1-Vk/s)-2ck^{3}/s^{3}\big)}.
\end{equation}

We examine the exponent $\alpha(T_{0})$ in the above equation, the proviso $s_{0}-k-1\leq s$ entailing \begin{align*}\alpha(T_{0})&\leq k-\frac{2kd}{s}-d\big(1-\frac{Vk}{s}\big)(1-2d/s-k/s-1/s)+T_{0}\Big(\frac{2d}{s}\big(1-\frac{Vk}{s}\big)-\frac{2ck^{3}}{s^{3}}\Big),
\end{align*}
and hence $\alpha(T_{0})\leq k-d+k\beta_{k}(T_{0})/s,$ wherein $$\beta_{k}(T_{0})=-d+dV\Big(1-\frac{k}{s}\Big)+\frac{d(2d+1)}{k}+\frac{2T_{0}d}{k}-\frac{2ck^{2}T_{0}}{s^{2}}.$$ 
In order to proceed in the proof we insert (\ref{VV}) into the line defining $\beta_{k}(T_{0})$ to obtain
\begin{equation}\label{jdjd}\beta_{k}(T_{0})\leq T_{0}\Big(\frac{ds}{k^{2}}-\frac{2k^{2}c}{s^{2}}+\frac{d}{k}\Big)+\frac{d(d+1)s}{k^{2}}+\frac{d^{2}}{k}-d+\frac{2c}{D}+\frac{cdk}{s}.\end{equation}
It is then worth noting that the aforementioned choice for $c$ delivers the identity 
$$\frac{d(d+1)s}{k^{2}}+\frac{d^{2}}{k}-d+\frac{2c}{D}+\frac{cdk}{s}=-\frac{d}{40}\Big(1-\frac{(d+1)s}{k^{2}}-\frac{d}{k}\Big).$$ We also remark that (\ref{frac}) in conjunction with the condition $s\geq 15k/2$ delivers $$\frac{s^{2}d}{2k^{3}}\Big(\frac{s}{k}+1\Big)\leq \frac{2s^{3}d}{3k^{4}}\leq c.$$ The above inequality then entails that the coefficient of $T_{0}$ in (\ref{jdjd}) is not positive. Therefore, by employing the preceding estimates and noting that the bound $D\leq 4c_{k}$ entails 
$$\frac{1}{39}\Big(\frac{1}{2}+\frac{Ddk}{4s}\Big)\leq \frac{d}{40}\Big(1-\frac{(d+1)s}{k^{2}}-\frac{d}{k}\Big),$$
 we obtain $\alpha(T_{0})< k-d-k/78s.$ We then combine the above lines with (\ref{B1B1}) to get
\begin{equation}\label{bb1}\sum_{m\in[N,2N]}\lvert B_{1}(m) \rvert^{2}\ll P^{k-d-k/78s}.\end{equation}
Observe that the above discussion is correct subject to the validity of the assumption on $d$ right above (\ref{psps}), deducing that the upper bound therein is superior to that in the statement of the proposition being a desideratum. Indeed, by noting in view of the restriction on $d$ and $s\geq \sqrt{D}k/2$, these in turn entailing
$d\leq  2s/kD,$ that the choice in (\ref{VV}) combined with the above bound yields $V\geq 2ck/s$, one gets when applying the latter in conjunction with the assumption $D\leq 4c_{k}$ that
\begin{equation*}\frac{s^{2}-ck^{2}}{Dk(s-Vk)}=\frac{s}{Dk}+\frac{sV-ck}{D(s-Vk)}\geq \frac{s}{Dk}+\frac{ck}{Ds}\geq  \frac{s}{Dk}+\frac{k}{4s}.\end{equation*}

In order to proceed we shall estimate $B_{2}(m)$ appropiately. We thus start by noting that
\begin{align}\label{B2B2}&\lvert B_{2}(m)\rvert\ll \sum_{T_{0}\leq T_{d}(k)} \sum_{\substack{\bfj\in\mathcal{J}_{T_{0}}(P)}}\int_{\grM\big(P^{\frac{k^{2}Dd}{2s}(1-\frac{Vk}{s})},P^{1-c(k/s)^{2}}\big)}\prod_{l=1}^{s-d}\lvert g_{s}(a_{l}\alpha, P_{j_{l}},R)\rvert d\alpha
\\
&\ll \sum_{T_{0}\leq T_{d}(k)} (P^{1-ck^{2}/s^{2}})^{\frac{kT_{0}}{s}}\sum_{\substack{j_{T_{0}+1},\ldots, j_{s-d}\\ P_{j_{l}}\geq P^{1-ck^{2}/s^{2}}}}\int_{\grM\big(P^{\frac{k^{2}Dd}{2s}(1-\frac{Vk}{s})},P^{1-\frac{ck^{2}}{s^{2}}}\big)}\prod_{l=T_{0}+1}^{s-d}\lvert g_{s}(a_{l}\alpha, P_{j_{l}},R)\rvert d\alpha,\nonumber
\end{align}
where in the last step we inserted the trivial bound (\ref{triv}) on the factors with $l\leq T_{0}$. We then observe as is customary that whenever $P^{1-c(k/s)^{2}}\leq P_{j_{l}}\leq P$ then one has
$$\grM\big(P^{\frac{k^{2}Dd}{2s}(1-Vk/s)},P^{1-c(k/s)^{2}}\big)\subset \grM\big(P^{\frac{k^{2}Dd}{2s}(1-Vk/s)+ck^{3}/s^{2}},P_{j_{l}}\big).$$
Equipped with the above relation it seems worth denoting $\grM_{0}=\grM\big(P^{\frac{k^{2}Dd}{2s}(1-Vk/s)},P^{1-c(\frac{k}{s})^{2}})$, recalling (\ref{sto}) and alluding to both Proposition \ref{prop2.1} and Lemma \ref{lem1} in order to derive 
\begin{align*}\int_{\grM_{0}}\prod_{l=T_{0}+1}^{s-d}\lvert g_{s}(a_{l}\alpha, P_{j_{l}},R)\rvert d\alpha&\ll \max_{l\geq T_{0}}\int_{ \grM\big(P^{\frac{k^{2}Dd}{2s}(1-Vk/s)+ck^{3}/s^{2}},P_{j_{l}}\big)}\lvert g_{s}(a_{l}\alpha, P_{j_{l}},R)\rvert^{s-d-T_{0}}d\alpha
\\
&\ll\max_{l\geq T_{0}}P_{j_{l}}^{(s-d-T_{0})k/s-k}P^{\frac{dDk}{s}(1-\frac{Vk}{s})\Delta_{s_{T_{0}}}+2c\Delta_{s_{T_{0}}}k^{2}/s^{2}+\varepsilon}
\\
&\ll P^{(d+T_{0})(\frac{ck^{3}}{s^{3}}-k/s)+\frac{dDk}{s}(1-\frac{Vk}{s})\Delta_{s_{T_{0}}}+2c\Delta_{s_{T_{0}}}k^{2}/s^{2}+\varepsilon},
\end{align*}
the latter being employed in view of the inequality $\Delta_{s_{T_{0}}}\geq 0$ when the height of the major arcs exceeds $P_{j_{l}}^{k/2}$. Note that in the last step we employed the fact that $P^{1-c(k/s)^{2}}\leq P_{j_{l}}$ for every $l\geq T_{0}$. Inserting the preceding estimate into (\ref{B2B2}) then yields the bound
\begin{equation}\label{pra}\lvert B_{2}(m)\rvert\ll \sum_{T_{0}\leq T_{d}(k)} P^{-dk/s+ck^{3}d/s^{3}+\frac{dDk}{s}(1-\frac{Vk}{s})\Delta_{s_{T_{0}}}+2c\Delta_{s_{T_{0}}}k^{2}/s^{2}+\varepsilon}.\end{equation}
We denote by $\gamma(T_{0})$ to the above exponent and note that the assumption (\ref{deelta}) yields
\begin{align*}&\gamma(T_{0})-cdk^{3}/s^{3}\leq  -dk/s+\Big(\frac{dk}{s}\Big(1-\frac{Vk}{s}\Big)+\frac{2ck^{2}}{Ds^{2}}\Big)\Big(1+\frac{T_{0}+d+1}{k}+\Big(\frac{T_{0}+d+1}{k}\Big)^{2}\Big)\nonumber
\\
&\leq \frac{k^{2}}{s^{2}}\Big(\frac{2c}{D}-dV\Big)\Big(1+\frac{T_{0}+d+1}{k}+\Big(\frac{T_{0}+d+1}{k}\Big)^{2}\Big)+\frac{dT_{0}}{s}+\frac{d(d+1)}{s}+\frac{d(T_{0}+d+1)^{2}}{ks}.
\end{align*}
We recall the choice (\ref{VV}) to get $\gamma(T_{0})<  -k^{2}d^{2}c/s^{3},$ which combined with (\ref{pra}) delivers
\begin{equation}\label{B2m}\lvert B_{2}(m)\rvert\ll P^{-k^{2}d^{2}c/s^{3}}.\end{equation} 

We conclude the proof by analysing the contribution flowing from $C(m)$. We first apply in a routinary manner as above Bessel's and Holder's inequalities subsequently to obtain
\begin{align*}\sum_{m\in[N,2N]}\lvert C(m) \rvert^{2}&\ll (\log P)^{s-d}\sum_{m\in[N,2N]}\sum_{T_{0}> T_{d}(k)}\sum_{\substack{\bfj\in\mathcal{J}_{T_{0}}(P)}}\big\lvert  I_{\bfj}(m,[0,1))\big\rvert^{2}\nonumber
\\
&\ll (\log P)^{s-d}\sum_{T_{0}> T_{d}(k)}\sum_{\substack{\bfj\in\mathcal{J}_{T_{0}}(P)}}\int_{0}^{1}\prod_{l=1}^{s-d}\lvert g_{s}(a_{l}\alpha, P_{j_{l}},R)\rvert^{2}d\alpha. 
\end{align*}
We utilise (\ref{triv}) to estimate trivially the factors in the above line with $l\leq T_{d}(k)$ and get
\begin{align}\label{BB1} \sum_{m\in[N,2N]}\lvert& C(m) \rvert^{2}\ll P^{\varepsilon}\sum_{T_{0}> T_{d}(k)}\sum_{\substack{j_{T_{0}+1},\ldots, j_{s-d}\\ P_{j_{l}}\geq P^{1-c(k/s)^{2}}}}(P^{1-c(\frac{k}{s})^{2}})^{\frac{2kT_{d}(k)}{s}}\int_{0}^{1}\prod_{l=T_{d}(k)+1}^{s-d}\lvert g_{s}(a_{l}\alpha, P_{j_{l}},R)\rvert^{2}d\alpha\nonumber
\\
&\ll  P^{\varepsilon}(P^{1-c(k/s)^{2}})^{\frac{2kT_{d}(k)}{s}}\max_{P_{j_{l}}\leq P}\int_{0}^{1}\lvert g_{s}(\alpha, P_{j_{l}},R)\rvert^{2(s-T_{d}(k)-d)}d\alpha,
\end{align}
where in the last step we employed orthogonality to eliminate the coefficient $a_{l}$. Equipped with the inequalities in (\ref{99}), the trivial bound (\ref{triv}) and the restriction $P_{j_{l}}\leq P$ one may estimate the integral in the preceding equation via Proposition \ref{prop4}, namely
\begin{align*}\int_{0}^{1}\lvert g_{s}(\alpha, P_{j_{l}},R)\rvert^{2(s-T_{d}(k)-d)}d\alpha&\ll (P_{j_{l}}^{k/s})^{2(s-T_{d}(k)-d)-s_{0}}\int_{0}^{1}\lvert g_{s}(\alpha, P_{j_{l}},R)\rvert^{s_{0}}d\alpha 
\\
&\ll P_{j_{l}}^{2(s-T_{d}(k)-d)\frac{k}{s}-k}\ll P^{2(s-T_{d}(k)-d)\frac{k}{s}-k},
\end{align*}
wherein we used the positivity of the corresponding exponents, it in turn stemming from (\ref{99}). We then recall (\ref{Td}) and insert the above estimate in (\ref{BB1}) to derive
\begin{equation*}\sum_{m\in[N,2N]}\lvert C(m) \rvert^{2}\ll P^{k-d-2dk/s+\varepsilon}.
\end{equation*}

We thus combine the above line with (\ref{Fdm}), (\ref{bb1}) and (\ref{B2m}) to get for $m\in [N,2N]$ that 
\begin{equation*}F_{d,\bfa}(m)\ll \lvert B_{1}(m)\rvert+\lvert C(m)\rvert+N^{- kd^{2}c/s^{3}},\end{equation*} wherein by a routine pidgeonhole argument one has for all but $O(N^{1-d/k-1/240s})$ integers $m\in [N,2N]$ that $\lvert B_{1}(m)\rvert=O( N^{-1/240s})$ and for all but $O(N^{1-d/k-d/s})$ integers $m\in [N,2N]$ that $\lvert C(m)\rvert=O( N^{-d/3s}).$ The proposition follows by the preceding discussion. 
\end{proof}

The upcoming lemma shall instead encompass the instance in which $d$ exceeds the threshold presented above. For such purposes we assume for some integer $1\leq d_{0}\leq s-2$ that $\delta_{d_{0}}=\Delta_{2\lfloor \frac{s-d_{0}}{2}\rfloor}$ is an admissible exponent and denote
$\tau_{d_{0}}=\frac{1}{2k}\lvert \delta_{d_{0}}-d_{0}k/s\rvert /(s-d_{0}).$
\begin{lem}\label{lem9.1111}Let $d_{0}\in\NN$ and $s\geq  d_{0}+2$ such that $s\delta_{d_{0}}<d_{0}k$. Suppose that $d\in\mathbb{N}$ with $d_{0}\leq d\leq s-1$ and let $\bfa\in [1,s]^{s-d}$. Then for every integer $m\in [N,2N]$ one has
$$F_{d,\bfa}(m)\ll m^{-\tau_{d_{0}}}.$$
\end{lem}
\begin{proof}
We allude to the dyadic dissections (\ref{oo}) and observe that inserting those in (\ref{Fda}) in conjunction with an application of Holder's inequality and a change of variables delivers
\begin{align*}F_{d,\bfa}(m)&\ll  \sum_{\substack{\bfi\in[0,\tilde{i}_{s}]\times \big[ \big\lfloor \frac{\log P}{\log 2}\big\rfloor\big]^{s-d-1}}}\prod_{l=1}^{s-d}\Big(\int_{0}^{1}\lvert g_{s}(\alpha,2^{-i_{l}}P,R)\rvert^{s-d_{0}}d\alpha\Big)^{1/(s-d_{0})},
\end{align*}
it being worth recalling (\ref{i0}) and (\ref{paar}) and remarking that $d\geq d_{0}$. We conclude by noting that $2^{-i_{1}}P\asymp P$, invoking Lemma \ref{lem1} and employing the proviso $s\delta_{d_{0}}<d_{0}k$ to get
\begin{align*}F_{d,\bfa}(m)&\ll P^{\varepsilon}(P^{\delta_{d_{0}}-\frac{d_{0}k}{s}})^{\frac{1}{s-d_{0}}}\max_{P_{1}\leq P}(P_{1}^{\delta_{d_{0}}-\frac{d_{0}k}{s}})^{\frac{s-d-1}{s-d_{0}}}\ll (m^{\delta_{d_{0}}-d_{0}k/s+\varepsilon})^{\frac{1}{k(s-d_{0})}}\ll m^{-\tau_{d_{0}}}.
\end{align*}
\end{proof}
We note for futher convenience that the above lemma yields for $d_{0}\leq d\leq s-1$ the existence of a constant $K_{2}>0$ for which for every $m\in\NN$ then 
\begin{equation}\label{begi}F_{d,\bfa}(m)\leq K_{2}m^{-\tau_{d_{0}}}.\end{equation} Let $K_{1}$ be the implicit constant underlying (\ref{primo}), and write $\tilde{K}=\max_{1\leq i\leq 3}K_{i},$ the constant $K_{3}$ being defined later on right before (\ref{deac}). We introduce for $n\in\mathbb{N}$, tuples $\bfx\in [1,n^{1/k}]^{d}$ and $\bfb\in [1,s]^{d}$ the function $Y_{\bfb}(n,\bfx)=n-b_{1}x_{1}^{k}-\ldots-b_{d}x_{d}^{k}.$ We also recall (\ref{nu0}) and write \begin{equation}\label{tau0}\nu_{0}=\frac{2\tau_{0}}{s},\ \ \ \ \ \ \ \ \ \ \ \ \ \ \ \ \tau_{0}=\min\big(\nu/16,s\min_{1\leq d_{0}\leq s-1}\tau_{d_{0}}/2,k/8s^{2}\big),\ \ \ \ \ \ \end{equation} and, whenever $\bfa\in [1,s]^{s-d}$ with $1\leq d\leq s-1$ the sets
\begin{equation}\label{eqa}\widetilde{\mathcal{Z}}_{d,\bfa}(N)=\Big\{n\in [1,2N]:\ \  F_{d,\bfa}(n)> \tilde{K}n^{-\nu_{0}} \Big\}.\end{equation} For each $ l\leq s-1$ we further consider $\mathcal{B}_{l}= [1,s]^{s-l}\times [1,s]^{d}\times [1,n^{1/k}]^{d}\cap\NN^{s+2d-l}$ and \begin{equation}\label{hhj}\mathcal{Z}_{d}(N)=\bigcup_{l=d}^{s-1}\Big\{n\in [1,2N]: \exists (\bfa,\bfb,\bfx)\in \mathcal{B}_{d}: Y_{\bfb}(n,\bfx)=m>0,   F_{l,\bfa}(m)>\tilde{K} m^{-\nu_{0}}\Big\}.\end{equation}
Likewise, we recall (\ref{nu0}), (\ref{cks}) and (\ref{xxx}), set $\upsilon_{0}=\min(\nu/30,\nu/8s)$ and define 
\begin{equation}\label{hhh}\mathcal{Z}_{0}(N)=\bigcup_{r\in\{r_{s,k},r_{s,k,\eta}^{\varphi}\}}\Big\{n\in [1,2N]: \lvert r(n,R)-c_{k,s}(\eta)\frak{S}(n)\rvert\geq (\log n)^{-\upsilon_{0}}+\varphi(n)^{-\upsilon_{0}}\Big\},\end{equation} where in the above line we are taking integers $n$ for which the preceding inequality holds for $r(n,R)$ being either $r_{s,k}(n,R)$ or $r_{s,k,\eta}^{\varphi}(n,R)$. In what follows we shall write
\begin{equation}\label{sas}\widetilde{\mathcal{Z}}(N)= \bigcup_{d=1}^{s-1}\bigcup_{\bfa\in [1,s]^{s-d} }\widetilde{\mathcal{Z}}_{d,\bfa}(N)  \   \ \ \ \ \text{and} \ \ \ \ \ \mathcal{Z}(N)=\widetilde{\mathcal{Z}}(N)\cup\bigcup_{d=0}^{s-1} \mathcal{Z}_{d}(N).\end{equation}
 It then transpires whenever $n\in [1,2N]\setminus \mathcal{Z}(N)$ that for every $1\leq d\leq l\leq s-1$, each $(\bfa,\bfb)\in [1,s]^{s-l}\times [1,s]^{d}$, every tuple of natural numbers $\bfx\in [1,n^{1/k}]^{d}$ and $m=Y_{\bfb}(n,\bfx)$ with $m>0$ one gets $F_{l,\bfa}(m)\leq \tilde{K} m^{-\nu_{0}}.$ Moreover, in view of (\ref{begi}) and (\ref{tau0}) one has for $d_{0}\leq d\leq s-1$ and $\bfa\in [1,s]^{s-d}$ that \begin{equation}\label{afri} \widetilde{\mathcal{Z}}_{d,\bfa}(N)=\o,\ \ \ \ \ \ \ \ \ \ \ \ \ \ \ \ \ \ \ \ \ \mathcal{Z}_{d}(N)=\o.\end{equation} We deduce from the analysis in the present section bounds for the size of the preceding set.
\begin{cor}\label{cor9.1}
Under the assumptions in Corollary \ref{cor4} and Proposition \ref{prop911} and the existence of some $d_{0}\leq \big\lceil\frac{s}{Dk}+\frac{k}{4s}\big\rceil$ as in Lemma \ref{lem9.1111}, there is some $\zeta_{k,s}>0$ for which
$$\lvert \mathcal{Z}(N)\rvert\ll N^{1-\zeta_{k,s}}.$$ 
\end{cor}
\begin{proof}
We first note for fixed $\bfa\in [1,s]^{s-d}$ with $1\leq d\leq d_{0}-1$ that Proposition \ref{prop911} yields $F_{d,\bfa}(n)\leq K_{1} n^{-k/4s^{3}}$ for all but $O(N^{1-d/k-1/240s})$ integers $n\in [N,2N]$. By (\ref{eqa}) then  \begin{equation}\label{Ztil}\lvert \widetilde{\mathcal{Z}}_{d,\bfa}(N)\rvert\ll N^{1-d/k-d/240s}\ll N^{1-1/240s}.\end{equation} 
Moreover, it transpires by Proposition \ref{prop8.2} and Corollary \ref{cor4} that whenever $s\geq 4k+2$ with $\Delta_{2s}^{*}<0$ then $\lvert \mathcal{Z}_{0}(N)\rvert\ll N^{1-\zeta_{k,s}}$. We note in view of (\ref{afri}) for each $1\leq d\leq s-1$ that 
$$\mathcal{Z}_{d}(N)\subset\bigcup_{l=d}^{s-1}\bigcup_{\substack{ (\bfa,\bfb)\in [1,s]^{s-l+d} }}\Big\{n\in [1,2N]:\ \ n=m+b_{1}x_{1}^{k}+\ldots+b_{d}x_{d}^{k},\ \ \ m\in \widetilde{\mathcal{Z}}_{l,\bfa}(N),\ \ x_{i}\in \NN\Big\}.$$ It thereby follows by (\ref{afri}) and (\ref{Ztil}) that $ \mathcal{Z}_{d}(N)=\o$ if $d_{0}\leq d\leq s-1$ and otherwise
$$\lvert \mathcal{Z}_{d}(N)\rvert\ll N^{d/k} \sum_{l=d}^{d_{0}-1}\lvert \widetilde{\mathcal{Z}}_{l,\bfa}(N)\rvert \ll N^{d/k} \sum_{l=d}^{d_{0}-1}N^{1-l/k-1/240s}\ll N^{1-1/240s},$$ as desired. The corollary then holds by combining the above equations.
\end{proof}

\section{The probabilistic method}\label{sec11}
We shall make use of Halberstam-Roth \cite[Theorem 13, \S 3]{Hal} in order to construct for a function $\psi$ of uniform growth with exponent $\varepsilon$ and upon recalling (\ref{cks}) the probabilistic space $\mathcal{S}_{\psi}(k,s,\eta)$ of sequences $\frak{X}\subset\mathbb{N}$ for which
\begin{equation}\label{probas}\mathbb{P}(y\in \frak{X})=\left\{
	       \begin{array}{ll}
	x^{-1+k/s}c_{k,s}(\eta)^{-1/s}\psi(x^{k})^{1/s}\ \ \ \ \ \ \ \ \ \ \ \ \ \ \ \ \ \ \ \ \ \ \text{if $y=x^{k}$ for some $x\in\mathcal{A}(x,x^{\eta})$}   \\
		0 \ \ \ \ \ \ \ \ \ \ \ \ \ \ \ \ \ \ \ \ \ \ \ \ \ \ \ \ \ \   \ \ \  \ \ \    \ \  \ \ \    \ \ \  \ \ \   \ \ \  \ \ \   \ \ \text{otherwise.}         \\
	       \end{array}
	     \right.
 \end{equation}
We denote for simplicity \begin{equation}\label{phi1}\psi_{1}(y)=c_{k,s}(\eta)^{-1}\psi(y),\end{equation} and consider for each $x\in\mathcal{A}(x,x^{\eta})$ the independent random variables $t_{x}$ defined by $t_{x}=1$ if $x^{k}\in \frak{X}$ and $t_{x}=0$ else. Equipped with the preceding objects we introduce 
\begin{equation}\label{jjs}R_{\frak{X}}^{s}(n)=\sum_{\substack{n=x_{1}^{k}+\ldots+x_{s}^{k}\\ x_{i}\in\mathcal{A}(x_{i},x_{i}^{\eta}) }}\prod_{j=1}^{s}t_{x_{j}}\end{equation}
for each $n\in\mathbb{N}$. We shall express the above random variable for convenience as
\begin{equation}\label{zz}R_{\frak{X}}^{s}(n)=R_{\frak{X},s}^{\neq}(n)+R_{\frak{X},s}^{=}(n),\end{equation} where the term $R_{\frak{X},s}^{\neq}(n)$ is defined analogously but with the underlying variables satisfying $x_{i}\neq x_{j}$ for $1\leq i<j\leq s.$ Likewise, $R_{\frak{X},s}^{=}(n)$ is defined similarly but comprising instead tuples with $x_{i}=x_{j}$ for some $1\leq i<j\leq s.$ Moreover, we recall (\ref{tau0}) and write
\begin{equation}\label{zzz}R_{\frak{X},s}^{\neq}(n)=R_{\frak{X},s}^{+}(n)+R_{\frak{X},s}^{0}(n),\end{equation} where $R_{\frak{X},s}^{+}(n)$ is defined by imposing to the underlying tuples the proviso $x_{i}>n^{\tau_{0}/k}$ for each $1\leq i\leq s$, and where $R_{\frak{X},s}^{0}(n)$ comprises tuples satisfying $x_{j}\leq n^{\tau_{0}/k}$ for some $1\leq j\leq s$.

In order to utilise suitable concentration inequalities it seems pertinent introducing first some notation. Let $t_{1},\ldots,t_{m}$ be independent Bernoulli random variables and let $Y(\mathbf{t})=Y(t_{1},\ldots,t_{n})$ be a \emph{positive simplified normal} polynomial of degree $d$, i.e. a polynomial of degree $d$ with positive coefficients of size at most $1$ with each of the factors $t_{i}$ in the monomials appearing at most once. We shall write for $A\subset [1,n]$ the symbol $\partial_{A}(F)$ to denote the partial derivatives of $Y$ with respect to the variables given by the indexes in $A$, and abbreviate by $\mathbb{E}_{A}(Y)$ the expected value $\mathbb{E}\big(\partial_{A}(Y(\mathbf{t}))\big).$ We also introduce for any $0\leq j\leq s-1$ the auxiliary expectation $\mathbb{E}_{j}(Y)=\max_{A\subset [1,n], j\leq  \lvert A\rvert\leq s-1}\mathbb{E}_{A}(Y).$ 
\begin{prop}\label{prop81}
Let $Y$ be a positive simplified normal polynomial of degree at most $s$. Then, for any $\varepsilon_{0},\lambda,B>0$ and every sufficiently large $K\geq K(k,s,B)$ there is some $C=C(K,k,s,B)>0$ such that when $\mathbb{E}_{1}(Y)\leq \varepsilon_{0}\leq 1$ and $4sK\lambda\leq \mathbb{E}(Y)=o(n)$ one has
$$\mathbb{P}\big(\lvert Y-\mathbb{E}(Y)\rvert\geq (4sK\lambda \mathbb{E}(Y))^{1/2}\big)\leq 2s e^{-\lambda/4}+Cn^{s}\varepsilon_{0}^{B}.$$
\end{prop}
\begin{proof}
See \cite[Theorem 1.3]{Vu2}.
\end{proof}

We shall next apply the procedure utilised in \cite{Vu}, it being desirable computing some auxiliary expectations with the aid of the arithmetic results obtained in previous sections.
\begin{lem}\label{lem9.1}
Let $s\geq 4k+1$ and $\Delta_{s}^{*}<0$ be an admissible exponent for minor arcs. Then, whenever $n\in [N,2N]$ one has
$$\mathbb{E}(R_{\frak{X},s}^{0}(n))\ll n^{-\tau_{0}/s}.$$
\end{lem}
\begin{proof}
By recalling the definitions in the previous page, (\ref{paar}) and (\ref{pij}) it transpires that
\begin{align*}\mathbb{E}(R_{\frak{X},s}^{0}(n))& \ll P^{\varepsilon}\sum_{\substack{(x_{1},\ldots,x_{s})\in\mathcal{C}(n,P^{\eta})\\ \min\limits_{i\leq s}(x_{i})\leq n^{\tau_{0}/k} }}(x_{1}\cdots x_{s})^{-1+k/s}.\end{align*}
We then note upon setting $\varphi(n)=n^{1-\tau_{0}}$ that (\ref{kkk}) yields
$$\mathbb{E}(R_{\frak{X},s}^{0}(n))\ll P^{\varepsilon}\lvert r_{s,k}(n,P^{\eta})-r_{s,k}^{\varphi}(n,P^{\eta})\rvert,$$
whence the lemma follows by (\ref{ki}), (\ref{tau0}) and the assumption $\Delta_{s}^{*}<0$.
\end{proof}

We next recall (\ref{cho}) and introduce first an auxiliary lemma used on multiple occassions.
\begin{lem}\label{lem9.65}
Let $s,\Delta_{s}^{*}$ be as in Lemma \ref{lem9.1} and $\gamma\neq 0$. For each $1\leq l\leq s-1$ then
\begin{equation*}I_{[0,1),\gamma}^{1,s-l}(m^{1/k},1)\ll m^{-2\tau_{0} /s}.\end{equation*}
\end{lem}
\begin{proof}
We apply Lemma \ref{lem8.10} for the choice $Q=m^{1/4s}$ and $P=m^{1/k}$ to deduce that $I_{\grm(Q),\gamma}^{1,s-l}(m^{1/k},1)=O( m^{-\nu /4s}),$ where $\nu$ was defined in (\ref{nu0}). On the other hand, the customary remark $\text{meas}(\grM(Q))\leq m^{1/2s-1}$ combined with the trivial bounds as in (\ref{hhhh}) yields
$$I_{\grM(Q),\gamma}^{1,s-l}(m^{1/k},1)\ll m^{l/s+1/2s-1}\ll m^{- 1/2s}.$$ The combination of the preceding estimates delivers the desired result.
\end{proof}
\begin{lem}\label{lem9.2}
Let $s,\Delta_{s}^{*}$ be as in Lemma \ref{lem9.1}. Then, for every $1\leq l\leq s-1$ one has
\begin{equation*}\mathbb{E}(R_{\frak{X},l}^{\neq }(m))\ll m^{-\tau_{0}/s}.\end{equation*}
\end{lem}
\begin{proof}
By recalling (\ref{ooo}) and a similar argument as in Lemma \ref{lem9.1} it is apparent that
$$\mathbb{E}(R_{\frak{X},l}^{\neq }(m))\ll m^{\varepsilon}I_{[0,1),1}^{1,s-l}(m^{1/k},1).$$
The lemma follows by combining the above estimate in conjunction with Lemma \ref{lem9.65}.
\end{proof}

We recall (\ref{tau0}), write for $\bfx\in\mathbb{N}^{s}$ and henceforth $\text{Set}(\bfx)=\{x_{1},\ldots,x_{s}\}$ and introduce 
\begin{equation}\label{Rn}\mathcal{R}^{+}(n)=\Big\{\bfx\in \mathbb{N}^{s}:  n=x_{1}^{k}+\ldots+x_{s}^{k}, x_{i}\in\mathcal{A}(x_{i},x_{i}^{\eta}):  x_{i}> n^{\tau_{0}/k},  x_{i}\neq x_{j}  \text{ for $i\neq j$}\Big\}.
\end{equation} 
\begin{lem}\label{lem9.4}
Let $s\geq 4k+1$ and $\Delta_{s}^{*}<0$ be an admissible exponent for minor arcs. Let $A=\{a_{1}^{k},\ldots, a_{l}^{k}\}\subset\mathbb{N}_{0}^{k}$ where $\lvert A\rvert=l$ with $a_{j}\in \mathcal{A}(a_{j},a_{j}^{\eta})$ and $1\leq l\leq s-1$. Then,
$$\mathbb{E}_{A}(R_{\frak{X},s}^{+ }(n))\ll n^{-\tau_{0}^{2}/s}.$$ If moreover $n\in [1,N]\setminus \mathcal{Z}(N)$ one gets $\mathbb{E}_{A}(R_{\frak{X},s}^{+ }(n))=O( n^{-\nu_{0}\tau_{0}/2}).$
\end{lem}
\begin{proof}
We write for convenience $m=n-\sum_{y^{k}\in A}y^{k}$ and $l=s-\lvert A\rvert.$ One then has that
$$\partial_{A}(R_{\frak{X},s}^{+ }(n))=\sum_{\substack{\bfx\in\mathcal{R}^{+}(n)\\ A\subset \text{Set}(\bfx)}}\prod_{x_{j}\in \text{Set}(\bfx)\setminus A}t_{x_{j}}.$$ We note that by relabelling if necessary, each $\bfx=(x_{1},\ldots,x_{s})$ in the preceding sum satisfies
$$m=x_{1}^{k}+\ldots+x_{l}^{k},$$ whence upon employing the definition (\ref{Rn}) it transpires that $m\geq x_{1}^{k}> n^{\tau_{0}}.$ We assume first $\Delta_{s}^{*}<0$ and observe that the application of Lemma \ref{lem9.2} yields
$$\mathbb{E}(\partial_{A}(R_{\frak{X},s}^{+ }(n)))\ll \mathbb{E}(R_{\frak{X},l}^{\neq }(m))\ll m^{-\tau_{0}/s}\ll n^{-\tau_{0}^{2}/s},$$ 
as desired. If $n\notin \mathcal{Z}(N)$ then $n\notin\mathcal{Z}_{\lvert A\rvert}(N)$ by (\ref{sas}), it in turn entailing $F_{\lvert A\rvert,\bf1}(m)=O( m^{-\nu_{0}})$. Therefore, the same argument in Lemma \ref{lem9.2} and orthogonality yield
\begin{align*}\mathbb{E}(\partial_{A}(R_{\frak{X},s}^{+ }(n)))&\ll m^{\varepsilon}\sum_{\substack{m=x_{1}^{k}+\ldots+x_{l}^{k}\\ x_{i}\in\mathcal{A}(m^{1/k},m^{\eta/k})}}(x_{1}\cdots x_{l})^{-1+k/s}\ll m^{\varepsilon} F_{\lvert A\rvert,\bf1}(m)\ll m^{-\nu_{0}/2}\ll n^{-\nu_{0}\tau_{0}/2},
\end{align*} wherein we employed the same devise as above.
\end{proof}
We shall next analyse the contribution of tuples with two of the components being equal. To such an end we introduce when $1\leq l\leq s-1$ for tuples $\boldsymbol{a}\in [1,s]^{l}$ the set
\begin{equation}\label{Rnb}\mathcal{R}_{\boldsymbol{a},l}(n)=\Big\{\bfx\in\NN^{l}:\ \ \  \ n=a_{1}x_{1}^{k}+\ldots+a_{l}x_{l}^{k},\ \ \ x_{i}\neq x_{j} \ \text{for $i\neq j$}, \ \ \ \ x_{i}\in\mathcal{A}(x_{i},x_{i}^{\eta})\  \Big\},
\end{equation} and consider the random variable
\begin{equation}\label{elli}R_{\mathfrak{X},\bfa}^{l}(n)=\sum_{\bfx\in \mathcal{R}_{\boldsymbol{a},l}(n)}\prod_{x_{j}\in\text{Set}(\bfx)}t_{x_{j}}.\end{equation}
\begin{lem}\label{lem9.5}
Let $s\geq 4k+1$ and let $\Delta_{s}^{*}<0$ be an admissible exponent for minor arcs. Let $1\leq l\leq s-1$ and $\boldsymbol{a}\in [1,s]^{l}$ with $a_{1}+\ldots+a_{l}\leq s$. Then, one has 
$$ \mathbb{E}(R_{\frak{X},\bfa}^{l }(n))\ll n^{-\tau_{0}/s},\ \ \ \ \ \ \ \ \ \mathbb{E}(R_{\frak{X},s}^{= }(n))\ll n^{-\tau_{0}/s}. $$
\end{lem}
\begin{proof}
By definition and upon writing $P=n^{1/k}$ it follows that
\begin{align*}\mathbb{E}(R_{\frak{X},\bfa}^{l }(n))&= \sum_{\substack{\boldsymbol{x}\in\mathcal{R}_{\boldsymbol{a},l}(n)\\  }}\mathbb{P}(x_{1}^{k},\ldots,x_{l}^{k}\in\frak{X}) \ll n^{\varepsilon}\sum_{\substack{n=a_{1}x_{1}^{k}+\ldots+a_{l}x_{l}^{k}\\ x_{i}\in\mathcal{A}(P,P^{\eta})}}\prod_{i=1}^{l}x_{i}^{-1+k/s}.
\end{align*}
Then it transpires by (\ref{i0}), (\ref{Fda}) and orthogonality that $\mathbb{E}(R_{\frak{X},\bfa}^{l }(n))=O( n^{\varepsilon}\lvert F_{s-l,\bfa}(n)\rvert),$ where we employed the fact, by relabelling if necessary, that $x_{1}>P_{-}$. We recall (\ref{cho}) and note that Holder's inequality in conjunction with a change of variables thereby entails
\begin{align*}\mathbb{E}(R_{\frak{X},\bfa}^{l }(n))&\ll n^{\varepsilon} \sum_{v=2}^{l} \int_{0}^{1}\lvert \tilde{f}_{s}(a_{1}\alpha,P,R)\rvert \lvert f_{s}(a_{v}\alpha,P,R)\rvert^{l-1}d\alpha\ll n^{\varepsilon}\sum_{v=2}^{l} I_{[0,1),a_{v}/a_{1}}^{1,s-l}(P,1).
\end{align*}
We apply Lemma \ref{lem9.65} to each of the terms on the inner sum in the above equation to get
\begin{equation}\label{RF}\mathbb{E}(R_{\frak{X},\bfa}^{l }(n))\ll n^{-\tau_{0}/s},\end{equation} as desired. The second estimate follows by observing that 
\begin{equation}\label{pps}R_{\frak{X},s}^{= }(n)=\sum_{l=1}^{s-1}\sum_{\substack{a_{1}+\ldots+ a_{l}=s\\ \bfa\in [1,s]^{l}}}R_{\frak{X},\bfa}^{l }(n),\end{equation} averaging on both sides and applying (\ref{RF}).
\end{proof}
Observe that the argument leading to (\ref{RF}) implies in particular upon recalling (\ref{Fda}) that whenever $\Delta_{s}^{*}<0$ then for any $1\leq d\leq s-1$, any $\bfa\in [1,s]^{s-d}$ and $m\in\NN$ one has $F_{d,\bfa}(m)\leq K_{3}m^{-2\tau_{0}/s}$ for some constant $K_{3}>0$, the definition (\ref{sas}) entailing
\begin{equation}\label{deac}\mathcal{Z}(N)=\mathcal{Z}_{0}(N).\end{equation}
We also record for future purposes upon alluding to (\ref{Fda}) and writing 
\begin{equation}\label{RRLL}R_{s}^{= }(n)=\sum_{l=1}^{s-1}\sum_{\substack{a_{1}+\ldots+ a_{l}=s\\ \bfa\in [1,s]^{l}}}F_{s-l,\bfa}(n)\end{equation} that the above procedure yields $R_{s}^{= }(n)\ll  n^{-2\tau_{0}/s}.$

We may now employ the argument in the last paragraph of \cite[page 128]{Vu} with Lemmata \ref{lem9.2} and \ref{lem9.1} in the present memoir replacing respectively their counterparts Lemmata 3.4 and 3.6 of the aforementioned paper to show that there exists some constant $C_{1}=C_{1}(s,k,\eta)>0$ such that with probability at least $4/5$ one has for every $n\in\mathbb{N}$ the bound \begin{equation}\label{RC}R_{\frak{X},s}^{0}(n)\leq C_{1}.\end{equation}
Likewise, one might utilise the aforementioned argument in \cite[page 128]{Vu} with Lemma \ref{lem9.5} herein (see also Lemma \ref{lem12.4}) replacing both Lemmata 3.4 and 3.6 therein  to show for some $C_{2}=C_{2}(s,k,\eta)>0$ that with probability at least $4/5$ one has for every $n\in\mathbb{N}$ the bound \begin{equation}\label{RC1}R_{\frak{X},s}^{=}(n)\leq C_{2}.\end{equation}

After the above sequel it remains to compute $\mathbb{E}(R_{\frak{X},s}^{+}(n))$. We then take a function $\varphi$ of uniform growth, recall (\ref{xz}), (\ref{hhh}) and note that one may assume $2\varphi(x)\leq \exp\big((\log x)^{1/2}\big)$.
\begin{prop}\label{prop9.1}
Let $s\geq 4k+1$ and $\psi,\varphi$ be of uniform growth with $\lvert\xi_{n}-1\rvert\asymp 1$ and $2\varphi(x)\leq \exp\big((\log x)^{1/2}\big)$. Whenever $\Delta_{s}^{*}<0$ is an admissible exponent for minor arcs then  \begin{equation}\label{mm}\mathbb{E}(R_{\frak{X},s}^{+}(n))\asymp_{k,s}\psi(n)\end{equation} for every $n\in [N,2N]$. If moreover $\xi(n)=o(1)$ then for every integer $n\in [N,2N]$ one has
\begin{equation}\label{bram}\mathbb{E}(R_{\frak{X},s}^{+}(n))=\frak{S}(n)\psi(n)+O\big(\psi(n)\big((\log n)^{-\upsilon_{0}}+\varphi(n)^{-\upsilon_{0}}+\xi_{n}\big)\big).\end{equation}
If instead $\Delta_{s}^{*}\geq 0$, both (\ref{mm}) and (\ref{bram}) also hold for every $n\in [N,2N]\setminus \mathcal{Z}(N).$
\end{prop}
\begin{proof}
We recall (\ref{paar}), (\ref{rsk}) and the definition of $R_{\frak{X},s}^{+}(n)$ right after (\ref{jjs}) to observe that
\begin{align*}\mathbb{E}(R_{\frak{X},s}^{+}(n))&\leq r_{k,s}(n,P^{\eta})\psi_{1}(n),\end{align*}wherein we applied the monotonicity of $\psi_{1}(x)$ and (\ref{fasi}). We remind the reader of (\ref{phi1}) and note under the hypothesis $\Delta_{s}^{*}<0$ that then Proposition \ref{prop7.1} delivers the upper bound
\begin{equation}\label{jjp}\mathbb{E}(R_{\frak{X},s}^{+}(n))\leq \frak{S}(n)\psi(n)+C_{s}\psi(n)\big((\log n)^{-\upsilon_{0}}+\varphi(n)^{-\upsilon_{0}})\big)\end{equation} for some $C_{s}>0$. If instead $\Delta_{s}^{*}\geq 0$ and $n\in [N,2N]\setminus \mathcal{Z}(N)$ then an analogous formula holds by (\ref{hhh}). In order to derive the corresponding lower bound we draw the attention back to (\ref{xxx}) and (\ref{RRLL}) and note by the monotonicity of $\psi_{1}$ that 
\begin{align*}\mathbb{E}(R_{\frak{X},s}^{+}(n))&\geq\sum_{\substack{n=x_{1}^{k}+\ldots+x_{s}^{k}\\ x_{i}\in\mathcal{A}(x_{i},x_{i}^{\eta})\\ x_{i}\neq x_{j}\\ x_{i}> (n\varphi(n)^{-1})^{1/k}}}\prod_{i=1}^{s}x_{i}^{-1+k/s}\psi_{1}(x_{i}^{k})^{1/s}\geq\psi_{1}(n/\varphi(n))r_{s,k,\eta}^{\varphi}(n)-\psi_{1}(n)R_{s}^{=}(n).\end{align*}
Then, by the estimate after (\ref{RRLL}) and Corollary \ref{cor2} combined with (\ref{phi1}) we deduce that 
\begin{equation}\label{aaa1}\mathbb{E}(R_{\frak{X},s}^{+}(n))\geq \frak{S}(n)\psi(n/\varphi(n))-C_{s,1}\psi(n)\big((\log n)^{-\upsilon_{0}}+\varphi(n)^{-\upsilon_{0}}\big),\end{equation}
where $C_{s,1}>0$ is some constant. If instead $n\notin \mathcal{Z}(N)$, an analogous formula holds by (\ref{eqa}) and (\ref{sas}) in conjunction with the monotonicity of $\psi_{1}$ and the fact that $n\notin \mathcal{Z}_{0}(N)$.

We shall allude now to the classical theory of Waring's problem to note that under the assumptions on $s$ described above it follows by \cite[Theorems 4.3 and 4.6]{Vau} that \begin{equation}\label{SSS}\frak{S}(n)\asymp_{k,s} 1.\end{equation} Equipped with the above bound we observe that under the first assumption on $\xi$ one has
$\mathbb{E}(R_{\frak{X},s}^{+}(n))\gg \psi(n),$ which in conjunction with (\ref{jjp}) and (\ref{SSS}) delivers (\ref{mm}). If moreover $\xi(n)=o(1)$ then by (\ref{aaa1}) and (\ref{SSS}) one has
\begin{align*}\mathbb{E}(R_{\frak{X},s}^{+}(n))&\geq  \frak{S}(n)\psi(n)-\tilde{C}_{s}\psi(n)\big((\log n)^{-\upsilon_{0}}+\varphi(n)^{-\upsilon_{0}}+\xi_{n}\big),\end{align*} with $\tilde{C}_{s}>0$. The above equation combined with (\ref{jjp}) yields the desired result.
\end{proof}

\section{Proof of Theorems \ref{thm1.1} and \ref{thm1.2}}\label{sec12}
We have now reached a point from which to conclude the proof of Theorems \ref{thm1.1}, \ref{thm1.2} and Corollaries \ref{cooor1}, \ref{cooor2}. To such an end we first note that whenever $v\geq k^{2}+k-2$ then \cite[Corollary 10.2]{Woo6} and orthogonality permits one to deduce that $\Delta_{2v}=0$ is an admissible exponent. We thereby recall (\ref{eq8}) to deduce whenever $s\geq 2k^{2}+2k$ that
\begin{align}\label{all}\Delta_{s}^{*}&\leq -(s-2k^{2}-2k+4)(\min(\tau(k), k/s))<0.
\end{align}
If instead $s<2k^{2}+2k$ then $s\leq 4k^{2}$ and $\tau(k)\leq k/s$ by (\ref{taucer}), from where it follows that
\begin{equation*}\Delta_{s}^{*}= \min_{\substack{1\leq v\leq s/2}}\Big(\Delta_{2v}-(s-2v)\tau(k)\Big)=\tau(k)\min_{\substack{1\leq v\leq s/2}}\Big(2v+\frac{\Delta_{2v}}{\tau(k)}\Big)-s\tau(k),\end{equation*}
where $v\in\NN$. We recall (\ref{eq67}) and observe that if $s\geq \max(\lfloor G_{0}(k)\rfloor+1,4k+1)$ then $s>G_{0}(k)$. Therefore, the underlying $v$ therein attaining the minimum would satisfy $s> 2v$, and hence
\begin{equation}\label{s00}\Delta_{s}^{*}=\tau(k)(G_{0}(k)-s)<0.\end{equation} We combine (\ref{all}) with (\ref{s00}) to deduce that under the assumption of Theorem \ref{thm1.1} then \begin{equation}\label{ssssp}s\geq 4k+1\ \ \ \ \ \ \ \ \ \text{and}\ \ \ \ \ \ \ \ \ \Delta_{s}^{*}<0.\end{equation}

We next invoke Proposition \ref{prop81}, note that the requirements in Lemma \ref{lem9.4} are met and observe that using the notation underlying the discussion thereof, $\mathbb{E}_{1}(R_{\frak{X},s}^{+ }(n))=O(n^{-\tau_{0}^{2}/s}).$ We also note that $\mathbb{E}(R_{\frak{X},s}^{+ }(n))\ll n^{\varepsilon},$ such a conclusion stemming from (\ref{ssssp}), Proposition \ref{prop9.1} and the proviso $\psi(n)\ll n^{\varepsilon}.$ We then consider a function $\delta: \mathbb{N}\rightarrow (0,1)$, set $\varepsilon_{0}=n^{-\tau_{0}^{2}/s}$, take $\lambda=\delta(n)^{2}\mathbb{E}(R_{\frak{X},s}^{+ }(n))/4sK$ for large $K=K(k,s)$ and apply Proposition \ref{prop81} to get 
\begin{equation}\label{schu}\mathbb{P}\big(\lvert R_{\frak{X},s}^{+ }(n)-\mathbb{E}(R_{\frak{X},s}^{+ }(n))\rvert\geq \delta(n) \mathbb{E}(R_{\frak{X},s}^{+ }(n))\big)\ll  e^{-\delta(n)^{2}\mathbb{E}(R_{\frak{X},s}^{+ }(n))/16sK}+n^{-2}.\end{equation}
We next assume $\psi(n)= C_{k,s,\eta}\log n$ for a large constant $C_{k,s,\eta}>0$ and note that under the conditions of the first instance described in Proposition \ref{prop9.1} which we may apply as a consequence of (\ref{ssssp}), it transpires upon taking $\delta(n)=1/2$ that the above estimate yields 
$$\mathbb{P}\big(\lvert R_{\frak{X},s}^{+ }(n)-\mathbb{E}(R_{\frak{X},s}^{+ }(n))\rvert\geq (1/2)\mathbb{E}(R_{\frak{X},s}^{+ }(n))\big)\ll  n^{-2}.$$ Therefore, the application of Borel-Cantelli Lemma combined with the aforementioned proposition enables one to deduce with probability $1$ for sufficiently large $n\in\mathbb{N}$ the relation
\begin{equation}\label{land}R_{\frak{X},s}^{+ }(n)\asymp_{k,s} \log n.\end{equation} 

If moreover $\psi(n)$ satisfies the requisites in Theorem \ref{thm1.2} we take $\delta(n)=C_{K} (\log n/\psi(n))^{1/2}$ for some large enough constant $C_{K}=C_{K}(k,s)$ and combine (\ref{mm}) with (\ref{schu}) to get
$$\mathbb{P}\big(\lvert R_{\frak{X},s}^{+ }(n)-\mathbb{E}(R_{\frak{X},s}^{+ }(n))\rvert\geq \delta(n) \mathbb{E}(R_{\frak{X},s}^{+ }(n))\big)\ll n^{-2}. $$ Consequently, Borel-Cantelli Lemma \cite[Theorem 7, §3]{Hal} combined with Proposition \ref{prop9.1} and (\ref{SSS}) enables one to obtain with probability $1$ and for sufficiently large $n$ the formula
\begin{align}\label{realdo}R_{\frak{X},s}^{+ }(n)&=\mathbb{E}(R_{\frak{X},s}^{+}(n))\big(1+O(\delta(n))\big)\nonumber
\\
&=\frak{S}(n)\psi(n)+O\big(\psi(n)(\delta(n)+\xi(n)+\varphi(n)^{-\upsilon_{0}}+(\log n)^{-\upsilon_{0}})\big).
\end{align}
We complete the proof of Theorem \ref{thm1.1} by appealing to equations (\ref{zz}) and (\ref{zzz}) and noting that the conclusions (\ref{RC}) and (\ref{RC1}) combined with (\ref{land}) permit one to deduce that with probability at least $3/5$ then one has for sufficiently large $n$ the estimate \begin{equation}\label{bb}R_{\frak{X}}^{s}(n)\asymp_{k,s} \log n.\end{equation} In the context underlying Theorem \ref{thm1.2} then (\ref{realdo}) in conjunction with the above entails 
\begin{equation}\label{iiu}R_{\frak{X}}^{s}(n)=\frak{S}(n)\psi(n)+O\Big(\psi(n)\Big(\xi(n)+\varphi(n)^{-\upsilon_{0}}+(\log n)^{-\upsilon_{0}}+\Big(\frac{\log n}{\psi(n)}\Big)^{1/2}\Big)\Big).\end{equation} The existence of $\frak{X}\subset \mathbb{N}_{0}^{k}$ satisfying either (\ref{bb}) or (\ref{iiu}) follows from the above discussion.

In order to deduce the corollaries it suffices to provide upper bounds for $G_{0}(k)$, the one cognate to Corollary \ref{cooor1} stemming from the discussion right after \cite[(7.8)]{Bru-Woo} whenever $k> 20$, and the ones pertaining to Corollary \ref{cooor2} flowing from the analysis in the proof of \cite[Theorem 1.2]{Bru-Woo}, the corresponding $s$ thereby satisfying (\ref{ssssp}). Note that the underlying $v$ taken therein to bound $G_{0}(k)$ introduced on that memoir is an even number, the corresponding estimate thereby delivering a legitimate bound for the function $G_{0}(k)$ introduced herein. 

If $14\leq k\leq 20$ we employ \cite[Section 8]{Bru-Woo} to deduce bounds for $G_{0}(k)$ included in Table 2 and derive Corollary \ref{cooor2}. For $3\leq k\leq 13$ we allude to \cite[Theorem 5.1]{WooVu} and note that the function $\mathfrak{H}(k)$ presented therein has the property that whenever $s\geq \mathfrak{H}(k)$ then (\ref{XXX}) holds. In the context underlying Theorem \ref{thm1.2} we observe as is done in \cite[Theorem 5.1]{WooVu} that $\Delta_{\mathfrak{H}(k)-1}=0$ is an admissible exponent, and hence $\Delta_{\mathfrak{H}(k)}^{*}<0$. The preceding remarks then yield both corollaries and the theorem for such a range. If $k=2$ then $\Delta_{8}=0$ is an admissible exponent, and hence $\Delta_{9}^{*}<0$. We conclude the proofs by making recourse to \cite{ErTe} when $k=1$ in Corollaries \ref{cooor1}, \ref{cooor2}, that cognate to Theorem \ref{thm1.2} following mutatis mutandis.
\section{Thinner sequences}\label{sec13}
We shall prepare the ground for the application of probabilistic results presented in previous sections to deliver the existence of thinner sequences satisfying suitable properties. We introduce for $\delta: \mathbb{N}\rightarrow (0,1)$ satisfying (\ref{delll}), a function $\psi$ of uniform growth with $\psi(n)=O(\log n)$, a sufficiently large $N\in\mathbb{N}$ and fixed constants $c,C>0$ the parameters
\begin{equation}\label{ML}L_{C}(N)=\Big\lfloor\frac{C\log N}{\psi(N) \delta(N)^{2}}\Big\rfloor,\ \ \ \ \ M_{c}(N)=(\log N)e^{c\delta(N)^{2}\psi(N)},\end{equation} and set $L=L_{C}(N)$. Recall (\ref{probas}), (\ref{jjs}) and take for $\mathcal{N}_{L}=\{n_{1},.,n_{L}\}$ the random variables
\begin{equation}\label{ES1}R_{\frak{X}}^{s}(\mathcal{N}_{L})=\sum_{i=1}^{L}R_{\frak{X}}^{s}(n_{i})\ \ \ \ \ \ \text{and}\ \ \ \ \ \ R_{\frak{X},s}^{+}(\mathcal{N}_{L})=\sum_{i=1}^{L}R_{\frak{X},s}^{+}(n_{i}).\end{equation}
It is then worth recalling (\ref{xz}) first and presenting the counterpart of Proposition \ref{prop9.1}. 
\begin{prop}\label{prop12.1}
Let $s\geq 4k+1$ and let $\mathcal{N}_{L}\subset [N,2N]\setminus \mathcal{Z}(N)$ be a subset satisfying $\lvert \mathcal{N}_{L}\rvert=L$. Then when $\psi,\varphi$ are as in Proposition \ref{prop9.1} with $\xi(n)=o(1)$ it follows that
\begin{equation*}\mathbb{E}(R_{\frak{X},s}^{+}(\mathcal{N}_{L}))=\psi(N)\sum_{i=1}^{L}\frak{S}(n_{i})+O\Big(\psi(N)L\big(\xi(N)+\varphi(N)^{-\upsilon_{0}}+(\log N)^{-\upsilon_{0}}\big)\Big).\end{equation*}
\end{prop}
\begin{proof}
We begin taking expected values in (\ref{ES1}) and using linearity to obtain
\begin{equation}\label{su}\mathbb{E}\big(R_{\frak{X},s}^{+}(\mathcal{N}_{L})\big)=\sum_{i=1}^{L}\mathbb{E}\big(R_{\frak{X},s}^{+}(n_{i})\big).\end{equation} Observe by the condition on $\xi$ that $\psi(n)\asymp \psi\big(\frac{n}{\varphi(n)}\big),$ which entails in view of the fact that $\psi$ and $\varphi$ are of uniform growth the bound $\psi(n)\asymp \psi(N).$ We note by the increasing property of $\psi$ in conjunction with (\ref{xz}) that whenever $n\in[N,2N]$ and $N$ is large enough one has
$$0\leq \psi(n)-\psi(N)\leq \psi\Big(\frac{n}{\varphi(n)}\Big)+\xi(N)\psi(n)-\psi(N).$$ Since $\varphi,\psi$ are functions of uniform growth it transpires that $n/\varphi(n)<N$, and hence
\begin{equation}\label{fax}\psi(n)=\psi(N)+O(\xi(N)\psi(N)),\end{equation}
where we used the estimate $\psi(n)\asymp \psi(N)$. We then insert (\ref{bram}) into (\ref{su}) and thus get
$$\mathbb{E}\big(R_{\frak{X},s}^{+}(\mathcal{N}_{L})\big)=\sum_{i=1}^{L}\mathfrak{S}(n_{i})\psi(n_{i})+ O\Big(\sum_{i=1}^{L}\psi(n_{i})\big(\varphi(n_{i})^{-\upsilon_{0}}+(\log n_{i})^{-\upsilon_{0}}+\xi_{n_{i}}\big)\Big).$$ We conclude by combining the above equation, (\ref{SSS}) and (\ref{fax}) for the choice $n=n_{i}$.
\end{proof}

In what follows we shall prepare the ground for the application of Proposition \ref{prop81}.

\begin{lem}\label{lem12.1}
Let $s\geq 2$ and let $A=\{a_{1}^{k},\ldots,a_{d}^{k}\}\subset \mathbb{N}_{0}^{k}$ be a subset satisfying $\lvert A\rvert=d$ with $a_{j}\in\mathcal{A}(a_{j},a_{j}^{\eta})$ and $1\leq j\leq d\leq s-1$. Let $\mathcal{N}_{L}\subset [N,2N]\setminus \mathcal{Z}(N)$ with $\lvert \mathcal{N}_{L}\rvert=L$. Then,
$$\mathbb{E}_{A}(R_{\frak{X},s}^{+}(\mathcal{N}_{L}))\ll N^{-\nu_{0}\tau_{0}/2}.$$
\end{lem}
\begin{proof}
We first write as is customary $\mathcal{N}_{L}=\{n_{1},\ldots,n_{L}\}$ and 
$$m_{i}=n_{i}-\sum_{y^{k}\in A}y^{k},\ \ \ \ \ \ \ \ \ \ \ \ \ \ 1\leq i\leq L,\ \ \ \ \ \ \ \ \ \ \ \ l=s-\lvert A\rvert.$$ We recall (\ref{Rn}) and draw the reader's attention to the description right after (\ref{zzz}) to get
$$\partial_{A}R_{\frak{X},s}^{+}(\mathcal{N}_{L})=\sum_{i=1}^{L}\sum_{\substack{\bfx\in \mathcal{R}^{+}(n_{i})\\ A\subset \text{Set}(\bfx)}}\prod_{x_{j}\in \text{Set}(\bfx)\setminus A}t_{x_{j}}.$$ It seems worth noting in view of the fact that each $m_{i}$ is expressible as a sum of $k$-th powers of some numbers in $\text{Set}(\bfx)$ for each $\bfx\in \mathcal{R}^{+}(n_{i})$ that then $m_{i}\gg N^{\tau_{0}}.$ Averaging on both sides of the preceding equation and recalling (\ref{Fda}) thereby delivers
\begin{equation}\label{ec13.1}\mathbb{E}_{A}\big(R_{\frak{X},s}^{+}(\mathcal{N}_{L})\big)\ll\sum_{i=1}^{L}\mathbb{E}\big(R_{\frak{X},l}^{+}(m_{i})\big)\ll N^{\varepsilon} \sum_{i=1}^{L}F_{s-l,\bf1}(m_{i}).\end{equation} Since $n_{i}\notin \mathcal{Z}_{s-l}(N)$ it transpires that $F_{s-l,\bf1}(m_{i})\ll m_{i}^{-\nu_{0}}$, the discussion in the above paragraph entailing $F_{s-l,\bf1}(m_{i})\ll N^{-\nu_{0}\tau_{0}}.$ We conclude by inserting the latter bound and the estimate $L\ll (\log N)^{2}$ in (\ref{ec13.1}), it in turn stemming from (\ref{delll}) and (\ref{ML}).
\end{proof}

We make further progress in the proof by presenting the counterparts of Lemmata \ref{lem9.1}, \ref{lem9.65} and \ref{lem9.5}, it being desirable recalling (\ref{eqa}), (\ref{hhj}) and (\ref{elli}). 
\begin{lem}\label{lem12.2}
Let $s\geq 2$ and $n\in [1,2N]\setminus \mathcal{Z}_{1}(N).$ Then one has
$$\mathbb{E}(R_{\frak{X},s}^{0}(n))\ll n^{-\nu_{0}/4}.$$Moreover, if $1\leq l\leq s-1$ and $\bfa\in[1,s]^{l}$ with either $n\in [1,2N]\setminus \widetilde{\mathcal{Z}}_{s-l,\bfa}(N)$ or $n\in [1,2N]\setminus \widetilde{\mathcal{Z}}_{s-l,\bf1}(N)$ one gets respectively the estimates \begin{equation*}\mathbb{E}(R_{\mathfrak{X},\bfa}^{l}(n))\ll n^{-\nu_{0}/2},\ \ \ \ \ \ \ \ \ \ \ \ \ \ \ \ \ \ \ \ \ \ \ \ \ \mathbb{E}(R_{\frak{X},l}^{\neq}(n))\ll n^{-\nu_{0}/2}.\end{equation*} 
\end{lem}
\begin{proof}
We draw the reader's attention to the proof of Lemma \ref{lem9.1}, recall (\ref{paar}) and get 
$$\mathbb{E}(R_{\frak{X},s}^{0}(n))\ll P^{\varepsilon}\sum_{\substack{(x_{1},\ldots,x_{s})\in \mathcal{C}(n,P^{\eta})\\ x_{1}\leq n^{\tau_{0}/k} }}(x_{1}\cdots x_{s})^{-1+k/s}\ll P^{\varepsilon}\sum_{\substack{ x\leq n^{\tau_{0}/k} }}x^{-1+k/s}F_{1,\bf1}(n-x^{k}),$$ wherein $\bf1=(1,\ldots,1)\in [1,s]^{s-1}$, whence upon noting that $n\notin \mathcal{Z}_{1}(N)$ one has
$$\mathbb{E}(R_{\frak{X},s}^{0}(n))\ll n^{\varepsilon-\nu_{0}}\sum_{\substack{ x\leq n^{\tau_{0}/k} }}x^{-1+k/s}\ll n^{\varepsilon+\tau_{0}/s-\nu_{0}}.$$
The first statement then would follow by recalling (\ref{tau0}). For the second one we recall (\ref{Rnb}) and note that $\mathbb{E}(R_{\mathfrak{X},\bfa}^{l}(n))\ll  n^{\varepsilon}F_{s-l,\bfa}(n)$. Consequently, since $n\notin \widetilde{\mathcal{Z}}_{s-l,\bfa}(N)$ then the right side of the previous equation is $O(n^{-\nu_{0}/2})$. The last claim follows by setting $\bfa=\bf1$.
\end{proof}

Having prepared the ground for the analysis of the preceding auxiliary random variables, we now proceed to estimate these with high probability. To such an end we recall (\ref{Rnb}) and (\ref{pps}), and introduce for $1\leq l\leq s-1$, each $\bfa\in[1,s]^{l}$ and $\mathfrak{X}\subset \mathbb{N}_{0}^{k}$ the set
\begin{equation}\label{Qaa}Q_{\mathfrak{X},\bfa}^{l}(n)=\Big\{\bfx\in \mathcal{R}_{\bfa,l}(n):\ \ \ \ \ \ \ \ x_{j}^{k}\in \mathfrak{X},\ \ \ \ \  1\leq j\leq l\Big\}.\end{equation}
We also write $Q_{\mathfrak{X}}^{l}(n)=Q_{\mathfrak{X},\bf1}^{l}(n)$ and denote by $\text{Disj}(Q_{\mathfrak{X}}^{l}(n))$ to the maximum $h$ such that $Q_{\mathfrak{X}}^{l}(n)$ contains $h$ pairwise disjoint tuples. Likewise, we further introduce for convenience
$$S_{\mathfrak{X},s}^{0}(n)=\Big\{\bfx\in Q_{\mathfrak{X}}^{s}(n):\ \ \ \ \ \ \ \ \ x_{i}\leq n^{\tau_{0}/k} \ \text{for some $1\leq i\leq s$}\Big\}.$$

\begin{lem}\label{lem12.3}
Let $s\geq 2$. There is some constant $K>0$ such that with probability at least $0.8$ and $n\notin \mathcal{Z}(N)$ for any $N\geq 1$ then
\begin{equation}\label{r0}R_{\mathfrak{X},s}^{0}(n)\leq K.\end{equation}
\end{lem}
\begin{proof}
We essentially follow the proof of \cite[Lemma 1.4]{Vu}. We note for fixed $K_{1}'>0$ that whenever $1\leq l\leq s-1$ and $m\notin\widetilde{\mathcal{Z}}_{s-l,\bf1}(N)$ for any $N\geq 1$ one has by Lemma \ref{lem12.2} that
\begin{align*}\mathbb{P}\big(\text{Disj}(Q_{\mathfrak{X}}^{l}(m))&\geq K_{1}'\big)\leq \Big(\mathbb{E}\Big(\sum_{\bfx\in Q_{\mathfrak{X}}^{l}(m)}\prod_{x_{i}\in\text{Set}(\bfx)}t_{x_{i}}\Big)\Big)^{K_{1}'}\leq \mathbb{E}\big(R_{\mathfrak{X},l}^{\neq }(m)\big)^{K_{1}'}\ll m^{-\nu_{0}K_{1}'/2}.
\end{align*}
Taking $K_{1}'>3/\nu_{0}$ sufficiently large it would follow with probability at least $0.9$ that then 
\begin{equation}\label{disj}\text{Disj}(Q_{\mathfrak{X}}^{l}(m))<K_{1}'\ \ \ \ \ \ \ \ \ \ \text{for all $m\notin\widetilde{\mathcal{Z}}_{s-l,\bf1}(N)$ for any $N\geq 1$}.\end{equation}

We return to Lemma \ref{lem12.2} for the purpose of noting that $\mathbb{E}(R_{\mathfrak{X},s}^{0}(n))=O(n^{-\nu_{0}/4})$, whence an analogous argument assures for sufficiently large $K_{2}'$ and probability at least $0.9$ that \begin{equation}\label{pranz}\text{Disj}(S_{\mathfrak{X},s}^{0}(m))\leq K_{2}'\ \ \ \ \ \ \ \ \text{for all $m\notin \mathcal{Z}_{1}(N)$ for any $N\geq 1$}.\end{equation} We next take $K=(K_{3}'-1)^{s} s!$ and $K_{3}'=\max(K_{2}',K_{1}')+1$ and note that by the Erd\H{o}s-Rado’s sunflower lemma \cite{Erd3} (see \cite[Lemma 1.4]{Vu}) one may deduce that whenever $R_{\mathfrak{X},s}^{0}(n)>K$ then either $\text{Disj}(S_{\mathfrak{X},s}^{0}(n))\geq K_{3}'$ or $\text{Disj}(Q_{\mathfrak{X}}^{l}(m'))\geq K_{3}'$ for some $1\leq l\leq s-1$ and some $$m'=n-\sum_{j=1}^{s-l}y_{j}^{k},\ \ \ \ \ \ \ \ \ \ y_{j}\in\mathbb{N}$$ according to whether the corresponding tuples $\bfx_{i}$ stemming from such an application satisfy $\cap_{i=1}^{K_{3}'}\text{Set}(\bfx_{i})=\o$ or not. In view of the assumption $n\notin \mathcal{Z}(N)$ for any $N\geq 1$ it transpires that $n\notin \mathcal{Z}_{s-l}(N)$, such a condition in turn implying $F_{s-l,\bf1}(m')\leq \tilde{K}(m')^{-\nu_{0}}$ and hence $m'\notin \widetilde{\mathcal{Z}}_{s-l,\bf1}(N)$. Consequently, by (\ref{disj}) and (\ref{pranz}), the union of the preceding events occurs with probability at most $0.2$, whence (\ref{r0}) happens with probability at least $0.8$.
\end{proof}

The perusal of $R_{\frak{X},s}^{=}(n)$ shall be analogous, concission being adopted at times. 
\begin{lem}\label{lem12.4}
Let $s\geq 2$. There is some constant $C^{=}>0$ such that with probability at least $0.9$ and $n\notin \mathcal{Z}(N)$ for any $N\geq 1$ then
\begin{equation*}R_{\mathfrak{X},s}^{=}(n)\leq C^{=}.\end{equation*}
\end{lem}
\begin{proof}
We recall (\ref{Qaa}) and Lemma \ref{lem12.2} to the end of observing that $\mathbb{E}(R_{\mathfrak{X},\bfa}^{j}(n))=O(n^{-\nu_{0}/4})$ for each $1\leq j\leq s-1$ and $\bfa\in[1,s]^{j}$. The argument in Lemma \ref{lem12.3} permits one to deduce with probability at least $1-0.1s^{-s}$ that for some sufficiently large $C_{1}>0$ one has $$\text{Disj}(Q_{\mathfrak{X},\bfa}^{j}(m))<C_{1},\ \ \ \ \ \ \ \ \ \ \ \ \ m\notin  \widetilde{\mathcal{Z}}_{s-j,\bfa}(N)\ \text{for any $N\geq 1$}.$$ Set $C_{2}=(C_{1}-1)^{s}s!+1$. Then, by a routinary application of the sunflower lemma it is apparent for $1\leq l\leq s-1$ and $\bfa\in[1,s]^{l}$ that if $R_{\mathfrak{X},\bfa}^{l}(n)>C_{2}$ for some $n\notin \mathcal{Z}(N)$ for any $N\geq 1$ then $\text{Disj}(Q_{\mathfrak{X},\tilde{\bfa}}^{l'}(n'))\geq C_{1}$ with $1\leq l'\leq l$, where $\tilde{\bfa}=(a_{i_{l-l'+1}},\ldots,a_{i_{l}})$, and
$$n'=n-\sum_{j=1}^{l-l'}a_{i_{j}}y_{i_{j}}^{k}\ \ \ \ \ \ \ \ \ \ \ y_{i_{j}}\in\NN,$$ wherein since $n\notin \mathcal{Z}(N)$ then $n'\notin  \widetilde{\mathcal{Z}}_{s-l',\tilde{\bfa}}(N)$, which by the above happens with probability at most $0.1s^{-s}$. Consequently, one has with probability at least $1-0.1s^{-s}$ that when $n\notin \mathcal{Z}(N)$ for any $N\geq 1$ then $R_{\mathfrak{X},\bfa}^{l}(n)\leq C_{2},$ such a bound combined with (\ref{pps}) entailing
$R_{\mathfrak{X},s}^{=}(n)\leq s^{s}C_{2}$ with probability at least $0.9$. The lemma follows by setting $C^{=}= s^{s}C_{2}.$
\end{proof}
It not being appropriate to dilate further on the preparatives, we present the aforementioned concentration inequality denoting beforehand $\mu_{\mathcal{N}_{L}}=\mathbb{E}\big(R_{\frak{X},s}^{+}(\mathcal{N}_{L})\big).$

\begin{prop}\label{prop12.2}
Let $s\geq 4k+1$ and let $\mathcal{N}_{L}\subset [N,2N]\setminus \mathcal{Z}(N)$ be a subset satisfying $\lvert \mathcal{N}_{L}\rvert=L$. Then for every $\beta>1$ there exists $\kappa_{0}=\kappa_{0}(k,s,\eta,\beta)>0$ with $\kappa_{0}\leq 1$ for which 
$$\mathbb{P}\big(\lvert R_{\frak{X},s}^{+}(\mathcal{N}_{L})-\mu_{\mathcal{N}_{L}}\rvert\geq \delta(N)\mu_{\mathcal{N}_{L}}\big)\ll N^{-\kappa_{0} C}+N^{-\beta}.$$
\end{prop}
\begin{proof}
We invoke Proposition \ref{prop81} and note in view of Lemma \ref{lem12.1} that $\mathbb{E}_{1}(R_{\frak{X},s}^{+ }(\mathcal{N}_{L}))=O( N^{-\nu_{0}\tau_{0}/2}).$ We also observe that $\mu_{\mathcal{N}_{L}}=O(N^{\varepsilon}),$ such a conclusion stemming from Proposition \ref{prop12.1} and the bound $\psi(N)L\ll N^{\varepsilon}.$  Then, setting $\varepsilon_{0}=N^{-\nu_{0}\tau_{0}/2}$ and taking a sufficiently large $K=K(k,s,\eta, \beta)$ and $\lambda=\delta(N)^{2}\mu_{\mathcal{N}_{L}}/4sK$ we apply Proposition \ref{prop81} to obtain 
\begin{equation}\label{schu4}\mathbb{P}\big(\lvert R_{\frak{X},s}^{+ }(\mathcal{N}_{L})-\mu_{\mathcal{N}_{L}}\rvert\geq \delta(N) \mu_{\mathcal{N}_{L}}\big)\ll  e^{-\lambda/4}+N^{-\beta}.\end{equation}
We recall (\ref{ML}) and observe by Proposition \ref{prop12.1} that there is some $\kappa'(s,k)>0$ for which $$\delta(N)^{2}\mu_{\mathcal{N}_{L}}/16sK\geq \kappa'(s,k)\delta(N)^{2}L\psi(N)/K\geq C\kappa'(s,k)(\log N)/2K.$$ Inserting the previous bound in (\ref{schu4}) delivers the desired result.
\end{proof}

In order to make further progress we recall (\ref{ML}) and consider as in Theorem \ref{thm1.9} and for every fixed constant $\kappa\geq 1$ and $1\leq j\leq N^{\kappa}$ a collection of sets $\mathcal{M}_{j}(N)\subset [N,2N]\setminus  \mathcal{Z}(N) $ with $\lvert \mathcal{M}_{j}(N)\rvert\leq M_{c}(N)$ for sufficiently large $N$. We then make in (\ref{ML}) the choices \begin{equation}\label{kapa}c=\kappa_{0}/2,\ \ \ \ \ \ \ C=1+\frac{6\kappa}{\kappa_{0}},\ \ \ \ \ \ \ \ \beta=7\kappa+2,\end{equation} the parameter $\kappa_{0}=\kappa_{0}(k,s,\eta,\beta)\leq 1$ stemming from the application of Proposition \ref{prop12.2}.
\begin{thm}\label{prop12.3}
Let $\psi(n)=O(\log n)$ be of uniform growth and $\xi(n)=o(1)$. Assume that \begin{equation}\label{xist}\max\Big(\Big(\frac{C\log(\psi(N))}{c\psi(N)}\Big)^{1/2},\tilde{C}\big(\xi(N)+\varphi(N)^{-\upsilon_{0}}+(\log N)^{-\upsilon_{0}}\big)\Big)\leq \delta (N)<1\end{equation} for some $\tilde{C}>1$ and large $N$. With probability at least $0.7$ one has for sufficiently large $N\in\NN$, each $\mathcal{M}_{j}(N)$ as above and all $\mathcal{N}_{L}\subset \mathcal{M}_{j}(N)$ with $\lvert \mathcal{N}_{L}\rvert=L$ for $1\leq j\leq N^{\kappa}$ that
\begin{equation}\label{qqq}R_{\frak{X}}^{s}(\mathcal{N}_{L})=\psi(N)\sum_{i=1}^{L}\frak{S}(n_{i})+O\big(\psi(N)L\delta(N)\big).\end{equation}
\end{thm}
\begin{proof}
We write $M=\lfloor M_{c}(N)\rfloor$ and examine firstly the associated binomial coefficient underlying the preceding statement, thereby obtaining by Stirling's formula
$$\binom{M}{L}\sim \frac{\Big(\frac{M}{e}\Big)^{M}\sqrt{2\pi M}}{\Big(\frac{L}{e}\Big)^{L}\sqrt{2\pi L} \Big(\frac{M-L}{e}\Big)^{M-L}\sqrt{2\pi (M-L)}}\ll \frac{1}{\big(1-\frac{L}{M}\big)^{M-L}}\Big(\frac{M}{L}\Big)^{L}.$$
Then, we recall (\ref{ML}) and employ the proviso (\ref{xist}) and the fact that $3/2<\log C$ to get
\begin{align}\label{ale}\binom{M}{L}&\ll e^{3(M-L)L/2M}\big(C^{-1}e^{c\delta(N)^{2}\psi(N)}\psi(N)\big)^{\frac{C(\log N)}{\delta(N)^{2}\psi(N)}}
\\
&\ll e^{cC\log N+\frac{C(\log N)}{\delta(N)^{2}\psi(N)}(3/2+\log (\psi(N))-\log C)}\ll  e^{cC\log N+\frac{C(\log N)\log (\psi(N))}{\delta(N)^{2}\psi(N)}}\ll  e^{c(C+1)\log N}.\nonumber
\end{align}

Equipped with the preceding bound we sum over the corresponding sets the probabilities examined in Proposition \ref{prop12.2}, it being pertinent to denote beforehand 
$$\Pi(k,s,\eta,\kappa)=\sum_{N=1}^{\infty}\sum_{j=1}^{N^{\kappa}}\sum_{\substack{\mathcal{N}_{L}\subset \mathcal{M}_{j}(N)\\ \lvert \mathcal{N}_{L}\rvert=L}}\mathbb{P}\big(\lvert R_{\frak{X},s}^{+}(\mathcal{N}_{L})-\mu_{\mathcal{N}_{L}}\rvert\geq \delta(N)\mu_{\mathcal{N}_{L}}\big).$$
We then apply the latter proposition for the choice $\beta= 7\kappa+2$ to derive
\begin{align*}
\Pi(k,s,\eta,\kappa)\ll \sum_{N=1}^{\infty}\sum_{j=1}^{N^{\kappa}}\sum_{\substack{\mathcal{N}_{L}\subset \mathcal{M}_{j}(N)\\ \lvert \mathcal{N}_{L}\rvert=L}}\big(N^{-\kappa_{0} C}+N^{- 7\kappa-2}\big)\ll \sum_{N=1}^{\infty}\sum_{j=1}^{N^{\kappa}}\binom{M}{L}N^{-\kappa_{0} C},
\end{align*}
it being appropiate to remark in view of (\ref{kapa}) and $\kappa_{0}\leq 1$ that
$\kappa_{0}C=6\kappa+\kappa_{0}\leq  7\kappa+2$. Consequently, inserting (\ref{ale}) in the above line permits one to deduce
$$\Pi(k,s,\eta,\kappa)\ll \sum_{N=1}^{\infty}\sum_{j=1}^{N^{\kappa}}N^{c(C+1)-\kappa_{0} C}\ll  \sum_{N=1}^{\infty} N^{\kappa+c(C+1)-\kappa_{0} C}. $$ We note upon recalling (\ref{kapa}) that the exponent in the right side of the above equation equals $-2\kappa,$ from where it follows that $\Pi(k,s,\eta,\kappa)<\infty.$ One thus may apply Borel-Cantelli to deduce with probability $1$ that for sufficiently large $N$ then every $\mathcal{N}_{L}\subset \mathcal{M}_{j}(N)$ satisfies
\begin{equation}\label{pere}R_{\frak{X},s}^{+}(\mathcal{N}_{L})=\mu_{\mathcal{N}_{L}}+O(\delta(N)\mu_{\mathcal{N}_{L}}).\end{equation}
Observe by Lemmata \ref{lem12.3} and \ref{lem12.4} that $R_{\frak{X},s}^{0}(n)+R_{\frak{X},s}^{=}(n)\ll 1$ for all $n \in [N,2N]\setminus \mathcal{Z}(N)$ with probability at least $0.7$. Recalling (\ref{ES1}) and inserting the above in (\ref{zz}), (\ref{zzz}) yields
$$ R_{\frak{X}}^{s}(\mathcal{N}_{L})=\sum_{i=1}^{L}R_{\frak{X},s}^{+}(n_{i})+O(L)=R_{\frak{X},s}^{+}(\mathcal{N}_{L})+O(L).$$
We conclude by combining the above line with (\ref{xist}), (\ref{pere}) and Proposition \ref{prop12.1}.
\end{proof}

\section{Upper bounds beyond the logarithmic barrier}\label{sec14}
We shall devote this section to derive upper bounds for the representation function via the application of concentration inequalities, it being required to such an end introducing some notation. We consider a finite set $\Gamma$ and the family $[\Gamma]^{\leq s}$ of subsets $I\subset \Gamma$ with the property that $\lvert I\rvert\leq s$. We then take $\mathcal{H}\subset [\Gamma]^{\leq s}$ and consider a family of non-negative random variables $Y_{I}$ for each $I\in [\Gamma]^{\leq s}$ such that $Y_{I}$ and $Y_{J}$ are independent whenever $I\cap J=\o.$ We further assume that there is another family $\xi_{\alpha}$, $\alpha\in\Gamma$ of independent random variables for which each $Y_{I}$ is a function of the collection $\{\xi_{\alpha}: \alpha\in I\}.$ We introduce $$X=\sum_{I\subset \mathcal{H}}Y_{I},\ \ \ \ \ \ \ X_{I}=\sum_{I\subset J}Y_{J} \ \ \ \ \  \text{for each $I\subset \mathcal{H}$}$$ and $\mu=\mathbb{E}(X).$ We write
$\mu_{I}= \sup \mathbb{E}(X_{I}| \xi_{\alpha}, \alpha\in I)$ and $\mu_{l}=\max_{\lvert I\rvert=l}\mu_{I}$ for each $1\leq l\leq s$.
\begin{prop}\label{prop10}
Let $\Gamma$ as above and $\lvert \Gamma\rvert=n$. Then, for every $t>0$ and every $r_{1},\ldots,r_{s}>0$ such that $r_{1}\ldots r_{l}\cdot\mu_{l}\leq t$ for each $1\leq l\leq s$, one has the estimate
$$\mathbb{P}(X\geq \mu+t)\leq \Big(1+\frac{t}{\mu}\Big)^{-r_{1}/8s}+\sum_{l=1}^{s-1}n^{l}\Big(1+\frac{t}{r_{1}\ldots r_{l}\mu_{l}}\Big)^{-r_{l+1}/8s}.$$
\end{prop}
\begin{proof}
See \cite[Theorem 3.10]{Jan2}.
\end{proof}
Equipped with the preceding utensil we are prepared to deduce the desired upper bound by combining the previous proposition with the analysis in the above sections.
\begin{cor}\label{cor12}
Let $s\geq 4k+1$ and let $\psi(n)=o(\log n)$ be a function of uniform growth. Then, whenever $n\in [N,2N]\setminus\mathcal{Z}(N)$ one gets with probability $1$ the bound
\begin{equation*}R_{\frak{X},s}^{+}(n)\ll \frac{\log n}{\log\big(\frac{\log n}{\psi(n)}\big)}.\end{equation*} Consequently, $R_{\frak{X}}^{s}(n)\ll \log n \log^{-1}\big(\frac{\log n}{\psi(n)}\big)$ holds with probability at least $0.6$. If moreover $\Delta_{s}^{*}<0$ is an admissible exponent for minor arcs, the above bounds hold for large enough $n$.
\end{cor}
\begin{proof}
In order to prepare the ground it seems desirable to recall (\ref{Rn}), introduce first $$\mathcal{R}_{n}=\Big\{B\subset\NN_{0}^{k}:\ \ \ \ \ B=\text{Set}(\bfx)\ \text{for some $\bfx\in\mathcal{R}^{+}(n)$}\Big\}$$ and use the above proposition by setting $X=R_{\frak{X},s}^{+}(n)$, the base set being $\Gamma=[1,n]$ and $$\mathcal{H}=\Big\{I\subset B_{n},\ \ \ \ B_{n}\in \mathcal{R}_{n} \Big\}.$$
We observe that with the notation presented above Proposition \ref{prop81} then $\mu_{I}=\mathbb{E}_{I}(R_{\frak{X},s}^{+}(n))$ for every $I\in\mathcal{H}$. We thereby deduce for $n\in [N,2N]\setminus \mathcal{Z}(N)$ and by means of Lemma \ref{lem9.4} that whenever $1\leq l\leq s-1$, then $\mu_{l}=O( n^{-2\beta_{k,s}})$ for some $\beta_{k,s}>0$. If moreover $\Delta_{s}^{*}<0$, the same holds for every sufficiently large $n$. We next set $$r_{1}(n)=\frac{32 s\log n}{\log\big(\frac{\log n}{\psi(n)}\big)},\ \ \ t_{n}=\frac{(32s)^{s}(s+1)!\beta_{k,s}^{-s+1}\log n}{2\log\big(\frac{\log n}{\psi(n)}\big)}, \ \ \ \ \  r_{l}=32s(l+1)\beta_{k,s}^{-1}\ \ \ \ \text{when $2\leq l\leq s$}.$$  It therefore transpires that
$r_{1}(n)\cdot r_{2}\ldots r_{s}=t_{n}$, and that for every $1\leq l\leq s-1$ then
$r_{1}(n)\cdot\ldots \cdot r_{l}\cdot \mu_{l}\ll n^{-\beta_{k,s}}\leq t_{n},$ the last inequality holding for sufficiently large $n$, and where we implicitly employed that $\psi(n)=o(\log n)$ and $\mu_{l}\ll  n^{-2\beta_{k,s}}$. Moreover, 
\begin{equation}\label{kkj}\frac{t_{n}}{r_{1}(n)\cdot r_{2}\ldots r_{l}\cdot \mu_{l}}\gg n^{\beta_{k,s}}.\end{equation}

We next write $\mu(n)=\mathbb{E}(R_{\frak{X},s}^{+}(n))$ and note in view of the fact that $\mu_{s}=1$ that the conditions described in the statement of Proposition \ref{prop10} hold, whence its application yields 
\begin{align}\label{sar}\mathbb{P}(R_{\frak{X},s}^{+}(n)\geq \mu(n)+t_{n})\leq \Big(1+\frac{t_{n}}{\mu(n)}\Big)^{-\frac{r_{1}(n)}{8s}}+\sum_{l=1}^{s-1}n^{l}\Big(1+\frac{t_{n}}{r_{1}(n)\cdot r_{l}\cdot \mu_{l}}\Big)^{-\frac{r_{l+1}}{8s}}.
\end{align}
We shall first examine the second summand and observe that (\ref{kkj}) delivers
\begin{equation}\label{sara}\sum_{l=1}^{s-1}n^{l}\Big(1+\frac{t_{n}}{r_{1}(n)\cdot r_{2}\ldots r_{l}\cdot \mu_{l}}\Big)^{-r_{l+1}/8s}\ll \sum_{l=1}^{s-1}n^{l-r_{l+1}\beta_{k,s}/8s}\ll \sum_{l=1}^{s-1}n^{-3l-8}\ll n^{-11}.\end{equation}
In order to examine the first one we recall Proposition \ref{prop9.1} both for the instance $\Delta_{s}^{*}<0$ and the situation in which $n\notin \mathcal{Z}(N)$ to note that there is some constant $C_{s,k}>0$ with $\mu(n)\leq C_{s,k}\psi(n).$ Therefore, upon denoting $c_{s,k}= \frac{1}{2}(32s)^{s}(s+1)!\beta_{k,s}^{-s+1} C_{s,k}^{-1}$ one has
$$ \Big(1+\frac{t_{n}}{\mu(n)}\Big)^{-r_{1}(n)/8s}\ll \Big(1+\frac{c_{s,k}\log n}{\psi(n)\log \big(\frac{\log n}{\psi(n)}\big)}\Big)^{-\frac{4\log n}{\log \big(\frac{\log n}{\psi(n)}\big)}}\ll e^{-4\log n+\frac{4(\log n)\log\big(c_{s,k}^{-1}\log\big(\frac{\log n}{\psi(n)}\big)\big)}{\log\big(\frac{\log n}{\psi(n)}\big)}}.$$
Then, in view of the condition $\psi(n)=o(\log n)$, it is apparent that 
$$\lim_{n\to\infty}\frac{\log\big(c_{s,k}^{-1}\log\big(\frac{\log n}{\psi(n)}\big)\big)}{\log\big(\frac{\log n}{\psi(n)}\big)}=0,$$ whence it follows that 
$\big(1+t_{n}/\mu(n)\big)^{-r_{1}(n)/8s}\ll n^{-3}.$ We combine the latter equation with (\ref{sar}), (\ref{sara}) to deduce $\mathbb{P}\big(R_{\frak{X},s}^{+}(n)\geq \mu(n)+t_{n}\big)=O(n^{-3}).$ Note that the relation $\mu(n)\leq C_{s,k}\psi(n)$ and the proviso $\psi(n)=o(\log n)$ entail $\mu(n)\leq  t_{n}$ for sufficiently large $n$. Therefore, $$\mathbb{P}\big(R_{\frak{X},s}^{+}(n)\geq 2t_{n}\big)\leq \mathbb{P}\big(R_{\frak{X},s}^{+}(n)\geq \mu(n)+t_{n}\big)\ll n^{-3},$$ a consequence of the above line being by a customary application of Borel-Cantelli that with probability $1$ one has for sufficiently large $N$ and $n\in [N,2N]\setminus\mathcal{Z}(N)$ the estimate
\begin{equation}\label{kkj2}R_{\frak{X},s}^{+}(n)< 2t_{n}\ll  \frac{\log n}{\log\big(\frac{\log n}{\psi(n)}\big)},\end{equation} the same bound holding for sufficiently large $n$ if $\Delta_{s}^{*}<0$. We conclude by noting that the bounds (\ref{RC}) and (\ref{RC1}) when $\Delta_{s}^{*}<0$ and Lemmata \ref{lem12.3}, \ref{lem12.4} if $n\notin \mathcal{Z}(N)$ yield 
$$ R_{\frak{X}}^{s}(n)=R_{\frak{X},s}^{+}(n)+O(1)$$ with probability at least $0.6$, which when combined with (\ref{kkj2}) completes the proof.
\end{proof}

\section{Proofs of Theorems \ref{thm1.3} and \ref{thm1.9}}\label{sec15}
Equipped with the preceding propositions we have reached a position from which to conclude our analysis concerning the aforementioned theorems.

\emph{Proof of Theorems \ref{thm1.3} and \ref{thm1.9}.} We write $\Xi_{s,k},\Xi_{2},\Xi_{1}>0$ to refer to the implicit constants latent in the error terms in (\ref{SSS}), (\ref{fax}) and (\ref{qqq}) respectively. Let $\Xi=\max(\Xi_{s,k}\cdot\Xi_{2},\Xi_{1})$ and consider for $N\in\mathbb{N}$ and $1\leq j\leq N^{\kappa}$ the collection of sets $\mathcal{M}_{j}(N)$ described right above (\ref{kapa}). It seems worth noting that in the context of Theorem \ref{thm1.9} one may assume that $\lvert \mathcal{M}_{j}(N)\rvert\leq M_{c}(N)$ since one might always write $\delta'(N)=\sqrt{c}\delta(N),$ the factor $\sqrt{c}$ being absorbed by the error term in (\ref{RsRs}). Whenever $\mathfrak{X}\subset \NN_{0}^{k}$ we thus introduce 
$$\mathcal{S}_{j,\mathfrak{X}}^{0}=\Big\{n\in\mathcal{M}_{j}(N):\ \ R_{\frak{X}}^{s}(n)\leq \frak{S}(n)\psi(n)-3\Xi\delta(N)\psi(N)\ \Big\}.$$ Theorem \ref{prop12.3} then entails the existence of a sequence $\mathfrak{X}\subset \mathbb{N}_{0}^{k}$ satisfying $\lvert \mathcal{S}_{j,\mathfrak{X}}^{0}\rvert\leq L-1$ for every $1\leq j\leq N^{\kappa}$ and sufficiently large $N$, since if otherwise $\lvert \mathcal{S}_{j,\mathfrak{X}}^{0}\rvert\geq L$ and $\{n_{1},\ldots,n_{L}\}\subset \mathcal{S}_{j,\mathfrak{X}}^{0}$ with $n_{i}\neq n_{l}$ for $i\neq l$ then by (\ref{delll}), (\ref{SSS}), (\ref{fax}) and (\ref{xist}) one would have
\begin{align*}\sum_{i=1}^{L}R_{\frak{X}}^{s}(n_{i})\leq\sum_{i=1}^{L}\frak{S}(n_{i})\psi(n_{i})- 3\Xi \delta(N)L\psi(N)&\leq\psi(N)\Big(\sum_{i=1}^{L}\frak{S}(n_{i})+\Xi L(\xi(N)-3\delta(N))\Big)
\\
&\leq\psi(N)\sum_{i=1}^{L}\frak{S}(n_{i})-2\psi(N)\Xi L\delta(N),\end{align*} which could only occur with probability at most $0.3$ by Theorem \ref{prop12.3}. A similar argument would apply whenever $1\leq j\leq N^{\kappa}$ for the sets
$$\mathcal{S}_{j,\mathfrak{X}}^{1}=\Big\{n\in\mathcal{M}_{j}(N):\ \ R_{\frak{X}}^{s}(n)\geq \frak{S}(n)\psi(n)+3\Xi\delta(N)\psi(N)\ \Big\},$$
and consequently, upon defining $\mathcal{S}_{j,\mathfrak{X}}=\mathcal{S}_{j,\mathfrak{X}}^{0}\cup\mathcal{S}_{j,\mathfrak{X}}^{1}$, one has $\lvert \mathcal{S}_{j,\mathfrak{X}}\rvert\leq 2L-2$. We observe that in the context of Theorem \ref{thm1.9} then (\ref{ssssp}) holds, whence (\ref{deac}) combined with Proposition \ref{prop7.1} and Corollary \ref{cor2} entails the bound $\lvert \mathcal{Z}(N)\rvert\ll 1.$ By the preceding discussion one has
\begin{equation*}R_{\frak{X}}^{s}(n)=\frak{S}(n)\psi(n)+O(\delta(N)\psi(N))\end{equation*} whenever $n\in\mathcal{M}_{j}(N)\setminus \mathcal{S}_{j,\mathfrak{X}}$ with $\mathcal{S}_{j,\mathfrak{X}}$ satisfying
$$\lvert\mathcal{S}_{j,\mathfrak{X}}\rvert\ll L=  \lvert \mathcal{M}_{j}(N)\rvert\frac{L}{ \lvert \mathcal{M}_{j}(N)\rvert}\ll\lvert\mathcal{M}_{j}(N)\rvert\omega(N)^{-1},$$ where $\omega$ is of uniform growth with $\omega(N)=O( e^{c\delta(N)^{2}\psi(N)})$, and where we employed both (\ref{ML}) and the fact that $\lvert \mathcal{M}_{j}(N)\rvert\gg \frac{(\log N)\omega(N)}{\delta(N)^{2}\psi(N)}$. This concludes the proof of Theorem \ref{thm1.9}.

In order to make progress in the proof of Theorem \ref{thm1.3} we introduce
$$\mathcal{S}_{N,\mathfrak{X}}^{\delta}=\Big\{n\in [N,2N]\setminus \mathcal{Z}(N):\ \lvert R_{\frak{X}}^{s}(n)-\frak{S}(n)\psi(n)\rvert\geq 3\Xi \delta(N)\psi(N) \Big\}.$$ 
We also partition $[N,2N]\setminus \mathcal{Z}(N) $ into $\mathcal{F}_{M}(N)$ sets $\mathcal{M}_{l}(N)$ with $\lvert \mathcal{M}_{l}(N)\rvert=M$ when $1\leq l\leq \mathcal{F}_{M}(N)-1$ and $\lvert \mathcal{M}_{\mathcal{F}_{M}}(N)\rvert\leq M$, and satisfying $\mathcal{M}_{l}(N)\cap \mathcal{M}_{j}(N)=\o$ whenever $l\neq j$. Noting that $ \mathcal{F}_{M}(N)\ll N/M$ and the fact that $(\mathcal{M}_{l}(N))_{l}$ is a partition one has
\begin{equation*}\lvert \mathcal{S}_{N,\mathfrak{X}}^{\delta}\lvert =\sum_{j\leq\mathcal{F}_{M}(N)}\lvert \mathcal{S}_{j,\mathfrak{X}}\rvert\leq 2(L-1)\mathcal{F}_{M}(N)\ll NL/M\ll \frac{Ne^{-c\delta(N)^{2}\psi(N)}}{\psi(N)\delta(N)^{2}},\end{equation*} where in the last steps we employed (\ref{ML}). The preceding discussion thereby yields
\begin{equation}\label{ash}R_{\frak{X}}^{s}(n)=\frak{S}(n)\psi(N)+O(\delta(N)\psi(N))\end{equation} whenever $n\in[N,2N]\setminus \big(\mathcal{Z}(N)\cup \mathcal{S}_{N,\mathfrak{X}}^{\delta}\big)$ with $\lvert\mathcal{S}_{N,\mathfrak{X}}^{\delta}\rvert$ satisfying the above bound. Observe that the constant $c$ in the previous estimate may be deleted upon writing $\delta'(N)=\sqrt{c}\delta(N)$.

In view of equation (\ref{ash}) and Corollary \ref{cor12} it remains to show that under the conditions concerning $k$ and $s$ described in Theorem \ref{thm1.3} then either the constraints in Propositions \ref{prop8.2}, \ref{prop911} and Lemma \ref{lem9.1111} are satisfied or else $\Delta_{s}^{*}<0$, the claims concerning the size $\lvert \mathfrak{X}_{k}\cap [1,X]\rvert$ following via Chernoff's inequality (see \cite[(10.2)]{Pli}). As a prelude to the discussion  we make recourse to \cite[Lemma 7.1]{Bru-Woo} and (\ref{tau}), assume that $k>20$ and note that one may take \begin{equation}\label{DDDD}D=9.027901.\end{equation} We allude to the analysis leading to (\ref{s00}) and observe that every $s_{0}\geq \max(\lfloor G_{0}(k)\rfloor +1,4k+1)$ has the property that $\Delta_{s_{0}}^{*}<0$, such a remark in conjunction with the choice for $D$, equation \cite[(7.8)]{Bru-Woo} and (\ref{eq67}) entailing that $s_{0}=\lfloor k(\log k+2+\log D)\rfloor +1$ satisfies the above. We then note first that if $s\geq s_{0}$ then $\Delta_{s}^{*}<0$, and hence $\lvert \mathcal{Z}(N)\rvert\ll 1$ by Proposition \ref{prop7.1} and Corollary \ref{cor2} combined with the conclusion leading to (\ref{deac}), as desired.

We next observe in view of the above definition for $s_{0}$ that it follows by the argument right after \cite[(7.8)]{Bru-Woo} that for any $s$ in the interval $k(\log k+3.20032)\leq s<s_{0}$ then
\begin{equation}\label{fl}k(\log k+1+\log D)\leq s\leq k(\log k+2+\log D)\end{equation} and hence $s_{0}-k-1\leq s<s_{0}$, as required in Proposition \ref{prop911}. We also note that $$2k(\log kD+1)\Big(1-\frac{1}{Dk}\Big)-2k-\frac{1}{2}> k(\log kD+2)+1$$ holds for $k\geq 100$ and $D$ as in (\ref{DDDD}), it thereby entailing when $d\leq s/Dk+k/4s$ that
$$s_{0}\leq k(\log kD+2)+1< 2k(\log kD+1)\Big(1-\frac{1}{Dk}\Big)-2k-\frac{1}{2}\leq 2(s-d-k),$$ as desired. We continue the verification of the constraints in Proposition \ref{prop911} noting when $k\geq 100$ that 
$\log kD+2\leq D(k-5)/4,$ it in turn entailing for $s$ in the range (\ref{fl}) that
\begin{equation}\label{sabes}\frac{s}{Dk}+\frac{k}{4s}\leq \frac{k}{4}-1.\end{equation}
We next recall (\ref{sto}) and remark that $s_{T_{0}}\geq s-d-T_{0}-1.$ Thus by \cite[Theorem 2.1]{Woo3} and when $s$ satisfies (\ref{fl}) and $T_{0}\leq T_{d}(k)\leq 3k/4$ there is some admissible exponent $\Delta_{s_{T_{0}}}$ with
\begin{equation}\label{sub}0\leq \Delta_{s_{T_{0}}}\leq ke^{1-s_{T_{0}}/k}\leq \frac{e^{(T_{0}+d+1)/k}}{D}\leq \frac{1}{D}\Big(1+\frac{T_{0}+d+1}{k}+\Big(\frac{T_{0}+d+1}{k}\Big)^{2}\Big),\end{equation} 
where we employed $d+1\leq k/4$, the last inequality holding in view of (\ref{sabes}). 

We may assume at this point first that $s< Dk$, and observe that whenever $s\geq k(\log (kD)+1)+2$ then by \cite[Theorem 2.1]{Woo3} it transpires that
\begin{equation}\label{jjj}\Delta_{2\lfloor \frac{s-1}{2}\rfloor}-\frac{k}{s}\leq  \frac{1}{D}-\frac{k}{s}<0.\end{equation}
Similarly, one has for $k(\log (kD)+1)\leq s< k(\log (kD)+1)+2$ and $s<Dke^{-2/k}$ that
\begin{equation}\label{jjjj}\Delta_{2\lfloor \frac{s-1}{2}\rfloor}-\frac{k}{s}\leq  \frac{e^{2/k}}{D}-\frac{k}{s}<0.\end{equation} By recalling (\ref{DDDD}) we note that $s$ as above exists only when $k(\log (kD)+1)<Dke^{-2/k},$ such an inequality being valid for $1\leq k\leq 300$. We conclude the cornucopia of instances for which an inequality of the flavour of (\ref{jjj}) holds by considering $s$ in the range (\ref{fl}) and satisfying $1\leq s/Dk<2.$ Under such circumstances, whenever $s\geq k\log(2ekD)+2$ then by \cite[Theorem 2.1]{Woo3} one has
$\Delta_{2\lfloor \frac{s-1}{2}\rfloor}\leq 1/2D$, and hence $\Delta_{2\lfloor \frac{s-1}{2}\rfloor}<k/s.$

Under the above assumptions it flows from (\ref{jjj}), (\ref{jjjj}), the last inequality and Lemma \ref{lem9.1111} combined with (\ref{afri}) that $\mathcal{Z}(N)=\mathcal{Z}_{0}(N)$. In view of (\ref{fl}), (\ref{s00}) and the discussion right after (\ref{iiu}) one then has $\Delta_{2s}^{*}<0,$ Proposition \ref{prop8.2} and Corollary \ref{cor4} in conjunction with the preceding observations entailing
\begin{equation}\label{vali2}\lvert \mathcal{Z}(N)\rvert=\lvert \mathcal{Z}_{0}(N)\rvert\ll N^{1-\zeta_{k}}.\end{equation}

We consider next $s$ as in (\ref{fl}) with $1\leq s/Dk+k/4s< 2$ and satisfying $s\leq k\log (2ekD)+2$, the combination of both inequalities entailing that $k\geq 200$. We shall next verify the conditions required in Proposition \ref{prop911}, it being worth noting first that in this range
\begin{equation}\label{vax}\frac{1}{k}+\frac{2s}{k^{2}}\leq \frac{1}{200}+\frac{2\log (2ekD)}{k}+\frac{1}{10000}\leq \frac{5}{1000}+\frac{93}{1000}= \frac{1}{10}\end{equation} and $\max(15/2,\sqrt{D}/2)k\leq s\leq Dk^{2}.$ We also indicate upon recalling (\ref{DDDD}) that 
$$1600\big(\log (2kD)+1+2/k\big)^{2}\big(\log(2kD)+1+D/2+2/k\big)\leq 1053Dk$$ holds whenever $k\geq 200$. Equipped with the above provisos it then follows that
$$ 1600s^{2}(s+Dk/2)\leq 1053Dk^{4}\leq 1170k^{4}D\Big(1-\frac{1}{k}-\frac{2s}{k^{2}}\Big)$$ when $k\geq 200$, such an inequality implying that $T_{1}(k)\leq 3k/4$ in the above range for $s$. Likewise, we remind the reader of (\ref{Td}) and note by (\ref{vax}) that whenever $d=1$ then \begin{equation*}4c_{k}\geq \frac{39\cdot 9}{200}\frac{D}{1+\frac{Dk}{2s}}\geq \frac{1053D}{1000}>D,\end{equation*} wherein we employed the assumption $3Dk\leq 4s$, as desired. In order to conclude the analysis of this instance we note that when $k\geq 200$ and $k(\log (kD)+1)\leq s< k(\log (ekD))+2$ with $Dke^{-2/k}\leq s\leq Dk$, then in particular $s\leq ke^{2/k}/4(e^{2/k}-1)$, and hence
$$1\leq e^{-2/k}+\frac{k}{4s}\leq \frac{s}{Dk}+\frac{k}{4s}\leq 5/4<2.$$

We assume next that $2\leq s/Dk+k/4s$ with $s$ in (\ref{fl}), such a proviso in particular entailing $2D(\log(kD)+7/8)\leq (\log(kD)+2)^{2},$ whence $k\geq 100000.$ Under such circumstances then the bound on $s$ right after (\ref{vax}) holds, and hence for $d\leq s/Dk+k/4s$ one has
\begin{align}\label{restr}\frac{d}{k}+\frac{(d+1)s}{k^{2}}&\leq \frac{s}{Dk^{2}}+\frac{1}{4s}+\frac{s}{k^{2}}+\frac{s^{2}}{Dk^{3}}+\frac{5}{4k}
\\
&\leq \frac{2}{k}+\frac{(\log (kD)+2)^{2}}{Dk}+\frac{(D+1)(\log(kD)+2)}{Dk}\leq \frac{1}{500}.\nonumber
\end{align}
Likewise, under the same assumptions on $k$ it is apparent that
$$(\log(kD)+2)^{2}\Big(\frac{3}{2}(\log(kD)+2\big)+D/8\Big)\leq \frac{58383Dk}{80000},$$ the combination of the preceding estimates and the restriction on $d$ delivering
\begin{align*} s^{2}(s+dDk/2)&\leq k^{3}(\log(kD)+2)^{2}\Big(\frac{3}{2}\big(\log(kD)+2\big)+\frac{D}{8}\Big)\leq\frac{117}{160}Dk^{4}\Big(1-\frac{d}{k}-\frac{(d+1)s}{k^{2}}\Big),
\end{align*} it in turn entailing $T_{d}(k)\leq 3k/4.$ We may further derive in an analogous manner as above 
$$4c_{k}\geq \frac{39\cdot 499}{10000}\frac{D}{1+\frac{dDk}{2s}}\geq \frac{39\cdot 499}{10000}\cdot\frac{D}{\frac{3}{2}+\frac{Dk^{2}}{8s^{2}}}\geq \frac{39\cdot 499D}{100\cdot 151}>D,$$ where we implicitly employed the fact that $Dk^{2}/8s^{2}\leq 1/100$. By the preceding discussion it transpires that the assumptions in Proposition \ref{prop911} hold for both instances.

We shall next verify the conditions in Lemma \ref{lem9.1111} and assume first that $1\leq \frac{s}{Dk}+\frac{k}{4s}< 2$ and $s\geq k(\log kD+1)+3$. We allude as is customary to \cite[Theorem 2.1]{Woo3} to deduce that
$$\delta_{2}-\frac{2k}{s}\leq  \frac{1}{D}-\frac{2k}{s}<\frac{1}{D}-\frac{1}{D(1-k/8s)}<0.$$
If instead $k(\log kD+1)\leq s< k(\log kD+1)+3$ then combining both constraints yields
$$D\Big(1-\frac{1}{4(\log kD+1)}\Big)\leq \frac{s}{k}\leq \log(kD)+1+\frac{3}{k},$$ from where it would follow that $k\geq 250.$ Moreover, the above theorem would entail 
$$\delta_{2}-\frac{2k}{s}\leq  \frac{e^{3/k}}{D}-\frac{2k}{s}<\frac{243}{80Dk}-\frac{k}{8Ds(1-k/8s)}\leq \frac{243}{80Dk}-\frac{2k}{(16D-1)s}$$ whenever $k\geq 250$, wherein we employed the fact that $s\leq 2Dk.$ One then has 
$$\frac{k^{2}}{s}\geq \frac{k}{\log(kD)+1+3/k}\geq \frac{250}{9}\geq 27,$$ the combination of the preceding bounds and the notation in Lemma \ref{lem9.1111} delivering
$$\delta_{2}-\frac{2k}{s}\leq \frac{1}{k}\Big(\frac{243}{80D}-\frac{54}{(16D-1)}\Big)\leq -\frac{27}{80Dk}.$$ 
If on the contrary $2<s/Dk+k/4s$ then by the conclusion stemming from (\ref{restr}) it transpires that $k\geq 100000$. We then set $d_{0}=\big\lceil\frac{s}{Dk}+\frac{k}{4s}\big\rceil$ and note that whenever $s$ is as in (\ref{fl}) and $k$ satisfies the above lower bound then 
$$d_{0}+1\leq \frac{s}{Dk}+\frac{k}{4s}+2\leq \frac{\log (kD)+2}{D}+\frac{9}{4}\leq \frac{k}{4},$$ as required right above (\ref{sub}). Therefore, \cite[Theorem 2.1]{Woo3} permits one to deduce that
\begin{align*}\delta_{d_{0}}-\frac{d_{0}k}{s}&\leq \frac{e^{(d_{0}+1)/k}}{D}-\frac{1}{D}-\frac{k^{2}}{4s^{2}}\leq \frac{1}{D}\Big(\frac{d_{0}+1}{k}+\Big(\frac{d_{0}+1}{k}\Big)^{2}\Big)-\frac{k^{2}}{4s^{2}}
\\
&\leq \frac{5(d_{0}+1)}{4Dk}-\frac{k^{2}}{4s^{2}}\leq \frac{5s}{4D^{2}k^{2}}+\frac{4}{Dk}-\frac{k^{2}}{4s^{2}}\leq \frac{8s}{Dk^{2}}-\frac{k^{2}}{4s^{2}}.
\end{align*}
Consequently, the inequality
$64(\log (kD)+2)^{3}\leq Dk,$ valid whenever $k\geq 100000$, in conjunction with the above bound yields
$s\delta_{d_{0}}<d_{0}k,$ as desired. The preceding discussion enables one to deduce that the required assumptions in Lemma \ref{lem9.1111} hold for the above choice of $d_{0}$. In view of the inequality $\Delta_{2s}^{*}<0$ right above (\ref{vali2}), Proposition \ref{prop8.2}, Corollary \ref{cor4} and the above conclusions one may apply Corollary \ref{cor9.1} to deduce
$\lvert \mathcal{Z}(N)\rvert\ll N^{1-\zeta_{k}},$ which completes when combined with Corollary \ref{cor12} the case $k> 20$ in Theorem \ref{thm1.3}.

If $14\leq k\leq 20$, the tables in \cite{Vau-Woo2} assure that the exponents $\Delta_{w-1}$ are admissible, $w$ being defined below.
\begin{table}[h!]
  \begin{center}
    \label{tab:table1}
    \begin{tabular}{l l l l l l c c c c c c r r r r r } 
      \hline
      $k$  & $14$ & $15$ & $16$ & $17$ & $18$ & $19$ & $20$\\
      $w$  &  $75$ & $81$ & $87$ & $95$ & $101$ & $109$ & $117$\\
      $\Delta_{w-1}$  & $0.1281620$ & $0.1355287$ & $0.1426626$ & $0.1318848$ & $0.1390360$ & $0.1306147$ & $0.1238487$\\
\hline
    \end{tabular}
  \end{center}
\end{table}
Consequently, whenever $w\leq s\leq k(\log k+4.20032)$ then 
$ s\Delta_{w-1}<k,$ from where it follows that $ s\delta_{1}<k$, the restriction in Lemma \ref{lem9.1111} holding for $d_{0}=1$. We further note by the tables in \cite{Vau-Woo2} that $\Delta_{2s-2}=0$, and hence $\Delta_{2s}^{*}<0$.  Therefore, the application of (\ref{afri}) combined with Proposition \ref{prop8.2} and Corollary \ref{cor4} entails $\lvert \mathcal{Z}(N)\rvert\ll N^{1-\zeta_{k}}$, as desired.


\begin{thebibliography}{99}
\bibitem{Bru-Woo} J. Bruedern, T. D. Wooley, \emph{On Waring's problem for larger powers}, J. Reine Angew. Math. 805 (2023), 115--142.
\bibitem{Bru-Woo2} J. Bruedern, T. D. Wooley, \emph{Estimates for smooth Weyl sums on major arcs}, arXiv:2405.18608.
\bibitem{Din} P. Ding and A. R. Freedman, \emph{Small sets of kth powers}, Canad. Math. Bull. 37 (1994),
168--173.
\bibitem{ErD} P. Erd\"os, \emph{Problems and results on additive properties of general sequences. II}, Acta Math. Hung.
48 (1--2) (1986), 201--211.
\bibitem{ERD} P. Erdős, \emph{Some of my favourite problems in number theory, combinatorics, and geometry}. Combinatorics Week (Portuguese) (São Paulo, 1994). Resenhas 2 (1995), no. 2, 165--186. 

\bibitem{Erdos} P. Erd\H{o}s, A. Renyi, \emph{Additive properties of random sequences of positive integers}, Acta Arith. 6 (1960).
\bibitem{Erd3} P. Erd\H{o}s, R. Rado, \emph{Intersection theorems for systems of sets}, J. Lond. Math. Soc. (2) 35 (1960), 85--90.

\bibitem{ErTe} P. Erd\H{o}s, P. Tetali, \emph{Representations of integers as the sum of $k$ terms}, Random Structures Algorithms 1 (1990), No. 3, 245--261.

\bibitem{Erd} P. Erd\H{o}s, P. Tur\'an, \emph{On a problem of Sidon in additive number theory, and on some related problems}, J. Lond. Math. Soc. 16 (1941), 212--215.

\bibitem{Gog} J. H. Goguel, \emph{Über Summen von zufälligen Folgen natürlischer Zahlen}, J. Reine
Angew. Math. 278/279 (1975), 63--77.
\bibitem{Hal} H. Halberstam, K. F. Roth, \emph{Sequences}, Springer-Verlag, New York, 1983.

\bibitem{Jan2} S. Janson, A. Rucinski, \emph{The deletion method for upper tail estimates}, Combinatorica 24 (2004), 615--640.
\bibitem{Kim} J. H. Kim, V. H. Vu, \emph{Concentration of multivariate polynomials and its applications}, Combinatorica 20 (2000), 417--434.
\bibitem{Lan} B. Landreau, \emph{Étude probabiliste des sommes des puissances s-ièmes}, Compositio Math. 99 (1995), 1--31.
\bibitem{Nat} M. B. Nathanson, \emph{Waring's problem for sets of density zero. Analytic number theory (Philadelphia, Pa.,
1980)}, 301--310, Lecture Notes in Math., 899, Springer, Berlin-New York, 1981.
\bibitem{Pli} J. Pliego, \emph{On the Erd\"os-Tur\`an Conjecture and the growth of $B_{2}[g]$ sequences}, arxiv: 	arXiv:2405.04154.


\bibitem{Taf} C. Tafula, \emph{Representation functions with prescribed rates of growth}, arXiv:2405.01530v1.
\bibitem{Vau} R. C. Vaughan, \emph{The Hardy-Littlewood method}, 2nd edition, Cambridge University Press, 1997.
 \bibitem{Vau2} R. C. Vaughan, \emph{A new iterative method in Waring's problem}, Acta Math. 162 (1989), no. 1-2, 1--71.
\bibitem{Vau-Woo2} R. C. Vaughan, T. D. Wooley, \emph{Further improvements in Waring's Problem, IV: higher powers}, Acta Arith. 94 (2000), no. 3, 203-285.


\bibitem{Vu2} V. H. Vu, \emph{On the concentration of multivariate polynomials with small expectation}, Random Structures and Algorithms 16 (2000), 344--363.
\bibitem{Vu} V. H. Vu, \emph{On a refinement of Waring's problem}, Duke Math. J. 105 (2000), 107--134.
\bibitem{Woo} T. D. Wooley, \emph{Large improvements in Waring's problem}, Ann. of Math. (2) 135 (1992), no. 1, 131--164.
\bibitem{Woo3} T. D. Wooley, \emph{The application of a new mean value theorem to the fractional parts of polynomials}, Acta Arith., LXV. 2 (1993), 163--179.
\bibitem{Woo2} T. D. Wooley, \emph{New estimates for smooth Weyl sums}, J. London Math. Soc. 51 (1995), 1--13.
\bibitem{WooVu} T. D. Wooley, \emph{On Vu's thin basis theorem in Waring's problem}, Duke Math. J. 120 (2003), no. 1, 1--34.
\bibitem{Woo6} T. D. Wooley, \emph{Vinogradov's mean value theorem via efficient congruencing}, Ann. of Math. (2) 175 (2012), no. 3, 1575--1627.
\bibitem{Woo33} T. D. Wooley, \emph{Rational solutions of pairs of diagonal equations, one cubic and one quadratic}, Proc. London Math. Soc. (3) 110 (2015), no. 2, 325--356.

\end{thebibliography}
\end{document}